\documentclass[a4paper]{siamltex1213}
\newif\ifproofended

\let\siamendproof=\endproof

\def\qedhere{\qquad\siamendproof\gdef\doendproof{\empty}}

\makeatletter
\renewenvironment{proof}[1][\proofname]{\par
    \normalfont
    \topsep6\p@\@plus6\p@ \trivlist
    \item[\hskip\labelsep\itshape
    #1\@addpunct{.}]\ignorespaces
    \gdef\doendproof{\siamendproof}
}{%
    \doendproof\endtrivlist
}
\makeatother

\newtheorem{remarkX}{\itshape Remark}
\newtheorem{exampleX}{\itshape Example}
\newenvironment{remark}[1][]{%
        \begin{remarkX}%
        \def\foo{#1}\ifx\foo\empty\else(#1)\fi\normalfont
    }{
        \end{remarkX}
    }
\newenvironment{example}[1][]{\begin{exampleX}\normalfont}{\end{exampleX}}

\makeatletter
\@addtoreset{assumption}{section}
\@addtoreset{remarkX}{section}
\@addtoreset{exampleX}{section}
\makeatother

\renewcommand{\grad}{\nabla}

\usepackage[utf8]{inputenc}
\usepackage[T1]{fontenc}
\usepackage{lmodern}
\usepackage[english]{babel}
\usepackage{comment}
\usepackage{bbm}
\usepackage{amsmath}
\usepackage{amssymb}
\usepackage{ifpdf}
\usepackage{color}
\definecolor{hrefcolor}{rgb}{0.0,0.5,0.8}
\definecolor{hlgreen}{rgb}{0,0.7,0}

\hypersetup{
   colorlinks=true,
   linkcolor=hrefcolor,
   citecolor=hlgreen,
   filecolor=hrefcolor,
   urlcolor=hrefcolor,
}

\newenvironment{enumroman}
    {\renewcommand{\theenumi}{(\roman{enumi})}
     
     \begin{enumerate}
        \setlength{\leftmargin}{3.0em}
        \setlength{\labelwidth}{2.5em}
        \setlength{\labelsep}{0.5em}
    }
    {\end{enumerate}}

\newcounter{remcount}

\newcommand{\term}{\emph}

\newcommand{\field}[1]{\mathbb{#1}}

\newcommand{\Z}{\mathbb{Z}}

\newcommand{\R}{\field{R}}
\newcommand{\B}{B}

\newcommand{\norm}[1]{\|#1\|}

\newcommand{\abs}[1]{|#1|}

\newcommand{\inv}[1]{#1^{-1}}
\newcommand{\freevar}{\,\boldsymbol\cdot\,}

\newcommand{\Union}\bigcup
\newcommand{\Isect}\bigcap
\newcommand{\union}\cup
\newcommand{\isect}\cap
\newcommand{\bigunion}\bigcup
\newcommand{\bigisect}\bigcap

\newcommand{\defeq}{:=}

\newcommand{\downto}{\searrow}
\newcommand{\upto}{\nearrow}
\newcommand{\subdiff}{\partial}

\DeclareMathOperator*{\lip}{lip}

\DeclareMathOperator{\ri}{ri}

\DeclareMathOperator{\closure}{cl}

\DeclareMathOperator{\sign}{sgn}

\makeatletter
\def \uminus@sym{\setbox0=\hbox{$\cup$}\rlap{\hbox 
        to\wd0{\hss\raise0.5ex\hbox{$\scriptscriptstyle{-}$}\hss}}\box0}
    \def \uminus    {\mathrel{\uminus@sym}}
\makeatother

\newcommand{\mathvar}[1]{\textup{#1}}

\def\Xint#1{\mathchoice
{\XXint\displaystyle\textstyle{#1}}%
{\XXint\textstyle\scriptstyle{#1}}%
{\XXint\scriptstyle\scriptscriptstyle{#1}}%
{\XXint\scriptscriptstyle\scriptscriptstyle{#1}}%
\!\int}
\def\XXint#1#2#3{{\setbox0=\hbox{$#1{#2#3}{\int}$ }
\vcenter{\hbox{$#2#3$ }}\kern-.6\wd0}}

\def\dashint{\Xint-}

\renewcommand{\tilde}{\widetilde}
\renewcommand{\ae}{a.e.~}
\newcommand{\iprod}[2]{\langle #1,#2\rangle}
\newcommand{\adaptiprod}[2]{\left\langle #1,#2\right\rangle}

\newcommand{\BVspace}{\mathvar{BV}}

\newcommand{\Meas}{\mathcal{M}}

\renewcommand{\H}{\mathcal{H}}
\renewcommand{\L}{\mathcal{L}}
\newcommand{\restrict}{\llcorner}
\newcommand{\BD}{\partial}
\renewcommand{\d}{\,d} %\text{d}}

\newcommand{\TGV}{\mathvar{TGV}}

\newcommand{\TV}{\mathvar{TV}}

\DeclareMathOperator*{\esssup}{ess\,sup}

\DeclareMathOperator{\Per}{Per}

\DeclareMathOperator{\divergence}{div}

\makeatletter
\def \weaktostar@sym{\setbox0=\hbox{$\rightharpoonup$}\rlap{\hbox 
        to\wd0{\hss\raise1ex\hbox{$\scriptscriptstyle{*\,}$}\hss}}\box0}
    \def \weaktostar    {\mathrel{\weaktostar@sym}}
\makeatother

\usepackage{soul}

\newcommand{\TODO}[1]{{\sethlcolor{magenta}\color{white}\bfseries\em\hl{#1}}}

\usepackage{color}

\newcommand{\Uspace}{L^1(\Omega)}
\newcommand{\BVUspace}{\BVspace(\Omega)}
\newcommand{\EUspace}{\Meas(\Omega; \R^m)}

\newcommand{\jacobian}{\mathcal{J}}
\newcommand{\jacobianf}[3]{\mathcal{J}_{#1} #2(#3)}

\DeclareMathOperator{\DIFF}{d}
\newcommand{\DIFFSS}{\grad}

\renewcommand{\norm}[1]{\|#1\|}

\newcommand{\Jfunct}{\mathcal J}

\renewcommand{\hat}[1]{\widehat{#1}}
\renewcommand{\tilde}[1]{\widetilde{#1}}

\newcommand{\Dfull}{D}
\newcommand{\Dabs}{\grad}
\newcommand{\Djump}{D^j}
\newcommand{\Dcantor}{D^c}
\newcommand{\ICTV}{\mathvar{ICTV}}

\newcommand{\allLambda}{\isect}

\newcommand{\huber}[2][\gamma]{\abs{#2}_{#1}}
\def\hubergamma{\eta}
\newcommand{\pf}[1]{#1{}_{\#}}
\def\Xint#1{\mathchoice
{\XXint\displaystyle\textstyle{#1}}%
{\XXint\textstyle\scriptstyle{#1}}%
{\XXint\scriptstyle\scriptscriptstyle{#1}}%
{\XXint\scriptscriptstyle\scriptscriptstyle{#1}}%
\!\int}
\def\XXint#1#2#3{{\setbox0=\hbox{$#1{#2#3}{\int}$ }
\vcenter{\hbox{$#2#3$ }}\kern-.6\wd0}}

\def\dashint{\Xint-}
\newcommand{\bitransfull}[2]{T_{#1,#2}}
\newcommand{\bitrans}[2]{G_{#1,#2}}
\newcommand{\bitransjac}[2]{\bar J_{#1,#2}}
\newcommand{\lipjac}[1]{A_{#1}}
\newcommand{\didtrans}[1]{D_{#1}}

\newcommand{\fst}[1]{\overline#1}
\newcommand{\snd}[1]{\underline#1}
\newcommand{\gammaone}{{\fst\gamma}}
\newcommand{\gammatwo}{{\snd\gamma}}
\newcommand{\uone}{{\overline u}}
\newcommand{\utwo}{{\underline u}}
\newcommand{\LipClass}{\mathcal{F}}
\newcommand{\gammax}[3]{\gamma_{#1,#3}}
\newcommand{\loc}{\mathrm{loc}}
\newcommand{\linear}{\mathcal{L}}
\newcommand{\RCa}{R^a}
\newcommand{\RCs}{R^s}
\newcommand{\basehspace}{W_0^{1,\infty}(z_\Gamma^\perp \isect \B(0, 1))}
\newcommand{\baseh}{\bar h}

\newcommand{\TDIFF}{\mathcal{D}}
\newcommand{\CURVP}{\mathcal{C}}
\newcommand{\curvf}{c}

\def\pgp{P_{z_\Gamma}^\perp}

\begin{document}

\title{The jump set under geometric regularisation.\\
    Part 1: Basic technique and first-order denoising
    }
\slugger{sima}{xxxx}{xx}{x}{x--x}%slugger should be set to mms, siap, sicomp, sicon, sidma, sima, simax, sinum, siopt, sisc, or sirev

\pagestyle{myheadings}
\thispagestyle{plain}
\markboth{T. VALKONEN}{THE JUMP SET UNDER GEOMETRIC REGULARISATION, P1}

\author{Tuomo Valkonen\thanks{
    Department of Applied Mathematics and Theoretical Physics,
    University of Cambridge, UK.
    The most significant part of this work was done while the author was a Prometeo Fellow at the Research Center on Mathematical Modeling (Modemat), Escuela Politécnica Nacional de Quito, Ecuador.
    E-mail: \texttt{tuomo.valkonen@iki.fi}.
    }
    }

\maketitle

\begin{abstract}
Let $u \in \BVspace(\Omega)$ solve the total variation denoising 
problem with $L^2$-squared fidelity and data $f$. Caselles et al. 
[Multiscale Model.~Simul.~6 (2008), 879--894] 
have shown the containment $\H^{m-1}(J_u \setminus J_f)=0$
of the jump set $J_u$ of $u$ in that of $f$. Their proof unfortunately
depends heavily  on the co-area formula, as do many results in this area,
and as such is not directly extensible to higher-order, curvature-based,
and other advanced geometric regularisers, such as total generalised 
variation ($\TGV$) and Euler's elastica. These have received 
increased attention in recent times due to their better practical 
regularisation properties compared to conventional total variation 
or wavelets.
We prove analogous jump set containment properties for a general 
class of regularisers. We do this with novel Lipschitz 
transformation techniques, and do not require the co-area formula.
In the present Part 1 we demonstrate the general technique on
first-order regularisers, while in Part 2 we will extend it to 
higher-order regularisers. In particular, we concentrate in this
part on $\TV$ and, as a novelty, Huber-regularised $\TV$. 
We also demonstrate that the technique would apply to non-convex
$\TV$ models as well as the Perona-Malik anisotropic diffusion,
if these approaches were well-posed to begin with.
\end{abstract}

    

%%%
\section{Introduction}
%%%

We study the structure of the approximate 
jump set $J_u$ of solutions $u \in \BVUspace$ to regularisation 
problems
\begin{equation}
    \label{eq:prob}
    \tag{P}
    \min_{u \in \Uspace} \int_\Omega \phi(f(x)-u(x)) \d x + R(u).
\end{equation}
We recall that $J_u$ is the $\H^{m-1}$-rectifiable set on which
$u$ has two distinct one-sided Lebesgue limits.
We consider domains $\Omega \subset \R^m$ for $m \ge 2$,
and assume that $\phi: [0, \infty) \to [0, \infty]$ is a convex,
lower semi-continuous, $p$-increasing fidelity function, and $R$ 
a regularisation functional, which generalises 
total variation (TV) in a suitable sense. 
The technical conditions that we set on $R$ are to ensure that
solutions satisfy $u \in \BVUspace$ and that $R$ behaves almost
like $\TV$ under small Lipschitz transformations. We state
these conditions in detail in \S\ref{sec:problem}. 
Briefly, we require the $\BVspace$-coercivity condition
\[
    \norm{Du}_{2,\EUspace} \le C\bigl(1+\norm{u}_{\Uspace}+R(u)\bigr),
    \quad (u \in \Uspace),
\]
and a \term{double-Lipschitz comparability condition} of the form
\begin{equation}
    \label{eq:double-lip}
    R(\pf\gammaone u) + R(\pf\gammatwo u) - 2 R(u) \le 
    C \bitransfull{\gammaone}{\gammatwo} \abs{D u}(\closure U).
\end{equation}
Here $\gammaone$ and $\gammatwo$ are Lipschitz transformations
on an open set $U$. The suitably defined distance
$\bitransfull{\gammaone}{\gammatwo}$ between the transformations turns out
to be $O(\rho^2)$ for specially constructed shift transformations,
dependent on a parameter $\rho$.
As we will see in \S\ref{sec:reg}, this class of regularisation 
functionals includes, in particular,  total variation ($\TV$) and 
Huber-regularised total variation. In Part 2 \cite{tuomov-jumpset2} of this 
pair of papers, we will look at the extension of this condition and technique
to cover higher-order regularisers, in particular total generalised 
variation (TGV) \cite{bredies2009tgv}, and infimal convolution TV (ICTV) 
\cite{chambolle97image}.
The analysis of these functionals, in particular that of $\TGV$, 
is significantly more involved than that of first-order regularisers, 
and enough to fill an additional manuscript or three. 

Assuming further that both the solution $u$ and the given data $f$
are in $\BVUspace \isect L^\infty_\loc(\Omega)$, 
we show for $p$-increasing $\phi$ for $1 < p < \infty$, 
including $L^p$ fidelities $\phi(t)=\abs{t}^p$, that the jump set $J_u$ of $u$ 
is contained, modulo a $\H^{m-1}$ null set, in the jump set $J_f$ of $f$.
That is, $\H^{m-1}(J_u \setminus J_f)=0$. The boundedness condition
of course holds for $\TV$ and Huber-$\TV$ by standard barrier arguments,
but has to be shown or imposed separately in the general case.
%\TODO{We will discuss this in a follow-up paper?}
For $\phi(x)=x$, i.e., the $L^1$ fidelity, the same conclusion does not
necessarily hold, as is known in the case of total variation 
regularisation \cite{duval2009tvl1}. Under assumptions of 
approximate piecewise constancy, we however show 
that $J_u \setminus J_f$ has a $C^{2,\gamma}$ structure with curvature
$1/\alpha$, for $\alpha$ the (asymptotic) regularisation parameter.
We state all of these results in detail in \S\ref{sec:problem}.
The proofs are split between multiple sections. 
%We provide various
%miscellaneous technical lemmas in \S\ref{sec:tech},
We construct the specific Lipschitz transformations in
\S\ref{sec:lipschitz}. The main part of the general
proof, studying the effect of the transformations on the
fidelity, can be found in \S\ref{sec:main} for $p>1$, and
in \S\ref{sec:main-l1} for $p=1$.

%For second-order total generalised variation, which may for parameters
%$(\beta,\alpha) > 0$ according to \cite{l1tgv,sampta2011tgv} be written 
%as the ``differentiation cascade''
%\[
%	\TGV^2_{(\beta,\alpha)}(u)
%       \defeq
%	\min_{w \in \Wspace}
%	\alpha \norm{Du - w}_{F,\EUspace}
%       +
%	\beta \norm{Eu}_{F,\EWspace},
%\]
%we additionally show that the mass $\norm{D^s u}_{\EUspace}$
%tends to zero as the second-order regularisation parameter
%$\beta \downto 0$. In particular 
%then $\norm{D u \restrict J_u}_{\EUspace}$ tends to zero.

The class of problems \eqref{eq:prob} is of importance, in particular, for 
image denoising. From an application point of view, it is desirable to know the
structure of $J_u$ in order to show that the regularisation method is
reliable -- that it does not introduce undesirable artefacts, new edges, and 
correctly restores edges where they are present in the original data $f$. 
Higher-order geometric regularisation functionals, such a total generalised 
variation (TGV) \cite{bredies2009tgv} and infimal convolution TV
\cite{chambolle97image}, are as a matter of fact motivated by 
other artefacts introduced by TV regularisation: the stair-casing effect. 
Total variation exhibits flat areas with sharp transitions, which higher-order 
regularisation tends to avoid. We also note that while 
the $L^2$ fidelity $\phi(t)=t^2$, is often easier from the computational point 
of view, the $L^1$ fidelity can better deal with outliers. The $L^2$
fidelity models Gaussian noise, while the $L^1$ fidelity closer models
impulse noise, and should therefore be preferred from the point of view of
robust statistics. The study of the structure of solutions to the general class
of problems \eqref{eq:prob}, with varying fidelity $\phi$ and regulariser $R$,
is therefore of interest, as the analytical knowledge of properties of the 
solutions can provide vital insight and help in choosing the most 
suitable regularisation model for a given problem.

Starting with \cite{lysaker2003noise}, and besides $\TGV$ and $\ICTV$, 
various other higher-order regularisation schemes have been proposed 
in the recent years \cite{burger2012regularized,papafitsoros2012combined,
chan2000high,DFLM2009,didas2009properties}. Curvature based
regularisers such as Euler's elastica \cite{chan2002euler,shen2003euler}
and  \cite{bertozzi2004low} have also recently received attention 
for the better modelling of curvature in images.
Further, non-convex total variation schemes are being studied for
their better modelling of real image gradient distributions
\cite{huang1999statistics,HiWu13_siims,HiWu14_coap,ochsiterated,tuomov-tvq}. 
In the other direction, ``lower-order schemes'' such as
Meyer's G-norm \cite{meyer2002oscillating,vese2003modelingtextures}
and $\TV$ with Kantorovich-Rubinstein discrepancy \cite{tuomov-krtv}
have recently been proposed for the improved modelling of 
texture in images.
Very little is known analytically about the solution of most of 
these models. We concentrate here on $\TV$ and, as a novelty,
Huber-regularised and other $\TV$ variants with convex energies.
We introduce these rigorously in \S\ref{sec:reg}. 
We also demonstrate that our technique would apply to non-convex
$\TV$ schemes, as well as the Perona-Malik anisotropic diffusion
\cite{perona1990scalespace,weickert1998anisotropic}, if these were
well-posed to begin with \cite{tuomov-tvq,guidotti2014anisotropic}.
In Part 2 we will discuss $\ICTV$ and $\TGV^2$. We moreover hope that our 
techniques will be useful and provide an impetus for the analytical 
study of other advanced regularisers as well.

In case of $\TGV^2$ regularisation, we know a little about solutions 
to \eqref{eq:prob} on one-dimensional domains $\Omega=(a,b)$ 
\cite{l1tgv,papafitsoros2013study}. In particular, we showed
in \cite{l1tgv} the jump set containment for $\phi(t)=t$,
the proof easily extensible to $\phi(t)=t^2$.
For first-order total variation regularisation,
i.e.~$R(u)=\alpha\TV(u)=\alpha\norm{Du}_{\EUspace}$, the literature is 
more plentiful on analytical results. With the squared $L^2$ fidelity, 
the problem \eqref{eq:prob} is also known as the 
Rudin-Osher-Fatemi \cite{Rud1992} problem (ROF), and written
\begin{equation}
    \label{eq:rof}
    %\tag{P-ROF}
    \min_{u \in \BVUspace} \frac{1}{2}\int_\Omega \abs{f(x)-u(x)}^2 \d x + \alpha \TV(u).
\end{equation}
The first structural results can be found in \cite{ring2000structural}, 
where the stair-casing property is studied. Regarding the jump set,
it is shown in \cite{caselles2008discontinuity} that 
$\H^{m-1}(J_u \setminus J_f)=0$.
The proofs of this and many other properties of total variation
regularisation heavily depend on the co-area formula for functions
of bounded variation, namely
\begin{equation}
    \label{eq:tv-coarea}
    \TV(u) = \int_{-\infty}^\infty \Per(\{u>t\}; \Omega) \d t,
\end{equation}
where $\Per(E; \Omega) \defeq \norm{D\chi_E}_{2,\Meas(\Omega;\R^m)}$
denotes the perimeter of the set $E$ within $\Omega$. 
Thanks to the co-area formula, the problem \eqref{eq:rof} can be 
shown to be equivalent to the family of minimal surface problems
\begin{equation}
    \label{eq:rof-levelset}
    \min_{E \subset \Omega}
        \int_E (t-f(x)) \d x
        +
        \alpha \Per(E; \Omega),
    \quad (\text{\ae} t \in \R).
\end{equation}
The level sets of $u$ can be found as solutions of \eqref{eq:rof-levelset}
for varying $t$.
This formulation and the regularity of the level sets is studied in
\cite{alter2005characterization,chambolle2004meanalgorithm,allard2008total}.

For TV regularisation with the $L^1$ fidelity $\phi(x)=x$, i.e., the problem
\begin{equation}
    \label{eq:l1-tv}
    \min_{u \in \Uspace} \int_\Omega \abs{f(x)-u(x)} \d x + \alpha\TV(u),
\end{equation}
various basic properties were first studied by \cite{chan2005aspects}.
These include thresholds for the regularisation parameter $\alpha > 0$,
under which the optimal solution $u$ equals $f$. Similar thresholds 
were also derived in \cite{l1tgv} for $\TGV^2$ regularisation in 
the case $m=1$. Some of these results readily generalise to $m > 1$.
In this context of parameter thresholds, we also mention \cite{Nik2002}.
In \cite{allard2008total} it is shown that
the level sets $E_t=\{u \ge t\}$ of solutions $u$ to \eqref{eq:l1-tv},
are solutions to the minimal surface problems
\begin{equation}
    \label{eq:levelset-split}
    \min_{E \subset \Omega} \int_E \sign(t-f(x))\d x + \alpha \Per(E; \Omega),
    \quad
    (\text{\ae} t \in \R).
\end{equation}
Some further properties of the level sets are studied
in \cite{duval2009tvl1}. We know from there that the essential
inclusion $\H^{m-1}(J_u \setminus J_f)=0$ does not generally hold,
but the remainder has curvature $\inv\alpha$ for $\alpha$ the
regularisation parameter.
In a similar fashion, the capability of \eqref{eq:l1-tv} to 
separate geometric objects according to their scales is studied
in \cite{Yin2007tvl1}. Finally, an interesting direction is taken
in \cite{benning2012ground} by studying singular vectors and 
ground states of regularisation functionals, eventually hoping
to obtain something resembling an eigendecomposition
for solutions to \eqref{eq:prob}. The singular vectors aside, 
all of these techniques heavily depend on the co-area formula.

As mentioned, in this paper, we introduce a novel technique
for the study of the jump set $J_u$ for rather general regularisers
$R$, which is not based on the co-area formula. It is, instead,
based on Lipschitz pushforwards. We push forward a purported solution
by two different Lipschitz transformations and show that this provides
a contradiction to $u$ solving \eqref{eq:prob} 
if $\H^{m-1}(J_u \setminus J_f)>0$, ($p>1$). 
Before starting to develop this technique in detail in 
\S\ref{sec:problem}, we now introduce some general notation,
concepts and tools in the following \S\ref{sec:prelim}.

%%%
\section{Notation and useful facts}
\label{sec:prelim}
%%%

We begin by introducing the tools necessary for our work.
First we introduce basic notation for sets, mappings,
and measures. We then move on to Lipschitz mappings 
and graphs. Our notation and definition of the latter will
be used extensively throughout the paper, as we perform
operations on the jump set $J_u$, which is $\H^{m-1}$-rectifiable.
%i.e., there exists countably many Lipschitz graphs $\{\Gamma_i\}_{i=1}^\infty$,
%such that $\H^{m-1}(J_u \setminus \Union_{i=1}^\infty \Gamma_i)=0$.
Having defined our notation for Lipschitz graphs, we
introduce functions of bounded variation.

%%%
\subsection{Basic notations}
%%%

We denote by $\{e_1,\ldots,e_m\}$ the standard basis of $\R^m$.
The boundary of a set $A$ we denote by $\BD A$, and the 
closure by $\closure A$.  The $\{0,1\}$-valued indicator function
we write as $\chi_A$.
We denote the open ball of radius $\rho$ centred at $x \in \R^m$ 
by $\B(x, \rho)$. We denote by $\omega_m$ the volume of the unit 
ball $\B(0, 1)$ in $\R^m$.
%Then $\B(x, \rho)$ has volume $V_m(\rho) \defeq \omega_m \rho^m$,
%while the the surface $\BD \B(x, \rho)$ has an area 
%of $S_{m-1}(\rho) \defeq m\omega_m\rho^{m-1}$.

For $z \in \R^m$, we denote by
$z^\perp \defeq \{ x \in \R^m \mid \iprod{z}{x}=0\}$
the hyperplane orthogonal to $z$ , whereas
$P_z$ denotes the projection operator onto the subspace
spanned by $z$, and $P_z^\perp$ the projection onto 
$z^\perp$. If $A \subset z^\perp$, we denote by
$\ri A$ the \emph{relative interior} of $A$ in
$z^\perp$ as a subset of $\R^m$.

Let $\Omega \subset \R^m$ be an open set.
We then denote the space of (signed) 
Radon measures on $\Omega$ by $\Meas(\Omega)$. 
If $V$ is a vector space, then the space of
Radon measures on $\Omega$ with values in $V$ 
is denoted $\Meas(\Omega; V)$.
The $k$-dimensional Hausdorff measure,
on any given ambient space $\R^m$, ($k \le m$), 
is denoted by $\H^k$,
while $\L^m$ denotes the Lebesgue measure on $\R^m$.

The total variation (Radon) norm of a measure $\mu$ is denoted 
$\norm{\mu}_{\Meas(\R^m)}$. For vector-valued measures
$\mu=(\mu_1,\ldots,\mu^k) \in \Meas(\Omega; \R^k)$, we use the notation
\begin{equation}
    \label{eq:radon-q}
    \norm{\mu}_{q, \Meas(\Omega; \R^k)}
    \defeq
    \sup\left\{ \int_\Omega \sum_{i=1}^k \varphi_i(x) \d \mu_i(x)
        \middle|
        \begin{array}{l}
        \varphi \in C_0^\infty(\Omega; \R^k),\\
        \norm{\varphi(x)}_p \le 1 \text{ for } x \in \Omega
        \end{array}
    \right\}
\end{equation}
to indicate that the finite-dimensional base norm is the $p$-norm where $1/p+1/q=1$.
When the choice of the finite-dimensional norm is inconsequential,
we use the notation $\norm{\mu}_{q, \Meas(\Omega; \R^k)}$.
In this work in practise we restrict ourselves to $q=2$ for measures.
In other words, we consider isotropic total variation type functionals.
We use the same notation for vector fields $w \in L^p(\Omega; \R^k)$,
namely
\[
    \norm{w}_{q, L^p(\Omega; \R^k)} \defeq \left(\int_\Omega \norm{w(x)}_q^p \d x \right)^{1/p}.
\]    

For a measurable set $A$, we  denote by $\mu \restrict A$ the 
restricted measure defined
by $(\mu \restrict A)(B) \defeq \mu(A \isect B)$.
The notation $\mu \ll \nu$ means that $\mu$ is absolutely
continuous with respect to the measure $\nu$, and $\mu \perp \nu$ 
that $\mu$ and $\nu$ are mutually singular. The singular and
absolutely continuous (with respect to the Lebesgue measure)
part of $\mu$ are denoted $\mu^a$ and $\mu^s$, respectively.

We denote the $k$-dimensional upper resp.~lower density
of $\mu$ by
\[
    \Theta^*_k(\mu; x) \defeq \limsup_{\rho \downto 0} \frac{\mu(\B(x, \rho))}{\omega_k \rho^k},
    \quad
    \text{resp.}
    \quad
    \Theta_{*,k}(\mu; x) \defeq \liminf_{\rho \downto 0} \frac{\mu(\B(x, \rho))}{\omega_k \rho^k}.
\]
The common value, if it exists, we denote by $\Theta_k(\mu; x)$. 

Finally, we often denote by $C$, $C'$, $C'''$
arbitrary positive constants,
%that are not always introduced in further detail, 
and use the plus-minus notation 
$a^\pm=b^\pm$ in to mean that both $a^+=b^+$ and $a^-=b^-$ hold.

\begin{comment}
\subsection{Differentials}

We denote by $\linear(V; W)$ the space of linear maps between
the vector spaces $V$ and $W$.
With $u \in C^1(\Omega; W)$, ($k\ge 0$), $\Omega \subset V$,
for finite-dimensional Hilbert spaces $V, W$,
the (Fr{\'e}chet) differential  $\DIFF u(x) \in \linear(V; W)$ 
at $x \in \Omega$ is defined by the limit
\[
    \lim_{h \to 0} \frac{\norm{u(x+h)-u(x)-\DIFF u(x) h}}{\norm{h}} = 0.
\]
If $V=W$, and $V$ has basis $\{\xi_1,\ldots,\xi_m\}$, we define 
the divergence $\divergence u \in C(\Omega; V)$ as
\[
    \divergence u(x)
    \defeq
    \sum_{i=1}^{m}
    %\iprod{\DIFF u(x)(e_i^1,c_2, \ldots,c_k)}{e_i^1}.
    \iprod{\DIFF u(x) \xi_i}{\xi_i}
\]
Green's identity
\[
    \int_{\Omega} \iprod{\DIFF u(x)}{\varphi(x)} \d x
    =
    - \int_{\Omega} u(x) \divergence \varphi(x) \d x
\]
holds for $u \in C^1(\Omega)$
and $\varphi \in C^1_0(\Omega; V)$.
\end{comment}

%%%
\subsection{Mappings from a subspace}
%%%

We denote by $\linear(V; W)$ the space of linear maps between
the vector spaces $V$ and $W$.
If $L \in \linear(V; \R^k)$, where $V \sim \R^n$, ($n \le k$),
is a finite-dimensional Hilbert space,
then $L^* \in \linear(\R^k; V^*)$ denotes the adjoint,
and the $n$-dimensional Jacobian is defined as
\cite{ambrosio2000fbv}
\[
    \jacobian_{n}[L] \defeq \sqrt{\det (L^* \circ L)}.
\]
With the gradient of a Lipschitz function $f: V \to \R^k$
defined in ``components as columns order'', 
$\grad f(x) \in \linear(\R^k; V)$, we extend this 
notation for brevity as
\[
    \jacobianf{n}{f}{x} \defeq \jacobian_{n}[(\grad f(x))^*].
\]
If $\Omega \subset V$ is a Borel set, 
and $g \in L^1(\Omega)$, the \term{area formula} 
may then be stated 
\begin{equation}
    \label{eq:areaformula}
    \int_{\R^k} \sum_{x \in \Omega \isect \inv f(y)} g(x) \d\H^{n}(y)
    = \int_\Omega g(x) \jacobianf{n}{f}{x} \d \H^n(x).
\end{equation}
That this indeed holds in our setting of finite-dimensional Hilbert
spaces $V \sim \R^n$ follows by a simple argument from the area formula
for $f: \R^n \to \R^k$, stated in, e.g, \cite{ambrosio2000fbv}. 
We only use the cases $V = z^\perp$ for some $z \in \R^m$ ($n=m-1$),
or $V=\R^m$ ($n=m$).

We also denote by
\[
    C^{2,\allLambda}(V) \defeq \Isect_{\lambda \in (0,1)} C^{2,\lambda}(V)
\]
the class of functions that are twice differentiable (as defined above
for tensor fields) with a Hölder continuous second differential for
all exponents $\lambda \in (0,1)$.

The Lipschitz factor of a Lipschitz mapping $f$ we denote by $\lip f$.
We also recall that a Lipschitz transformation $\gamma: U \to V$ with
$U, V \subset \R^m$ has the \term{Lusin $N$-property} if it maps
$\L^m$-negligible sets to $\L^m$-negligible sets.

If $\gamma: \Omega \to \Omega$ is a bijective Lipschitz transformation, 
and $u: \Omega \to \Omega$ a Borel function, we define the 
pushforward $u_\gamma \defeq \pf\gamma u \defeq u \circ \gamma^{-1}$.
Finally, we denote the identity transformation by $\iota(x)=x$.

%%%
\subsection{Lipschitz graphs}
%%%

A set $\Gamma \subset \R^m$ is called a Lipschitz ($m-1$)-graph
(of Lipschitz factor $L$), if there exist a unit vector $z_\Gamma$,
an open set $V_\Gamma \subset z_\Gamma^\perp$, and a Lipschitz map 
$f_\Gamma: V_\Gamma \to \R$, of Lipschitz factor at most $L$,
such that
\[
    \Gamma = \{ v+f_\Gamma(v)z_\Gamma \mid v \in V_\Gamma \}.
\]
We also define $g_\Gamma: V_\Gamma \to \R^m$ by
\[
    g_\Gamma(v)=v+z_\Gamma f_\Gamma(v).
\]
Then
\[
    \Gamma = g_\Gamma(V_\Gamma).
\]

We denote the open domains ``above'' and ``beneath''
$\Gamma$, respectively, by
\[
    \Gamma^+ \defeq \Gamma+(0,\infty)z_\Gamma,
    \quad
    \text{and}
    \quad
    \Gamma^- \defeq \Gamma+(-\infty, 0)z_\Gamma.
\]
%An illustration of these concepts may be found in 
%Figure \ref{fig:illustr-lipschitz}. 
We recall
that by Kirszbraun's theorem, we may extend 
the domain of $f_\Gamma$ and $g_\Gamma$ from $V_\Gamma$
to the whole space $z_\Gamma^\perp$ without
altering the Lipschitz constant. Then
$\Gamma$ splits $\Omega$ into the two open halves
$\Gamma^+ \isect \Omega$ and $\Gamma^- \isect \Omega$.
We often use this fact. 

%\begin{figure}
%    \centering
%    \input{lipschitz.tikz}
%    \caption{Illustration of our definitions regarding Lipschitz graphs.
%        The thick line represents $\Gamma$.}
%    \label{fig:illustr-lipschitz}
%\end{figure}

%%%
\subsection{Functions of bounded variation}% and sets of finite perimeter}
%%%

We say that a function $u: \Omega \to \R$ on a bounded 
open set $\Omega \subset \R^m$ is of \term{bounded variation}
(see, e.g., \cite{ambrosio2000fbv} for a more thorough introduction),
denoted $u \in \BVspace(\Omega)$, if $u \in L^1(\Omega)$, 
and the distributional gradient $\Dfull u$ is a Radon measure.
%We define the norm
%$\norm{u}_{\BVspace(\Omega)} \defeq \norm{u}_{L^1(\Omega)}
%    + \abs{D u}(\Omega)$.
Given a sequence $\{u^i\}_{i=1}^\infty \subset \BVspace(\Omega)$, 
weak* convergence is defined as $u^i \to u$ strongly in $L^1(\Omega)$
along with $\Dfull u^i \weaktostar \Dfull u$ weakly* in 
$\Meas(\Omega)$. The sequence converges \emph{strictly} if,
in addition to this, $\abs{Du^i}(\Omega) \to \abs{Du}(\Omega)$.
%strong convergence to $u \in \BV[K]{\Omega}$ is defined as 
%strong $L^1$ convergence $\norm{u^i - u}_{L^1(\Omega;\R^K)} \to 0$ 
%together with convergence of the total variation 
%$\abs{u-u^i}(\Omega) \to 0$.

We denote by $S_u$ the approximate discontinuity set, i.e.,
the complement of the set where the Lebesgue limit $\tilde u$ exists. 
The latter is defined by
\[
    \lim_{\rho \downto 0}  \frac{1}{\rho^m}
        \int_{\B(x, \rho)} \norm{\tilde u(x) - u(y)} \d y = 0.
\]

The distributional gradient can
be decomposed as $\Dfull u = \Dabs u \L^m + \Djump u + \Dcantor u$, where
the density $\Dabs u$ of the \term{absolutely continuous part} of
$\Dfull u$ equals (a.e.) the approximate differential of $u$.
We also define the \term{singular part} as $D^s u=D^j u + D^c u$.
The \term{jump part} $\Djump u$ may be represented as
\begin{equation}
    %\label{eq:ejump}
    \notag
    \Djump u = (u^+ - u^-) \otimes \nu_{J_u} \H^{m-1} \restrict J_u,
\end{equation}
where $x$ is in the \term{jump set} $J_u \subset S_u$ of $u$ if for some 
$\nu \defeq \nu_{J_u}(x)$ there exist two \emph{distinct} one-sided
traces $u^+(x)$ and $u^-(x)$, defined as satisfying
\begin{equation}
    %\label{eq:one-sided-trace}
    \notag
    \lim_{\rho \downto 0}  \frac{1}{\rho^m}
        \int_{\B^\pm(x, \rho, \nu)} \norm{u^\pm(x) - u(y)} \d y= 0,
\end{equation}
where
$\B^\pm(x, \rho, \nu) := \{ y \in \B(x,\rho) \mid \pm \iprod{y-x}{\nu} \ge 0\}$.
It turns out that $J_u$ is countably $\H^{m-1}$-rectifiable and $\nu$ is
(a.e.) the normal to $J_u$. This former means that
there exist Lipschitz $(m-1)$-graphs $\{\Gamma_i\}_{i=1}^\infty$ such
that $\H^{m-1}(J_u \setminus \Union_{i=1}^\infty \Gamma_i)=0$.
Moreover, we have $\H^{m-1}(S_u \setminus J_u)=0$. 
The remaining \term{Cantor part} $\Dcantor u$ vanishes
on any Borel set $\sigma$-finite with respect to $\H^{m-1}$.

We will frequently use the following basic properties of
densities of $Du$; for the proof, see, e.g.,
\cite[Proposition 3.92]{ambrosio2000fbv}.

\begin{proposition}
    \label{prop:bv-density}
    Let $u \in \BVspace(\Omega)$ for an open domain $\Omega \subset \R^m$.
    Define
    \begin{equation}
        \label{eq:tilde-js-u}
        \tilde S_u \defeq \{x \in \Omega \mid \Theta_{*,m}(\abs{Du}; x)=\infty\},
        \quad
        \text{and}
        \quad
        \tilde J_u \defeq \{x \in \Omega \mid \Theta_{*,m-1}(\abs{Du}; x) >0\}.
    \end{equation}
    Then the following decomposition holds.
    \begin{enumroman}
        \item $\Dabs u \L^m = D u \restrict (\Omega \setminus \tilde S_u)$.
        \item $\Djump u = D u \restrict \tilde J_u$, 
            precisely $\tilde J_u \supset J_u$, and $\H^{m-1}(\tilde J_u \setminus J_u)=0$.
        \item $\Dcantor u = Du \restrict (\tilde S_u \setminus \tilde J_u)$.
    \end{enumroman}
\end{proposition}

%%%    
\section{Problem statement}
\label{sec:problem}
%%%

We now have most of the tools needed to state 
our main results, particularly the containment of $J_u$ in $J_f$ 
modulo a $\H^{m-1}$-null set for convex $p$-increasing ($p>1$) $\phi$.
We just have to rigorously state our assumptions 
on the regularisation functional $R$ and the fidelity $\phi$.
These ensure firstly the existence of a solution $u \in \BVspace(\Omega)$
to \eqref{eq:prob}. Secondly we want to ensure that $R$ behaves almost
like $\TV$ under averaged Lipschitz transformations, and $\phi$
comparably slower.

\subsection{Admissible regularisation functionals and fidelities}

We begin by stating our assumptions on $R$.

\begin{definition}
    \label{def:admissible}
    We call $R$ an \emph{admissible regularisation functional}
    on $\Uspace$, where the domain $\Omega \subset \R^m$,
    if it is convex, lower semi-continuous with respect to weak*
    convergence in $\BVUspace$, and there exists $C > 0$ such that
    \begin{equation}
	\label{eq:r-tv-bound}
	\norm{Du}_{\EUspace} \le C\bigl(1+\norm{u}_{\Uspace}+R(u)\bigr),
	\quad (u \in \Uspace).
    \end{equation}
    Here and throughout the paper, unless otherwise indicated, the finite-dimensional base norm on $\R^m$ is the Euclidean or 2-norm. This is the one most appropriate to most image processing tasks due to its rotational invariance.
\end{definition}

\begin{subequations}
\label{eq:bitrans-components}
\begin{definition}
    %Let $\lambda_1(x) \le \ldots \le \lambda_m(x)$ be the eigenvalues of
    %$\lipjac{\gamma}(x) \defeq \sqrt{[\grad \gamma(x)]^* \grad \gamma(x)}$.
    We denote by $\LipClass(\Omega)$ the set of one-to-one Lipschitz
    transformations $\gamma: \Omega \to \Omega$ with
    $\inv \gamma$ also Lipschitz and both satisfying the Lusin $N$-property.
    With $U \subset \Omega$ an open set, we further denote
    \[
        \LipClass(\Omega, U) \defeq \{ \gamma \in \LipClass(\Omega) \mid \gamma(x)=x \text{ for } x \not\in U\}.
    \]
    With $\gammaone, \gammatwo \in \LipClass(\Omega)$, 
    we then define the \term{basic double-Lipschitz comparison constants}
    \begin{equation}
        \label{eq:bitrans}
        \bitrans{\gammaone}{\gammatwo} 
        \defeq
        %\sup_{x \in \Omega} \norm{\jacobian_m(\grad \gammaone(x)) \grad \inv \gammaone(\gammaone(x))
        %    +\jacobian_m(\grad \gammatwo(x)) \grad \inv \gammatwo(\gammatwo(x)) - 2 I}.
        %\sup_{x \in \Omega} \norm{[\lipjac{\gammaone}(x)]^* \lipjac{\gammaone}(x) + [\lipjac{\gammatwo}(x)]^* \lipjac{\gammatwo}(x) - 2I}
        \sup_{x \in \Omega, v \in \R^m, \norm{v}=1} \norm{\lipjac{\gammaone}(x) v} + \norm{\lipjac{\gammatwo}(x) v} - 2\norm{v}.
    \end{equation}
    and
    \begin{equation}
        \label{eq:bitransjac}
        \bitransjac{\gammaone}{\gammatwo} 
        \defeq
        \sup_{x \in \Omega} \abs{\jacobianf{m}{\gammaone}{x} + \jacobianf{m}{\gammatwo}{x} - 2}.
    \end{equation}
    Here $I$ is the identity mapping on $\R^m$, and
    \begin{equation}
        \label{eq:lipjac}
        \lipjac{\gamma}(x) \defeq \grad \inv \gamma(\gamma(x)) \jacobianf{m}{\gamma}{x}.
    \end{equation}
    We also define the distance-to-identity
    \begin{equation}
        \label{eq:didtarns}
        \didtrans{\gamma} \defeq \sup_{x \in \Omega} \norm{\grad \inv \gamma(\gamma(x))-I},
    \end{equation}
    and finally combine all of these into the overall
    \term{double-Lipschitz comparison constant}
    \begin{equation}
        \label{eq:bitransfull}
        \bitransfull{\gammaone}{\gammatwo}
        \defeq \bitrans{\gammaone}{\gammatwo}
            + \bitransjac{\gammaone}{\gammatwo}
            + \didtrans{\gammaone}^2
            + \didtrans{\gammatwo}^2.
    \end{equation}
\end{definition}
\end{subequations}

\begin{remark}
    With the help of $\bitransjac{\gammaone}{\gammatwo}$ and a bound on the
    Lipschitz factors, we could generally replace $\bitrans{\gammaone}{\gammatwo}$ by
    \[
        \tilde G_{\gammaone,\gammatwo}
        \defeq
        \sup_{x \in \Omega, \norm{v}=1} \norm{\grad \inv\gammaone(\gamma(x)) v} + \norm{\grad\inv\gammatwo(\gamma(x)) v} - 2\norm{v}.
    \]
    It however turns out that for our transformations of interest, $\lipjac{\gamma}$ is
    easier work with than $\grad \inv\gamma \circ \gamma$ directly. 
    With the help of a bound on the Lipschitz factor, we could also replace $\didtrans{\gamma}$ 
    by $\bitrans{\gamma}{\iota}$, and then $\bitransfull{\gammaone}{\gammatwo}$ by
    \[
        \tilde T_{\gammaone,\gammatwo}
        \defeq
          \overline T_{\gammaone,\gammatwo}
        + \overline T_{\gammaone,\iota}^2
        + \overline T_{\gammatwo,\iota}^2,
    \]
    where $\overline T_{\gammaone,\gammatwo} \defeq \bitrans{\gammaone}{\gammatwo} + \bitransjac{\gammaone}{\gammatwo}$
    provides a form of simultaneous distance of the two transformations to the identity.
    %In fact, using the area formula, it is not difficult to see that
    %if $\overline T_{\gammaone,\gammatwo}=0$, then
    %$u \mapsto \pf{(\gammaone \circ \inv\gammatwo)} u$ is an isometry on $\BVspace(\Omega)$
    %and $L^1(\Omega)$. Then, if $\Omega$ is connected, it can be seen that
    %$\gammaone \circ \inv\gammatwo$ is an isometry on $\Omega$. 
    %With $U \Subset \Omega$, it follows that $\gammaone=\gammatwo$.
    %
    Various further over-estimations of $\bitransfull{\gammaone}{\gammatwo}$ are possible,
    but these usually destroy the crucial $O(\rho^2)$ property that we will later
    derive for specific transformations.
    At this point, it is threfore not clear whether $\bitransfull{\gammaone}{\gammatwo}$ in 
    general can be replaced by something significantly simpler and intuitive; alternative
    uses of the comparison condition with different transformations from those in the present
    work, may have different requirements from the useful set of comparison constants,
    possibly in some ways sharper than those herein.
\end{remark}

\begin{definition}
    \label{def:lipschitz-trans}
    We say that $R$ is \term{double-Lipschitz comparable}
    if there exists a constant $\RCa=\RCa(\Omega)$
    such that
    \[
        \text{for every }
        \begin{cases}
            u \in \BVspace(\Omega), \\
            \text{open set } U \subset \Omega, \text{~and}\\
            \gammaone,\gammatwo \in \LipClass(\Omega, U) \text{ with }
            \bitransfull{\gammaone}{\gammatwo} < 1,
        \end{cases}
    \]
    there holds
    \[
        R(\pf\gammaone u) + R(\pf\gammatwo u) - 2 R(u) \le 
        \RCa \bitransfull{\gammaone}{\gammatwo} \abs{D u}(\closure U).
    \]
    We also say that $R$ is \term{separably double-Lipschitz comparable}
    if there exist constants $\RCa=\RCa(\Omega)$ and $\RCs=\RCs(\Omega)$
    such that
    \[
        \text{for every }
        \begin{cases}
            u \in \BVspace(\Omega), \\
            \text{open set } U \subset \Omega, \\
            \gammaone,\gammatwo \in \LipClass(\Omega, U) \text{ with } \bitransfull{\gammaone}{\gammatwo} < 1, & \text{and} \\
            \text{Lipschitz $(m-1)$-graph } \Gamma \\
        \end{cases}
    \]
    holds
    \[
        \begin{split}
        R(\pf\gammaone u) + R(\pf\gammatwo u) - 2 R(u) 
        &
        \le
        \RCa \bitransfull{\gammaone}{\gammatwo}
        \abs{D u}(\closure U \setminus \Gamma)
        \\
        &
        \phantom{\le}
        +
        \RCs \bigl(\abs{D \pf\gammaone u}(\gammaone(\Gamma)) + \abs{D \pf\gammatwo u}(\gammatwo(\Gamma)) - 2 \abs{D u}(\Gamma)\bigr).
        \end{split}
    \]
\end{definition}

\begin{remark}
    Strictly speaking, we only need a local $\H^{m-1}$-\ae
    version of double-Lipschitz comparability, but the
    regularisation functionals in this Part 1
    satisfy the stronger and simpler definition above.
    Therefore we use it. In Part 2, we will need to 
    consider much more detailed variants.
    Also the bound $\bitransfull{\gammaone}{\gammatwo} < 1$
    is primarily needed for $\didtrans{\gammaone}, \didtrans{\gammatwo} < 1$
    and an arbitrary bound on $\bitransjac{\gammaone}{\gammatwo}$ for
    the treatment of Huber-regularised $\TV$.
\end{remark}

In order to show the existence of solutions to \eqref{eq:prob}, 
we require the following property from $\phi$.

\begin{definition}
    Let the domain $\Omega \subset \R^m$.
    We call $\phi: \R \to [0, \infty]$ an \term{admissible fidelity function on $\Omega$} if
    $\phi$ is convex and lower semi-continuous, $\phi(0)=0$, $\phi(-t)=\phi(t)$, ($t>0$),
    and for some $C>0$ holds
    \begin{equation}
        \label{eq:phi-l1-bound}
        \norm{u}_{\Uspace}  \le C\left(\int_\Omega \phi(u(x)) \d x + 1 \right),
        \quad
        (u \in \Uspace).
    \end{equation}
\end{definition}

For the study of the jump set $J_u$ of solutions to \eqref{eq:prob},
we require additionally the following increase criterion
to be satisfied by $\phi$.

\begin{definition}
    \label{def:phi-increase}
    We say that $\phi$ is \term{$p$-increasing} for $p \ge 1$,
    if there exists a constant $C_\phi>0$ for which
    \begin{equation}
        %\label{eq:phi-increase}
        \notag
        \phi(x) - \phi(y) \le C_\phi (\abs{x}-\abs{y})\abs{x}^{p-1},
        \quad
        (x, y \in \R).
    \end{equation}
\end{definition}

\begin{remark}
    Definition \ref{def:phi-increase} implies that $\phi$
    is increasing. In fact $\phi'(x) \ge C_\phi \abs{x}^{p-1}$.
\end{remark}

\begin{example}
    Let $\phi(x)=\abs{x}^p$, ($p \ge 1$). Then $\phi$ is $p$-increasing
    with $C_\phi=p$ because
    \[
        \phi(x) - \phi(y) \le \phi'(x)(\abs{x}-\abs{y})
        =p(\abs{x}-\abs{y})\abs{x}^{p-1}.
    \]
    The $L^1$ fidelity $\phi(x)=\abs{x}$ is admissible for any $\Omega \subset \R^m$,
    including $\Omega=\R^m$. The $L^p$ fidelities $\phi(x)=x^p$ for $p>1$
    are admissible for any bounded $\Omega \subset \R^m$. 
    %Indeed, any
    %convex, lower semi-continuous, increasing $\phi: [0, \infty) \to [0,\infty]$ 
    %with $\phi(0)=0$ is admissible for bounded $\Omega$,
    %see, e.g., \cite{fonseca2007mmc}.
\end{example}

The following, standard, result states that the problem \eqref{eq:prob}
is well-posed under the above assumptions. We may therefore proceed
with the analysis of the structure of the solutions
$u \in \BVspace(\Omega)$.

\begin{theorem}
    Let $f \in \Uspace$ satisfy $\int_\Omega \phi(f(x)) \d x < \infty$. 
    Suppose that $R$ is an admissible regularisation functional
    on $\Uspace$, and $\phi$ an admissible fidelity function for $\Omega$.
    Then there exists a solution $u \in \Uspace$ to \eqref{eq:prob},
    and any solution satisfies $u \in \BVUspace$.
\end{theorem}
\begin{proof}
    Clearly  the minimum in \eqref{eq:prob} is finite.
    Let $\{u^i\}_{i=0}^\infty$ be a minimising sequence
    for \eqref{eq:prob}. Minding that \eqref{eq:r-tv-bound}, 
    \eqref{eq:phi-l1-bound} bound $\{u^i\}_{i=0}^\infty$ 
    in $\BVUspace$, and that $R$ is lower semi-continuous
    with respect to weak* convergence in $\BVUspace$, and $\phi$ is
    lower semi-continuous, the claim follows by the 
    standard method of calculus of variations, see, e.g.,
    \cite{fonseca2007mmc}.
\end{proof}

\subsection{The main results}

Our main task in the rest of this paper is to prove the 
following result. 

\begin{theorem}
    \label{theorem:jumpset-strict}
    Let the domain $\Omega \subset \R^m$ be bounded with Lipschitz boundary.
    Suppose $R$ is an admissible double-Lipschitz comparable
    regularisation functional on $\Uspace$,
    and $\phi: [0, \infty) \to [0, \infty)$ 
    an admissible $p$-increasing fidelity function 
    for some $1 < p < \infty$.
    Let $f \in \BVUspace \isect L^\infty_\loc(\Omega)$, 
    and suppose $u \in \BVUspace \isect L^\infty_\loc(\Omega)$ 
    solves \eqref{eq:prob}.
    Then
    \[
        \H^{m-1}(J_u \setminus J_f)=0.
    \]
\end{theorem}

\begin{remark}
    Observe that we require $u$ to be locally bounded. 
    This does not necessarily hold, and needs to be proved
    separately. It is well-known that it holds for $\TV$
    if $f \in L^\infty(\Omega)$, and is easy to show for 
    Huber-$\TV$ using the same barrier technique.
\end{remark}

We also show the following. We note that we only get strong
regularity if the solution is approximately piecewise constant.

\begin{theorem}
    \label{theorem:jumpset-l1}
    Let the domain $\Omega \subset \R^m$ be bounded with Lipschitz boundary.
    Suppose $R$ is an admissible separably double-Lipschitz
    comparable regularisation functional on $\Uspace$, and
    $\phi: [0, \infty) \to [0, \infty)$ an admissible 
    $1$-increasing fidelity function.
    Let $f \in \BVUspace$, and suppose $u \in \BVUspace$ % \isect L^\infty_\loc(\Omega)$ 
    solves \eqref{eq:prob}. Then 
    \[
        \H^{m-1}(J_u \setminus (J_f \union \Lambda))=0,
    \]
    where $\Lambda = \Union_{i=1}^\infty \Lambda_i$
    for Lipschitz graphs $\{\Lambda_i\}_{i=1}^\infty$
    such that at $\H^{m-1}$-\ae point $x \in (\Lambda_i \isect J_u) \setminus J_f$,
    the \emph{$R$-curvature of $u$ along $\Lambda_i$} satisfies 
    \begin{equation}
        \label{eq:jumpset-l1-r-curvature}
        \CURVP_u^{R,\Lambda_i}(x)=\abs{u^+(x)-u^-(x)} C_\phi.
    \end{equation}
    This technical definition will be provided later
    in Definition \ref{def:r-curvature}.
    
    If $\Theta_m(\abs{Du} \restrict \Omega \setminus \Lambda_i; x)=0$
    at $\H^{m-1}$-\ae point $x \in \Lambda_i$, 
    then each $\Lambda_i$, ($i \in \Z^+$), is of class $C^{2,\allLambda}$ 
    and curvature $C_\phi/\RCs$ in the sense that
    $f_{\Lambda_i} \in C^{2,\allLambda}(V_{\Lambda_i})$ and
    \begin{equation}
        \label{eq:jumpset-l1-curvature}
        %\adaptabs{
        -\divergence \frac{\DIFFSS f_{\Lambda_i}(v)}{\sqrt{1+\norm{\DIFFSS f_{\Lambda_i}(v)}^2}}
        %}
        =C_\phi/\RCs,
        \quad (v \in V_{\Lambda_i}).
    \end{equation}
\end{theorem}

\begin{remark}
    It is not very difficult to improve Theorem \ref{theorem:jumpset-l1}
    a little bit. Namely, we can replace the assumption
    $\Theta_m(\abs{Du} \restrict \Omega \setminus \Lambda_i; x)=0$
    by that of $\grad u$ having one-sided Lebesgue limits at $x$,
    with corresponding normal $\nu=\nu_{J_u}(x)$. We will in another
    context in Part 2 study techniques that would allow us 
    to do this. We do not however pursue this route of improving
    Theorem \ref{theorem:jumpset-l1}, as the small improvement 
    would still not be entirely satisfactory -- at least not
    without corresponding results to prove that the limits
    actually do exist. This fascinating question is outside
    the scope of the present manuscript.
\end{remark}

\begin{remark}
    For a demonstration that we cannot in general set $\Lambda=\emptyset$ in
    Theorem \ref{theorem:jumpset-l1}, we refer to the comprehensive treatment
    in \cite{duval2009tvl1} about the case of $R=\TV$, $\phi=\abs{\cdot}$, 
    and $f=\chi_A$ for a suitable set $A \subset \Omega$. For example if $A$
    is  a square, then the optimal solution $u$ will be a square with
    rounded corners of curvature $1/\alpha=C_\phi/R^s$
    (or the empty set, when this is not possible).
    Using \cite[Theorem 8]{papafitsoros2015asymptotic}, this example
    can also be extended to the higher-order regulariser $\TGV^2$ \cite{bredies2009tgv},
    which we treat in Part 2 \cite{tuomov-jumpset2}.
\end{remark}

\section{First-order regularisation functionals}
\label{sec:reg}

Before embarking on the proofs of the main results, we introduce
a class of admissible first-order regularisation functionals:
the conventional total variation, as well as a class with 
an additional convex energy, including Huber-regularised TV.
We finish by discussing the Perona-Malik %\cite{perona1990scalespace}
and non-convex $\TV$ models %\cite{huang1999statistics,HiWu13_siims,HiWu14_coap,ochsiterated,tuomov-tvq}
in a few remarks. In Part 2 \cite{tuomov-jumpset2}
we will concentrate on higher-order
regularisation functionals: second-order total generalised 
variation ($\TGV^2$), and infimal convolution TV (ICTV),
whose analysis is more involved. We however remark that both
Theorem \ref{theorem:jumpset-strict} and Theorem \ref{theorem:jumpset-l1}
can easily be derived for ICTV from the corresponding result for TV.
For $\TGV^2$ this is not the case.

%%%
\subsection{Total variation}
%%%

As we well recall, (isotropic) total variation is defined as
\[
    \TV(u) \defeq \abs{Du}(\Omega)
        =\sup\left\{ \int_\Omega \iprod{\divergence \varphi(x)}{u(x)} \d x
            \middle| 
            \varphi \in C_c^\infty(\Omega), \norm{\varphi}_{2,L^\infty(\Omega)} \le 1
        \right\}.
\]
We now show that it satisfies the conditions of 
Theorem \ref{theorem:jumpset-strict}.

\begin{proposition}
    \label{prop:tv-admissible}
    The functional $R(u)=\alpha\TV(u)$ for $\alpha>0$ is admissible
    and (separably) double-Lipschitz comparable.
    Moreover, $\norm{u}_{L^\infty(\Omega)} \le \norm{f}_{L^\infty(\Omega)}$
    for solutions $u$ to \eqref{eq:prob}.
\end{proposition}

We use the following simple lemma for the proof.

\begin{lemma}
    \label{lemma:da-gammau}
    Let $u \in \BVspace(\Omega)$ and $\gamma \in \LipClass(\Omega)$. Then
    $
        D^a(\pf\gamma u) =  \grad \inv \gamma \pf\gamma \grad u \L^m,
    $
    that is,
    \[
        \grad(\pf\gamma u)(x)
        = \grad \inv \gamma(x) \pf\gamma \grad u (x)
        %= \grad \inv\gamma(\inv\gamma(x)) \grad u(\inv \gamma(x)),
        = \pf\gamma(\inv{[\grad \gamma]} \grad u)(x),
        \quad
        (\L^m\text{-\ae } x \in \Omega).
    \]
\end{lemma}

\begin{proof}
    We observe first of all that the inverse function theorem trivially holds almost
    everywhere for $\gamma \in \LipClass(\Omega)$. Indeed, using Rademacher's theorem 
    and the Lusin $N$-property, we see that $\iota(x) = \gamma(\inv \gamma(x))$ satisfies
    $\grad \iota(x) =  \grad \gamma(\inv\gamma(x)) \grad \inv \gamma(x)$
    for $\L^m$-\ae $x \in \Omega$. Repeating the same argument on the formulation
    $\iota(x) = \inv \gamma(\gamma(x))$, therefore
    $[\grad \gamma(\inv\gamma(x))]^{-1}=\grad \inv \gamma(x)$ for $\L^m$-\ae $x \in \Omega$.
    % See also Clarke.
    
    By the Calder\'on-Zygmund theorem, $D^a v=\grad v \L^m$, where $\grad v$ is the approximate
    differential of $v \in \BVspace(\Omega)$. This is defined at almost every $x \in \Omega$ as 
    $\grad v(x)=L$, where $L$ satisfies
    \[
        \lim_{\rho \downto} \dashint_{\B(x, \rho)} \frac{\abs{v(y)-\tilde v(x)-\iprod{L}{y-x}}}{\rho} \d y = 0.
    \]
    Let $x \in \Omega$ be a point such that $\tilde u(y)$ and $\grad u(y)$ exist for $y=\inv\gamma(x)$.
    Since $\gamma$ has the Lusin $N$-property, $\L^m$-\ae $x \in \Omega$ satisfies this.
    Clearly, by a simple application of the area formula, if $v=\pf\gamma u$, then
    $\tilde v(x)=\tilde u(y)$. Therefore, if we define 
    %\TODO{Need point of continuous differentiability of $\gamma$ for inverse function theorem?} 
    $L=[\grad \gamma(y)]^{-1} \grad u(y)=\grad \inv \gamma(x) \pf\gamma\grad u(x)$, 
    it is easily seen that $\grad(\pf\gamma u)(x)=L$ exists, and has the required form.
\end{proof}

\begin{proof}[Proof of Proposition \ref{prop:tv-admissible}]
    The requirements of admissibility in Definition \ref{def:admissible} are 
    trivial in this case.
    Also $\norm{u}_{L^\infty(\Omega)} \le M \defeq \norm{f}_{L^\infty(\Omega)}$ 
    is well-known for solutions $u$ to \eqref{eq:prob} with total variation
    regularisation. Indeed, comparing a purported solution $u$ that violates this
    with $\tilde u \defeq \min\{\max\{-M, u\}, M\}$, and referring to the co-area
    formula \eqref{eq:tv-coarea}, we easily obtain a contradiction.

    We therefore only have to prove (separable) double-Lipschitz comparability.
    We may without loss of generality take $\alpha=1$.
    We let $\gamma \in \LipClass(\Omega; U)$ for some open set $U \subset \Omega$,
    and pick $u \in \BVUspace$. By Lemma \ref{lemma:da-gammau}, we have
    \begin{equation}
        \label{eq:tv-sep}
        \begin{split}
        \abs{D \pf\gamma u}(\Omega)
        &
        =\abs{D^a \pf\gamma u}(\Omega) + \abs{D^s \pf\gamma u}(\Omega)
        \\
        &
        =\abs{ \grad \inv \gamma \pf\gamma \grad u \L^m}(\Omega) + \abs{D^s \pf\gamma u}(\Omega).
        \end{split}
    \end{equation}
    Since $\gamma(x) = x$ for $x \in \Omega \setminus U$, we may calculate
    using the area formula
    \[
        \begin{split}
        \abs{ \grad \inv \gamma \pf\gamma \grad u \L^m}(\Omega)
        &
        =\int_\Omega \norm{ \grad \inv \gamma(x) \grad u(\inv\gamma(x))} \d x
        %\\
        %&
        %=
        %\abs{D^a u}(\Omega \setminus U)
        %+
        %\int_U \norm{[\grad \inv \gamma(x)]^*\grad u(\inv\gamma(x))} \d x
        %\\
        %&
        %=
        %\abs{D^a u}(\Omega \setminus U)
        %+
        %\int_U \norm{\grad \inv \gamma(\gamma(x)) \grad u(x)} \jacobian_m[\grad \gamma(x)] \d x.
        \\
        &
        =
        \abs{D^a u}(\Omega \setminus U)
        +
        \int_U \norm{\lipjac{\gamma}(x) \grad u(x)} \d x.
        \end{split}
    \]
    Thus, with $\gammaone,\gammatwo \in \LipClass(\Omega, U)$,
    referring to the definition of $\bitrans{\gammaone}{\gammatwo}$, we obtain
    \[
        \begin{split}
        \abs{ \grad \inv \gammaone \pf\gammaone \grad u \L^m}(\Omega) 
        +
        \abs{ \grad \inv \gammatwo \pf\gammatwo \grad u \L^m}(\Omega) 
        - 2 \abs{D^a u}(\Omega)
        \le
        %\frac{1}{2}\bitrans{\gammaone}{\gammatwo}
        %\int_U \norm{\grad u(x)} \d x
        %=
        \bitrans{\gammaone}{\gammatwo}
        \abs{D^a u}(U).
        \end{split}
    \]
    Recalling \eqref{eq:tv-sep}, and minding that
    \[
        \abs{D^s \pf\gammaone u}(\Omega \setminus U)
        =
        \abs{D^s \pf\gammatwo u}(\Omega \setminus U)
        =
        \abs{D^s u}(\Omega \setminus U),
    \]
    we deduce
    \begin{equation}
        \label{eq:tv-double-lip-singular}
        \begin{split}
        \abs{D \pf\gammaone u}(\Omega)
        &
        +
        \abs{D \pf\gammatwo u}(\Omega)
        -2\abs{D u}(\Omega)
        \\
        &
        \le
        \bitrans{\gammaone}{\gammatwo}
        \abs{D^a u}(U)
        +
        \bigl(
        \abs{D^s \pf\gammaone u}(U)
        + \abs{D^s \pf\gammatwo u}(U)
        -
        2 \abs{D^s u}(U)
        \bigr).
        \end{split}
    \end{equation}
    This almost proves (separable) double-Lipschitz comparability,
    we just have to modify the singular part appropriately.
    
    To see (non-separable) double-Lipschitz comparability,
    we take a strictly converging approximation sequence
    $\{u^i\}_{i=1}^\infty \subset C^1(\Omega)$
    of $u$. Then $D^s u^i=0$, so
    \eqref{eq:tv-double-lip-singular} proves
    \[
        \abs{D \pf\gammaone u^i}(\Omega)
        +
        \abs{D \pf\gammatwo u^i}(\Omega)
        -2\abs{D u^i}(\Omega)
        \le
        \bitrans{\gammaone}{\gammatwo}
        \abs{D^a u^i}(U),
        \quad
        (i=1,2,3,\ldots).
    \]
    Since the strict convergence of $\{u^i\}_{i=1}^\infty$ to $u$
    bounds the right hand side, we deduce 
    \[
        \sup_i \abs{D \pf\gammaone u^i}(\Omega)
        +
        \sup_i \abs{D \pf\gammatwo u^i}(\Omega)
         < \infty.
    \]
    We may therefore extract a subsequence, unrelabelled, such that both
    $\{\pf\gammaone u^i\}_{i=1}^\infty$
    and
    $\{\pf\gammatwo u^i\}_{i=1}^\infty$
    are convergent weakly* to some $\uone \in \BVspace(\Omega)$
    and $\utwo \in \BVspace(\Omega)$, respectively.
    Moreover, by lower semicontinuity of the total variation, and strict convergence of the approximating
    sequence
    \begin{equation}
        \notag
        \begin{split}
        \abs{D \uone}(\Omega) 
        +
        \abs{D \utwo}(\Omega) 
        - 2 \abs{Du}(\Omega)
        &
        \le
        \liminf_{i \to \infty} \abs{D \pf\gammaone u^i}(\Omega) + \abs{D \pf\gammatwo u^i}(\Omega) - 2 \abs{Du^i}(\Omega)
        \\
        &
        \le \liminf_{i \to \infty} \bitrans{\gammaone}{\gammatwo} \abs{D u^i}(U).
        \end{split}
    \end{equation}
    Let us pick an open set $U' \supset U$ such that $\abs{Du}(\BD U')=0$.
    Then $\abs{D u^i}(U') \to \abs{D u}(U')$ because $u^i \to u$ strictly in $\Uspace$;
    see \cite[Proposition 1.62]{ambrosio2000fbv}. It follows
    \[
        \abs{D \uone}(\Omega) 
        +
        \abs{D \utwo}(\Omega) 
        - 2 \abs{Du}(\Omega)
        \le
        \bitrans{\gammaone}{\gammatwo} \abs{D u}(U').
    \]
    By taking the intersection over all admissible $U' \supset U$,
    we deduce
    \begin{equation}
        \label{eq:tv-almost-double-lip}
        \abs{D \uone}(\Omega) 
        +
        \abs{D \utwo}(\Omega) 
        - 2 \abs{Du}(\Omega)
        \le
        \bitrans{\gammaone}{\gammatwo} \abs{D u}(\closure U).
    \end{equation}
    This is almost the double-Lipschitz comparability. We just have to
    show that $\uone = \pf\gammaone u$ and $\utwo = \pf\gammatwo u$. 
    Indeed
    \[
        \begin{split}
        \int_\Omega \abs{\uone(x) - \pf\gammaone u(x)} \d x
        &
        \le
        \int_\Omega \abs{\uone(x) - \pf\gammaone u^i(x)} \d x
        +
        \int_\Omega \abs{\pf\gammaone u(x) - \pf\gammaone u^i(x)} \d x
        \\
        &
        \le
        \int_\Omega \abs{\uone(x) - \pf\gammaone u^i(x)} \d x
        +
        C\int_\Omega \abs{u(x) - u^i(x)} \d x,
        \end{split}
    \]
    where
    \[
        C \defeq \sup_x \jacobianf{m}{\gamma}{x}
        \le (\lip \gamma)^m < \infty.
    \]
    But the integrals on the right hand side tend to zero since the
    strict and weak* convergences imply strong convergence in $L^1$
    of $\{u^i\}_{i=1}^\infty$ to $u$ and of $\{\pf\gammaone u^i\}_{i=1}^\infty$
    to $\uone$. It follows that $\uone = \pf\gammaone u$.
    Analogously we show that $\utwo = \pf\gammatwo u$.
    Double-Lipschitz comparability is now immediate from \eqref{eq:tv-almost-double-lip}.
    
    Finally, to see separable double-Lipschitz comparability,
    we proceed analogously as above, but approximate $u$ on
    both sides of $\Gamma$. More specifically, referring to Kirzbraun's
    theorem, we may assume that $V_\Gamma = z^\perp$. Thus $\Omega$
    splits into two domains $\Omega^\pm \defeq \Omega \isect \Gamma^\pm$.
    We approximate $u$ separately on both $\Omega^+$ and $\Omega^-$ 
    by strictly converging sequences of $C^1$  functions
    $\{u^{(+),i}\}_{i=1}^\infty$ and $\{u^{(-),i}\}_{i=1}^\infty$. 
    By the continuity of the trace operator with respect to strict
    convergence in $\BVspace(\Omega)$ (see, e.g., \cite[Theorem 3.88]{ambrosio2000fbv}),
    also $u^i \defeq u^{(+), i} + u^{(-), i}$ then converge strictly 
    to $u$. Moreover, $D^s u^i = D u^i \restrict \Gamma$.
    By \eqref{eq:tv-double-lip-singular} we therefore have
    \[
        \begin{split}
        \abs{D \pf\gammaone u^i}(\Omega)
        &
        +
        \abs{D \pf\gammatwo u^i}(\Omega)
        -2\abs{D u^i}(\Omega)
        \\
        &
        \le
        \bitrans{\gammaone}{\gammatwo}
        \abs{D^a u^i}(U)
        +
        \bigl(
        \abs{D \pf\gammaone u^i}(\gammaone(\Gamma))
        + \abs{D \pf\gammatwo u^i}(\gammatwo(\Gamma))
        -
        2 \abs{D u^i}(\Gamma)
        \bigr).
        \end{split}
    \]
    Since the traces of $u^i$ on $\Gamma$ converge in $L^1(\Gamma)$
    to the trace of $u$ on $\Gamma$, we deduce the separable
    double-Lipschitz property by analogous arguments as 
    the (non-separable) double-Lipschitz property above.
\end{proof}

\begin{remark}
    Observe that we obtained the estimate
    \[
        \TV(\pf\gammaone u) + \TV(\pf\gammatwo u) - 2 \TV(u) \le 
        \bitrans{\gammaone}{\gammatwo} \abs{D u}(\closure U)
    \]
    that involves none of $\didtrans{\gammaone}^2$, $\didtrans{\gammatwo}^2$,
    and $\bitransjac{\gammaone}{\gammatwo}$.
    If we take $\gammatwo(x)=\iota(x) \defeq x$, then the property shows
    \[
        \abs{D \pf\gammaone u}(\Omega) \le \left(1 + \bitrans{\gammaone}{\iota}\right) \abs{Du}(\Omega).
    \]
    This can in specific cases can improve upon standard estimates 
    \cite{ambrosio2000fbv} on the total variation under Lipschitz 
    transformations.

    Observe also that our proof does not depend on the finite-dimensional norm in the definition of $\TV$ being the Euclidean norm. The result holds for any choice of finite-dimensional norm as long as the definition of $\bitrans{\gammaone}{\gammatwo}$ also reflects this.
    In §\ref{sec:lipschitz} we will however depend on the properties of
    the Euclidean norm.
\end{remark}

\subsection{Huber-regularised TV}

For a parameter $\hubergamma>0$, Huber-regularised total variation may be defined as
\begin{equation}
    \label{eq:huber-tv-def}
    \TV_\hubergamma(u)
    \defeq
    \sup\left\{
        \int_\Omega u \divergence \varphi \d x - \frac{\hubergamma}{2}\norm{\varphi}_{L^2(\Omega; \R^m)}^2
        \middle|
            \begin{array}{l}
            \varphi \in C_c^\infty(\Omega; \R^m), \\
            \norm{\varphi}_{2,L^\infty(\Omega; \R^m)} \le 1
            \end{array}
    \right\}.%,
    %\quad
    %(u \in \BVUspace).
\end{equation}
If $u$ is smooth, this corresponds to replacing the pointwise $2$-norm by
\[
    \huber[\hubergamma]{g} = 
    \begin{cases}
        \norm{g}_2 - \frac{1}{2\hubergamma}, & \norm{g}_2 \ge 1/\hubergamma,
        \\
        \frac{\hubergamma}{2}\norm{g}_2^2, & \norm{g}_2 < 1/\hubergamma,
    \end{cases}
\]
giving
\[
    \TV_\hubergamma(u) = \int_\Omega \huber[\hubergamma]{\grad u(x)} \d x,
    \quad
    (C^1(\Omega)).
\]
Huber-regularisation of TV is sometimes helpful numerically,
especially in the context of second-order optimisation methods
\cite{hintermuller2006infeasible},
as well as primal-dual methods for non-convex problems
\cite{tuomov-nlpdhgm}.
It also helps to avoid the stair-casing effect to some extent.
As with higher-order regularisers, there is no apparent useful 
coarea formula for Huber-TV, that would allow us to show regularity
properties through level sets. However, $\TV_\hubergamma$ satisfies
our assumptions, as stated in the following. In fact, since our
proof is based on rather general properties, we consider
a slightly larger class of functionals, based on a class of 
convex energies.

\begin{definition}
    We denote by $\Psi$ the family of increasing convex functions
    $\psi: [0, \infty) \to [0, \infty)$ with $\psi(0)=0$, that
    satisfy the following two properties. Firstly, 
    $0 < \psi^\infty < \infty$ for
    \[
        \psi^\infty \defeq \lim_{t \upto \infty} \psi(t)/t.
    \]
    Secondly, $\psi$ satisfies for some constants $K_\psi, C_\psi>0$ the property
    \begin{equation}
        \label{eq:limited-coco}
        \iprod{\subdiff \psi(t) - \subdiff \psi(s)}{t-s}
        \le 
        \begin{cases}
            0, & \min\{t, s\} \in [K_\psi, \infty), \\
            C_\psi \abs{t-s}^2, & \min\{t, s\} \in [0, K_\psi).
        \end{cases}
    \end{equation}
\end{definition}

%\begin{remark}
%    The property \eqref{eq:limited-coco} could in lack of a better name be called
%   ``limited co-coercivity'', following \cite[Theorem 18.15]{bauschke2011convex},
%   which shows the equivalence of the co-coercivity
%    \[
%        \iprod{\grad \psi(t)-\grad \psi(s)}{t-s} \ge c \abs{\grad \psi(t)-\grad \psi(s)}^2,
%        \quad (t,s \in \R),
%    \]
%    of $\grad \psi$ for a Fr\'echet differentiable convex function $\psi: \R \to \R$ to
%    \[
%        \iprod{\grad \psi(t)-\grad \psi(s)}{t-s} \le \inv c \abs{t-s}^2,
%        \quad (t,s \in \R).
%    \]
%\end{remark}

In essence, the threshold $K_\psi$ says that $\TV_\psi$, defined after
the next lemma, behaves linearly like $\TV$ for large gradients.
This of course holds for Huber-$\TV$.

\begin{lemma}
    We have $\huber[\hubergamma]{\freevar} \in \Psi$, ($\eta > 0$).
\end{lemma}

\begin{proof}
    Only \eqref{eq:limited-coco} demands proof.
    We set $\psi(t) \defeq \huber[\hubergamma]{t}$.
    Clearly, if $s, t \ge 1/\hubergamma$, we have
    \[
        \iprod{\subdiff \psi(t) - \subdiff \psi(s)}{t-s}=0,
    \]
    so the property holds.
    If $s, t < 1/\hubergamma$, we have
    \[
        \iprod{\subdiff \psi(t) - \subdiff \psi(s)}{t-s}
        = \hubergamma\abs{t-s}^2,
    \]
    so we have \eqref{eq:limited-coco} in this range.
    In the case $s < 1/\hubergamma < t$, we also have
    \[
        \iprod{\subdiff \psi(t) - \subdiff \psi(s)}{t-s}
        =(1 - \hubergamma s)(t-s)
        \le
        (\hubergamma t - \hubergamma s)(t-s)
        \le \hubergamma \abs{t-s}^2.
    \]
    The property \eqref{eq:limited-coco} follows
    with $K_{\psi} = 1/\hubergamma$
    and $C_{\psi} = \hubergamma$.
\end{proof}

\begin{definition}
    Let $\psi \in \Psi$. Then we set
    \[
        \TV_\psi(u) \defeq \int_\Omega \psi(\norm{\grad u(x)}) \d x
        +
        \psi^\infty \abs{D^s u}(\Omega),
        \quad
        (u \in \BVUspace).
    \]
\end{definition}

\begin{remark}
    It can be shown that
    \begin{equation}
        \label{eq:psi-tv-def}
        \TV_\psi(u)
        =
        \sup\left\{
            \int_\Omega u(x) \divergence \varphi(x) \d x - \int_\Omega \psi^*(\norm{\varphi(x)}) \d x 
            \Bigm|
                \varphi \in C_c^\infty(\Omega; \R^m)
        \right\}.%,
        %\quad
        %(u \in \BVUspace).
    \end{equation}
\end{remark}

We now claim that $\TV_\psi$ satisfies the requirements
of our main results, Theorem \ref{theorem:jumpset-strict} 
and Theorem \ref{theorem:jumpset-l1}.

\begin{proposition}
    \label{prop:huber-admis}
    Let $\Omega \subset \R^m$ be a bounded domain with Lipschitz boundary
    and $\psi \in \Psi$.
    Then $\TV_\psi$ is an admissible (separably) double-Lipschitz 
    comparable regularisation functional on $\Uspace$.
    Moreover, $\norm{u}_{L^\infty(\Omega)} \le \norm{f}_{L^\infty(\Omega)}$
    for solutions $u$ to \eqref{eq:prob}.
\end{proposition}

For the proof, we require the estimate from the following lemma.

\begin{lemma}
    \label{lemma:huber-biapprox}
    Let $\psi \in \Psi$ and $A, B: \R^m \to \R^m$ be linear transformations,
    and $c, d> 0$. Denote
    \[
        \begin{aligned}
        \tilde G_{A, B} & \defeq \sup_{\norm{v}=1} \norm{Av}+\norm{Bv}-2\norm{v}, \\ %\quad\text{and}\\
        \tilde J_{c, d} & \defeq \abs{c+d-2}, \quad\text{and}\\
        \tilde D_A & \defeq \norm{A-I}.
        \end{aligned}
    \]
    If $\tilde J_{c, d}, \tilde D_A, \tilde D_B \le M$ for some constant $M \in (0, 1)$,
    then there exists a constant $C=C(M, K_\psi, \psi^\infty)$ such that
    \begin{equation}
        \label{eq:huber-biapprox}
        c\psi(\norm{Av}) + d\psi(\norm{Bv}) - 2\psi(\norm{v})
        \le C(\tilde G_{cA,dB} 
            + \tilde J_{c,d} 
            + \tilde D_A^2 + \tilde D_B^2)\norm{v},
    \end{equation}
    for all $v \in \R^m$.
\end{lemma}

\begin{proof}
    Let us first of all observe directly from the definition of 
    the subdifferential, and $\psi$ being increasing that
    \begin{equation}
        \label{eq:psi-subdiff-bnd}
        0 < \sup_{t} \norm{\subdiff \psi(t)} \le \psi^\infty,
    \end{equation}
    From convexity also
    \[
        \psi(s) \le \psi(0) + t\frac{\psi(t/\lambda)}{t/\lambda},
        \quad (\lambda \in (0, 1]). 
    \]
    Letting $\lambda \downto 0$, and recalling $\psi(0)=0$, we deduce
    \[
        \psi(s) \le \psi^\infty s, \quad (s\ge 0).
    \]
    
    Let us define
    \[
        L \defeq c\psi(\norm{Av}) + d\psi(\norm{Bv}) - 2\psi(\norm{v}).
    \]
    We want to bound $L$.
    Using the definition of the subdifferential, we have
    \begin{equation}
        \label{eq:huber-energy-first-approx}
        L  \le c\subdiff \psi(\norm{Av})(\norm{Av}-\norm{v})
            + d\subdiff \psi(\norm{Bv})(\norm{Bv}-\norm{v})
            + (c+d-2) \psi(\norm{v}).
    \end{equation}

    If $\norm{Av}, \norm{Bv} \ge \norm{v}$, we immediately deduce
    using \eqref{eq:psi-subdiff-bnd} that
    \[
        \begin{split}
        L &
        \le \psi^\infty\left(c\norm{Av}+d\norm{Bv}-2\norm{v}\right)
        +
        (c+d-2)\bigl(\psi(\norm{v}) - \psi^\infty\norm{v}\bigr)
        \\
        &
        \le \psi^\infty \bigl( \tilde G_{cA,dB}\norm{v} + \tilde J_{c,d} \norm{v}\bigr).
        \end{split}
    \]
    This is what we need.

    If $\norm{Av}, \norm{Bv} \le \norm{v}$, we deduce
    $L \le \psi^\infty \tilde J_{c,d} \norm{v}$.
    Again the claim holds.

    It remains to consider the case $\norm{Av} \ge \norm{v} > \norm{Bv}$,
    the case with $A$ and $B$ exchanged being analogous. We now
    use \eqref{eq:limited-coco} as follows. We
    pick $z_B \in \subdiff \psi(\norm{Bv})$
    and $z_A \in \subdiff \psi(\norm{Av})$, and define
    \[
        \hat\chi \defeq 1-\chi_{[K_\psi, \infty)}(\norm{Bv})
        = 1-\chi_{[K_\psi, \infty)}(\norm{Av})\chi_{[K_\psi, \infty)}(\norm{Bv}).
    \]
    Then
    \[
        \begin{split}
            z_B (\norm{Bv}-\norm{v})
            &
            =
            \bigl(z_B -z_A\bigr)\bigl(\norm{Bv}-\norm{v}\bigr)
            + z_A \bigl(\norm{Bv}-\norm{v}\bigr)
            \\
            &
            =
            \bigl(z_B -z_A\bigr)\bigl(\norm{Bv}-\norm{Av}\bigr)
            \\
            &
            \phantom{=}
            +
            \bigl(z_B - z_A\bigr)\bigl(\norm{Av}-\norm{v}\bigr)
            + z_A \bigl(\norm{Bv}-\norm{v}\bigr)
            \\
            &
            \le
            \hat\chi C_\psi\bigl(\norm{Bv}-\norm{Av}\bigr)^2
            + z_A \bigl(\norm{Bv}-\norm{v}\bigr).
        \end{split}
    \]
    In the final step we have used \eqref{eq:limited-coco},
    $\norm{Av}-\norm{v}\ge 0$, and the fact that $z_B - z_A \le 0$ 
    which follows from the monotonicity of $\subdiff \psi$.
    Thus, continuing from
    \eqref{eq:huber-energy-first-approx}, we calculate
    \[
        \begin{split}
            L 
            &
            \le
            c z_A (\norm{Av}-\norm{v})
            + d z_B (\norm{Bv}-\norm{v})
            + (c + d - 2)\psi(\norm{v})
            \\
            &
            \le
            cz_A(\norm{Av}-\norm{v})
            +
            dz_A(\norm{Bv}-\norm{v})
            + (c + d - 2)\psi(\norm{v})
            \\ & \phantom{=}
            +
            d\hat\chi C_\psi\bigl(\norm{Bv}-\norm{Av}\bigr)^2
            \\
            &
            \le
            z_A (\norm{cAv}+\norm{dBv}-2\norm{v})
            +
            (c + d - 2)\bigl(\psi(\norm{v})-z_A\norm{v}\bigr)
            \\ & \phantom{=}
            +
            %C\norm{(B-A)v}^2.
            2d\hat\chi C_\psi\bigl((\norm{Av}-\norm{v})^2+(\norm{Bv}-\norm{v})^2\bigr)
            \\
            &
            \le
            \psi^\infty\bigl(\tilde G_{cA,dB} + \tilde J_{c,d}\bigr)\norm{v}
            +
            %\hat\chi C_\psi\bigl(T_{A,I}^2+T_{B,I}^2\bigr)\norm{v}^2.
            2(2+M)\hat\chi C_\psi (\tilde D_A^2 + \tilde D_B^2) \norm{v}^2.
        \end{split}
    \]
    In the final step, we have used $d \le 2+M$, which follows from the bound $\tilde J_{c,d} \le M$ and $c>0$.
    Finally, we observe from
    \[
        \norm{Bv} \ge \norm{v} - \norm{Bv-v} \ge (1-M)\norm{v}
    \]
    that
    \[
        \hat \chi \norm{v}
        \le
        \frac{1}{1-M}\hat \chi \norm{Bv}
        \le
        \frac{K_\psi}{1-M}.
    \]
    %for some constant $C=C(M, K_\psi)$.
    This proves the claim.
\end{proof}

\begin{remark}
    We only needed \eqref{eq:limited-coco} and the bound $M$ on
    $\tilde D_A$ and $\tilde D_B$ for the estimate 
    $\norm{v}^2 \le C \norm{v}$ below $K_\psi$.
    The bound $M$ on $\tilde J_{c,d}$ was used to get rid of $d$ in the final steps.
\end{remark}

\begin{proof}[Proof of Proposition \ref{prop:huber-admis}]
    Weak* lower semicontinuity is immediate from the formulation \eqref{eq:psi-tv-def}.
    To see \eqref{eq:r-tv-bound}, we observe for some constants $c,r>0$ that
    \[
        \psi^*(t) \le c + \delta_{[0, r)}(t).
    \]
    Indeed, for $t,s > K_\psi$, \eqref{eq:limited-coco} implies
    $\subdiff \psi(s)=\subdiff \psi(t)=\{r\}$ for some constant $r>0$.
    Approximating
    \[
        \psi^*(t)=\sup_{s \ge 0}(st - \psi(s))
                \le \sup_{s \ge 0}(st - r(s-2K_\psi) - \psi(2K_\psi))
    \]
    and using $\psi(2K_\psi)<\infty$ now gives the above bound.
    Let us then pick $\varphi \in C_c^\infty(\Omega; \R^m)$
    with $\norm{\phi}_{2,L^\infty(\Omega)} \le r$. We estimate
    \[
        \int_\Omega u \divergence \varphi \d x - \int_\Omega \psi^*(\norm{\varphi(x)}) \d x
        \ge
        \int_\Omega u \divergence \varphi \d x - c\L^m(\Omega).
    \]
    It follows
    \[
        \inv r \TV_\psi(\Omega) \ge \abs{Du}(\Omega) - c\L^m(\Omega).
    \]
    This shows \eqref{eq:r-tv-bound} and that the assumptions of
    Definition \ref{def:admissible} hold.
    
    The proof of the assumptions of Definition \ref{def:lipschitz-trans} 
    is analogous to Proposition \ref{prop:tv-admissible}. 
    Let $\gamma \in \LipClass(\Omega, U)$ for some open $U \subset \Omega$.
    The singular part $\abs{D^s \pf\gamma u}(\Omega)$ is unaltered
    in \eqref{eq:tv-sep}. In place of 
    the absolutely continuous part $\abs{D^a \pf\gamma u}(\Omega)$, 
    we have $\int_\Omega \psi(\norm{\grad \pf\gamma u}) \d x$.
    Using the area formula, we calculate
    \[
        \int_\Omega \psi(\norm{\grad \pf\gamma u(x)}) \d x
        =
        \int_U \psi(\norm{\grad \inv \gamma(\gamma(x)) \grad u(x)}) \jacobianf{m}{\gamma}{x} \d x.
    \]
    Let now $\gammaone,\gammatwo \in \LipClass(\Omega, U)$
    with $\bitransfull{\gammaone}{\gammatwo} < 1$.
    Summing the previous equation for $\gamma=\gammaone,\gammatwo$,
    subtracting $2\int_\Omega \psi(\norm{\grad u(x)}) \d x$ and
    using Lemma \ref{lemma:huber-biapprox} with
    $A=\grad \inv \gammaone(\gamma(x))$,
    $B=\grad \inv \gammatwo(\gamma(x))$,
    $c=\jacobianf{m}{\gammaone}{x}$,
    $d=\jacobianf{m}{\gammatwo}{x}$,
    and $v=\grad u(x)$ yields
    \[
        \begin{split}
        \int_\Omega \psi(\norm{\grad \pf\gammaone u}) \d x
        &
        +
        \int_\Omega \psi(\norm{\grad \pf\gammatwo u}) \d x
        - 2 \int_\Omega \psi(\norm{\grad u(x)}) \d x
        \\
        &
        \le
        C(\bitrans{\gammaone}{\gammatwo} + \bitransjac{\gammaone}{\gammatwo} + \didtrans{\gammaone}^2 + \didtrans{\gammatwo}^2)
        \int_U \norm{\grad u(x)} \d x.
        \end{split}
    \]
    The various elements in the sum on the right hand side are defined in \eqref{eq:bitrans-components}.
    With this estimate at hand, the rest follows exactly as in Proposition \ref{prop:tv-admissible}.
    
    Finally, to see, $\norm{u}_{L^\infty(\Omega)} \le \norm{f}_{L^\infty(\Omega)}$
    for solutions $u$ to \eqref{eq:prob}, we follow the corresponding proof
    for $\TV$. Namely, if we set $L \defeq \norm{f}_{L^\infty(\Omega)}$,
    and replace $u$ by
    \[
        \bar u(x) \defeq \max\{-L, \min\{u(x), L\}\},
    \]
    then it follows from the chain rule in BV
    \cite[Theorem 3.96]{ambrosio2000fbv} that
    \[
        \abs{D\bar u}(A) \le \abs{Du}(A)
        \quad
        \text{for any Borel set } A.
    \]
    From this it is immediate that
    \[
        \TV_\psi(\bar u) \le \TV_\psi(u),
    \]
    Moreover, if $\bar u \ne u$, i.e.,
    $\norm{\bar u}_{L^\infty(\Omega)} < \norm{u}_{L^\infty(\Omega)}$,
    we deduce
    \[
        \int_\Omega \phi(\bar u(x)-f(x)) \d x < \int_\Omega \phi(u(x)-f(x)) \d x.
    \]
    This provides a contradiction. Thus we must have
    $\norm{u}_{L^\infty(\Omega)} \le \norm{f}_{L^\infty(\Omega)}$.
\end{proof}

%\begin{remark}
%    The  nonlinearity of $\psi$ above is the only reason we 
%    require the double-Lipschitz Jacobian comparison constant
%    $\bitransjac{\gammaone}{\gammatwo}$ in this paper.
%\end{remark}

\subsection{Remarks on ill-posed non-convex regularisers}
\label{sec:illposed}

We now provide a few remarks on using the technique on
popular models involving non-convex energies $\psi$. 
Problems of the form \eqref{eq:prob} with regularisation
functionals $\TV_\psi$ employing the energies constructed below
do not, however, in general have solutions in $\BVspace(\Omega)$.
Some remedies exist \cite{tuomov-tvq} which make these models still
worth considering.

\begin{remark}[Non-convex total variation]
    \label{remark:tvq}
    Regularisation functionals $\TV_\psi$ based on concave energies $\psi$ 
    have recently received increased attention, for the better modelling
    of real image gradient statistics \cite{huang1999statistics,HiWu13_siims,HiWu14_coap,ochsiterated,tuomov-tvq}.
    Let us define
    \[
        \psi^0 \defeq \lim_{t \downto 0} \psi(t)/t,
    \]
    and suppose $\psi: [0, \infty) \to [0, \infty)$ is concave, increasing,
    $\psi(0)=0$ and $\psi^0, \psi^\infty \in (0, \infty)$.
    The upper bound on $\psi^0$ forbids the typical energies $\psi(t)=t^q$
    for $q \in (0, 1)$, but allows slightly modified ones, linearised
    for small values.
    The above proof can now be extended to show that $\TV_\psi$ 
    is still (separably) double-Lipschitz comparable. Indeed,
    Lemma \ref{lemma:huber-biapprox} is trivial to prove in this case.
    Assuming for simplicity that $c+d=2$, by concavity
    \[
        \begin{split}
        c\psi(\norm{Av})+d\psi(\norm{Bv})-2\psi(\norm{v})
        &
        \le c\subdiff \psi(\norm{v})(\norm{Av}-\norm{v})
            + d\subdiff \psi(\norm{v})(\norm{Bv}-\norm{v})
        \\
        &
        \le \psi^0\max\{0, \norm{cAv}+\norm{dBv}-2\norm{v}\}
        \\
        &
        \le \psi^0 \tilde G_{A,B} \norm{v}.
        \end{split}
    \]
    The assumption $c+d=2$ is also not difficult to remove
    with the help of $\tilde J_{c,d}$, and is satisfied by
    the Lipschitz transformation we construct in 
    \S\ref{sec:lipschitz}.
    
    Unfortunately, the $\TV_\psi$ regularisation functional
    constructed with concave $\psi$ is not admissible. 
    It lacks weak* lower semicontinuity.
    Using area strict convergence \cite{rindler2013strictly}
    and additional multiscale regularisation, first introduced 
    in \cite{tuomov-bd} for discontinuous optical flow,
    problems involving such regularisation can however 
    be made well-posed \cite{tuomov-tvq}.
    It is outside the scope of the present paper to prove
    that the multiscale regularisation term $\eta$ is separably
    double-Lipschitz comparable.
    Nevertheless, assuming we had a solution $u \in \BVspace(\Omega)$
    to the basic model, or a remedied model still satisfying the
    double-Lipschitz comparability condition, then our technique
    would show $\H^{m-1}(J_u \setminus J_f)=0$.
\end{remark}

\begin{remark}[Perona-Malik]
    The Perona-Malik anisotropic diffusion \cite{perona1990scalespace}
    may also be written in the variational form \eqref{eq:prob}
    with $\phi(t)=t^2$ and $R(u)=\TV_\psi(u)$ for
    $\psi(t)=\log(1+t^2)$. Observe that this $\psi$ is not even concave, 
    unlike the models of the previous remark. The model
    suffers from exactly the same ill-posedness problems;
    for a review on approaches to make the problem well-posed,
    we refer to \cite{guidotti2014anisotropic}.
    Nevertheless, our approach can be adapted to study it,
    assuming we have a solution $u \in \BVUspace$ to the problem. 
    The existence is however generally not guaranteed.
    Indeed, with the notation of Lemma \ref{lemma:huber-biapprox},
    by properties of the logarithm, we have
    \begin{equation}
        \label{eq:pm-energy-1}
        c\psi(\norm{Av})
        =c\log(1+\norm{Av}^2)
        =c\log(c+c\norm{Av}^2) - c \log c.
    \end{equation}
    By the concavity of the logarithm, we have
    \[
        \begin{split}
        c\log(c+c\norm{Av}^2) - c\log(1+\norm{v}^2)
        &
        \le \frac{c}{1+\norm{v}^2}(c\norm{Av}^2-\norm{v}^2 + c -1)
        \\
        &
        = \frac{1}{1+\norm{v}^2}\bigl(\norm{cAv}^2-\norm{v}^2 + (c-1)(\norm{v}^2+1)\bigr).
        \end{split}
    \]
    If $c+d=2$, summing with the corresponding estimate for $B$ and $d$, we have
    \[
        c\log(c+c\norm{Av}^2) + d\log(d+d\norm{Bv}^2) - 2\log(1+\norm{v}^2)
        \le \frac{1}{1+\norm{v}^2} G'_{cA,dB}\norm{v}^2,
    \]
    where
    \[
        G'_{cA,dB}=\norm{c^2A^*A+d^2B^*B-2I}.
    \]
    With $c+d=2$, we also have
    \[
        - c \log c - d \log d \le 0.
    \]
    Therefore, using $1+\norm{v}^2 \ge \norm{v}$
    and referring back to \eqref{eq:pm-energy-1}, we obtain
    \[
        c\psi(\norm{Av})
        +d\psi(\norm{Bv})
        -2\psi(\norm{v})
        \le  G'_{cA,dB}\norm{v}.
    \]
    
    This is not exactly the estimate of Lemma \ref{lemma:huber-biapprox},
    but in \S\ref{sec:lipschitz}, we will in fact estimate
    $\tilde G_{cA,dB}$ through $G'_{cA,dB}$; 
    see Lemma \ref{lemma:biestim}. Therefore
    we may follow this reasoning with the argument of 
    Proposition \ref{prop:huber-admis}.
    We also assumed $c+d=2$, but it would not be difficult to remove this
    assumption with the help of $\tilde J_{c,d}$. The specific
    Lipschitz transformations that we construct next in 
    \S\ref{sec:lipschitz} however do satisfy $c+d=2$,
    that is, $\jacobianf{m}{\gammaone}{x}+\jacobianf{m}{\gammatwo}{x}=2$.
\end{remark}

%%%
\section{Lipschitz shift transformations}
\label{sec:lipschitz}
%%%

We now introduce the ``shift'' class of Lipschitz transformations
that we will use to push forward a purported solution
$u$ to \eqref{eq:prob} that does not satisfy
$\H^{m-1}(J_u \setminus J_f)=0$. The idea is to take
a Lipschitz graph $\Gamma$ on $J_u$ containing a violating
point $x_0$, and then to move the jump forward or
backward by an amount $\rho$, in order to construct
a better candidate solution.
It is interesting to relate the specific constructions here to the
general framework for designing transformations satisfying a 
PDE on the Jacobian determinant in \cite{dacorogna1990partial}.
Before the construction, we provide a simple general lemma stating
one  of the most crucial parts of our technique.
So far, we have not made significant use of the fact that our finite-dimensional
norm on $\R^m$ is the Euclidean norm, but here it is essential.

%%%
\subsection{A crucial estimate}
%%%

We will use the following simple lemma to simplify the estimation of the
double-Lipschitz comparison factor $\bitrans{\gammaone}{\gammatwo}$.

\begin{lemma}
    \label{lemma:biestim}
    Let $\gammaone, \gammatwo \in \LipClass(\Omega)$.
    For $\bitrans{\gammaone}{\gammatwo}$ defined in \eqref{eq:bitrans}, and
    $\lipjac{\gamma}$ defined in \eqref{eq:lipjac}, we have
    \[
        \bitrans{\gammaone}{\gammatwo} 
        \le
        \sup_{x \in \Omega} \frac{1}{2}\norm{[\lipjac{\gammaone}(x)]^* \lipjac{\gammaone}(x) + [\lipjac{\gammatwo}(x)]^* \lipjac{\gammatwo}(x) - 2I}.
    \]
\end{lemma}
\begin{proof}
    By the concavity of the square root, for any $0 \ne v \in \R^m$, we have
    \[
        \norm{\lipjac{\gamma}(x) v} - \norm{v} \le \frac{1}{2\norm{v}} 
            \bigl(
                \norm{\lipjac{\gamma}(x) v}^2 - \norm{v}^2
            \bigr)
        =
        \frac{1}{2\norm{v}} 
            \iprod{v}{([\lipjac{\gamma}(x)]^* \lipjac{\gamma}(x) -I) v}.
    \]
    Therefore
    \begin{equation}
        \notag
        \begin{split}
        \norm{\lipjac{\gammaone}(x) v} + \norm{\lipjac{\gammatwo}(x) v} - 2\norm{v}
        &
        \le
        \frac{1}{2\norm{v}} 
            \iprod{v}{([\lipjac{\gammaone}(x)]^* \lipjac{\gammaone}(x) + [\lipjac{\gammatwo}(x)]^* \lipjac{\gammatwo}(x) - 2I) v}
        \\
        &
        \le
        \frac{1}{2}
            \norm{[\lipjac{\gammaone}(x)]^* \lipjac{\gammaone}(x) + [\lipjac{\gammatwo}(x)]^* \lipjac{\gammatwo}(x) - 2I} \norm{v}.
        \end{split}
    \end{equation}
    The lemma is proved.
\end{proof}

%%%
\subsection{The construction}
%%%

%\newcommand{\vp}[1]{P_{z_\Gamma}^\perp #1}
%\newcommand{\tp}[1]{P_{z_\Gamma} #1}
\newcommand{\vp}[1]{\mathfrak{v}_{#1}}
\newcommand{\tp}[1]{\mathfrak{t}_{#1}}
\def\rr{s}

\begin{figure}
    \centering
    %\asyinclude{asy/gamma.asy}
    \includegraphics{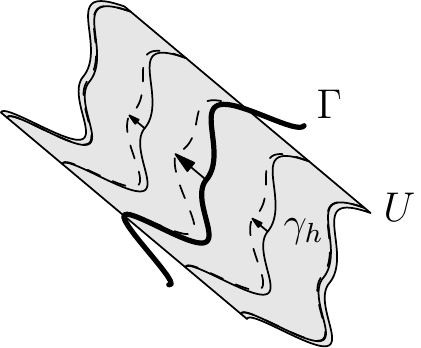}
    \caption{Illustration of the transformation $\gamma_h$ constructed in Lemma \ref{lemma:h-shift-gamma}. The solid lines are transformed into the dotted lines, with larger movement deep in the interior of $U$ (gray region), and none on the boundary.}
    \label{fig:h-shift-gamma}
\end{figure}

For brevity, we now define $\vp{x} \defeq P_{z_\Gamma}^\perp x$
and $\tp{x} \defeq \iprod{z_\Gamma}{x}$ to be the components
of $x$ on $V_\Gamma$ and along $z_\Gamma$, for a Lipschitz graph
$\Gamma$. The transformation $\gamma_h$ constructed in the next lemma
is illustrated in Figure \ref{fig:h-shift-gamma}.

\begin{lemma}
    \label{lemma:h-shift-gamma}
    Let $\Gamma \subset \Omega$ be a Lipschitz $(m-1)$-graph
    and $\rr>0$ be such that
    \[
        U \defeq \Gamma + (-\rr, \rr) z_\Gamma \subset \Omega.
    \]
    %Suppose that $h \in \basehspace$ satisfies
    Suppose that $h \in W_0^{1,\infty}(V_\Gamma)$ satisfies
    $-\rr/3 \le h \le \rr/3$.
    Define the transformation $\gamma_h: U \to U$ by
    \begin{equation}
        \label{eq:gamma-h}
        \gamma_h(x) \defeq \vp{x} + \tilde \gamma_{s,h}(\tp{x}; \vp{x})z_\Gamma
    \end{equation}
    where
    \[
        \tilde \gamma_{s,h}(t; v) \defeq
        \begin{cases}
            %t + \frac{t-t_0+\rr}{f_\Gamma(v)-t_0+\rr}h(v), & t_0-\rr < t < f_\Gamma(v), \\
            %t + \frac{t-t_0-\rr}{f_\Gamma(v)-t_0-\rr}h(v), & f_\Gamma(v) \le t < t_0+\rr. \\
            t + \frac{\rr-f_\Gamma(v)+t}{\rr}h(v), & f_\Gamma(v)-\rr < t < f_\Gamma(v), \\
            t + \frac{\rr+f_\Gamma(v)-t}{\rr}h(v), & f_\Gamma(v) \le t < f_\Gamma(v)+\rr. \\
        \end{cases}
    \]
    Then $\gamma_h$ is 1-to-1 and Lipschitz 
    and
    \[
        \jacobianf{m}{\gamma_h}{x} =
        \abs{ \tilde \gamma_{s,h}'(\tp{x}; \vp{x})}
        =
        \begin{cases}
            %1 + \frac{h(\vp{x})}{f_\Gamma(\vp{x})-t_0+\rr}, & t_0-\rr < \tp{x} < f_\Gamma(\vp{x}), \\
            %1 + \frac{h(\vp{x})}{f_\Gamma(\vp{x})-t_0-\rr}, & f_\Gamma(\vp{x}) \le \tp{x} < t_0+\rr. \\
            1 + \frac{h(\vp{x})}{\rr}, & f_\Gamma(\vp{x})-\rr < \tp{x} < f_\Gamma(\vp{x}), \\
            1 - \frac{h(\vp{x})}{\rr}, & f_\Gamma(\vp{x}) \le \tp{x} < f_\Gamma(\vp{x})+\rr. \\
        \end{cases}
    \]
    There also exists a constant $C=C(h, f_\Gamma)>0$ such that
    for $\rho \in (-1, 1)$, we have for the various double-Lipschitz comparison constants defined in \eqref{eq:bitrans-components} the estimates
    \begin{equation}
        \label{eq:shift-trans-bounds}
        \begin{aligned}
        \bitransfull{\gamma_{\rho h}}{\gamma_{-\rho h}} & \le C \rho^2, &
        \bitrans{\gamma_{\rho h}}{\gamma_{-\rho h}} & \le C \rho^2, &
        \bitransjac{\gamma_{\rho h}}{\gamma_{-\rho h}} & = 0, &
        \didtrans{\gamma_{\rho h}} & \le C \abs{\rho}, \\
        \bitransfull{\gamma_{\rho h}}{\iota} & \le C \abs{\rho}, &
        \bitrans{\gamma_{\rho h}}{\iota} & \le C \abs{\rho}, &
        \bitransjac{\gamma_{\rho h}}{\iota} & \le C \abs{\rho}, \\
        \bitransfull{\inv\gamma_{\rho h}}{\iota} & \le C \abs{\rho}, &
        \bitrans{\inv\gamma_{\rho h}}{\iota} & \le C \abs{\rho}, &
        \bitransjac{\inv\gamma_{\rho h}}{\iota} & \le C \abs{\rho}, &
        \didtrans{{\inv \gamma_{\rho h}}} & \le C \abs{\rho}.
        \end{aligned}
    \end{equation}
    %\[
    %    \norm{\grad \gamma(x)}
    %    \le
    %    \sqrt{
    %        \frac{
    %                1
    %                +[\gamma_{\vp{x}}'(\tp{x})]^2
    %                +\norm{\grad_v \gamma_{\vp{x}}(\tp{x})}^2
    %                +\abs{[\gamma_{\vp{x}}'(\tp{x})]^2
    %                    -1
    %                    -\norm{\grad_v \gamma_{\vp{x}}(\tp{x})}^2
    %                }
    %            }{2}.
    %    }
    %\]
    %\TODO{More useful expressions?}
\end{lemma}

\begin{remark}
    It does not hold that $\inv \gamma_h=\gamma_{-h}$, although it would
    be possible to construct such a transformation satisfying the same
    essential properties. We could then simply plug it into our proofs 
    instead of the above one. We do not however do this, since the 
    bounds \eqref{eq:shift-trans-bounds} are in that case a bit more
    work to prove, and having such a transformation would 
    only help very little with Lemma \ref{lemma:curvp-d-pf}.
\end{remark}

\begin{proof}
    After rotation and translation, if necessary, we may assume w.log 
    that $z_\Gamma=(0, \ldots, 0, 1)$, so that $x=(v,t) \defeq (\vp{x}, \tp{x})$.
    Then
    \[
        \tilde\gamma_{s,h}'(t; v) = d_h \defeq
        \begin{cases}
            1 + \frac{h(v)}{\rr}, & f_\Gamma(v)-\rr < t < f_\Gamma(v), \\
            1 - \frac{h(v)}{\rr}, & f_\Gamma(v) \le t < f_\Gamma(v) + \rr. \\
        \end{cases}
    \]
    and
    \[
        \grad_v \tilde\gamma_{s,h}(t; v) = c_h \defeq
        \begin{cases}
            \frac{\rr-f_\Gamma(v)+t}{\rr}\grad h(v)
            -
            \frac{\grad f_\Gamma(v)}{\rr}h(v),
            & f_\Gamma(v)-\rr < t < f_\Gamma(v), \\
            \frac{\rr+f_\Gamma(v)-t}{\rr}\grad h(v)
            +
            \frac{\grad f_\Gamma(v)}{\rr}h(v),
            & f_\Gamma(v) \le t < f_\Gamma(v) + \rr. \\
        \end{cases}
    \]
    Observe that $\abs{\tilde\gamma_{s,h}'(t; v)}=\tilde\gamma_{s,h}'(t; v)$ due to the bounds
    $-\rr/3 \le h \le \rr/3$.%, and \eqref{eq:f-transf-bound}.
    
    To calculate the Lipschitz factor and the Jacobian determinant
    \[
        \jacobianf{m}{\gamma_h}{x} =
        \sqrt{\det(\grad \gamma_h(x)[\grad \gamma_h(x)]^T)},
    \]
    we study the eigenvalues $\lambda_1,\ldots,\lambda_m$ of
    \[
        \grad \gamma_h(x)[\grad \gamma_h(x)]^T = 
        \begin{pmatrix}
            I & c_h \\
            0 & d_h \\
        \end{pmatrix}
        \begin{pmatrix}
            I & 0 \\
            c_h^T & d_h
        \end{pmatrix}
        =
        \begin{pmatrix}
            I + c_h \otimes c_h & d_h c_h \\
            d_h c_h^T & d_h^2
        \end{pmatrix}.
    \]
    Easily we see that $\lambda_3=\cdots=\lambda_{m}=1$,
    with the corresponding eigenvector orthogonal to $c_h$.
    For the two remaining, important, eigenvalues, setting
    $y=(c_h, \alpha)$ for the unknown eigenvector, we obtain 
    the system of equations
    \[
        1 + \norm{c_h}^2 + \alpha d_h = \lambda
        \quad
        \text{and}
        \quad
        d_h \norm{c_h}^2 + \alpha d_h^2 = \lambda \alpha.
    \]
    Solving this system of equations, we obtain the solutions
    \[
        \lambda_1(x), \lambda_2(x)=\frac{1+d_h^2+\norm{c_h}^2 
            \pm \sqrt{(1+d_h^2+\norm{c_h}^2)^2-4d_h^2}
            }{2}.
    \]
    %We have
    %\[
    %    \lambda_1(x) \lambda_2(x)
    %    =\frac{(1+d_h^2+\norm{c_h}^2)^2 -
    %        \bigl((1+d_h^2+\norm{c_h}^2)^2-4d_h^2\bigr)
    %        }{4}
    %    =d_h^2.
    %\]
    This gives as claimed
    \[
        \jacobianf{m}{\gamma_h}{x} = \sqrt{\prod_{i=1}^m \lambda_i}
        =\sqrt{\lambda_1(x)\lambda_2(x)}=\abs{d_h}=\tilde\gamma_{s,h}'(t; v).
    \]

    It remains to consider the bounds \eqref{eq:shift-trans-bounds}.
    Letting $\gammaone \defeq \gamma_{\rho h}$ and
    $\gammatwo \defeq \gamma_{-\rho h}$, and recalling that
    \[
        \lipjac{\gamma}(x) \defeq \grad \inv \gamma(\gamma(x)) \jacobianf{m}{\gamma}{x},
    \]
    by Lemma \ref{lemma:biestim} we have
    \[
        \bitrans{\gammaone}{\gammatwo}
        \le
        %\sup_{x \in \Omega} \norm{\jacobian_m(\grad \gammaone(x)) \grad \inv \gammaone(\gammaone(x))
        %    +\jacobian_m(\grad \gammatwo(x)) \grad \inv \gammatwo(\gammatwo(x)) - 2 I}.
        \sup_{x \in \Omega} \frac{1}{2} \norm{[\lipjac{\gammaone}(x)]^* \lipjac{\gammaone}(x) + [\lipjac{\gammatwo}(x)]^* \lipjac{\gammatwo}(x) - 2I}.
    \]
    We want to estimate this further.
    Recalling the proof of Lemma \ref{lemma:da-gammau},
    $\grad \inv\gamma_h(\gamma_h(x))=\inv{[\grad \gamma_h(x)]}$ 
    holds (a.e.). With
    \[
        \grad \gamma_h(x)
        =
        \begin{pmatrix}
            I & 0 \\
            c_{h}^T & d_{h} \\
        \end{pmatrix},
    \]
    it can easily be verified that
    \begin{equation}
        \label{eq:inv-grad-gamma-lineartrans}
        \inv{[\grad \gamma_h(x)]}
        =
        \begin{pmatrix}
            I & 0 \\
            -c_{h}^T \inv d_{h} & \inv d_{h} \\
        \end{pmatrix}.
    \end{equation}
    It follows that
    \[
        \lipjac{\gamma_h}(x) 
        =
        \begin{pmatrix}
            d_{h}I & 0 \\
            -c_{h}^T & 1 \\
        \end{pmatrix},
    \]
    and
    \[
        [\lipjac{\gamma_h}(x)]^*\lipjac{\gamma_h}(x) 
        =
        \begin{pmatrix}
            d_{h}^2 I + c_{h}^T c_{h} & -c_{h} \\
            -c_{h}^T & 1 \\
        \end{pmatrix}.
    \]
    We observe that
    \begin{align}
        \label{eq:c-gammaone}
        c_{\pm\rho h} & = \pm \rho c_h, \quad\text{and}\\
        \label{eq:d-gammaone}
        d_{\pm \rho h}^2 & =1\pm 2 \rho d_h + \rho^2 d_h^2, 
    \end{align}
    Thus $d_{\rho h}^2+d_{-\rho h}^2-2=2\rho^2d_h^2$,
    and we deduce
    \[
        [\lipjac{\gammaone}(x)]^*\lipjac{\gammaone}(x) 
        +
        [\lipjac{\gammatwo}(x)]^*\lipjac{\gammatwo}(x) 
        -2I
        =
        2\rho^2
        \begin{pmatrix}
            d_h^2 I + c_h^T c_h & 0 \\
            0 & 0 \\
        \end{pmatrix}.
    \]
    %Using the fact that $-s/3 \le h \le s/3$. 
    It follows
    \[
        \bitrans{\gammaone}{\gammatwo}
        \le C \rho^2.
    \]
    Moreover, \eqref{eq:d-gammaone} also yields
    \[
        \bitransjac{\gamma_{\rho h}}{\gamma_{-\rho h}}=d_{\rho h}+d_{-\rho h}-2 =0.
    \]
    
    In order to estimate
    \[
        \didtrans{\gamma_{\rho h}} = \sup_{x \in \Omega} \norm{\grad\inv\gamma_{\rho h}\gamma_{\rho h}(x)-I}, 
    \]
    we calculate using \eqref{eq:inv-grad-gamma-lineartrans} that
    \[
        \inv{[\grad \gamma_{\rho h}(x)]}-I
        =
        \inv d_{\rho h}
        \begin{pmatrix}
            0 & 0 \\
            -c_{\rho h}^T & 1-d_{\rho h}.
        \end{pmatrix}
    \]
    By the bounds on $r$, $f_\Gamma$ and $h$,
    we have $\abs{d_{\rho h}} \ge 2/3$.
    Therefore, it follows  from
    \eqref{eq:c-gammaone}, \eqref{eq:d-gammaone} that
    \[
        \didtrans{\gamma} 
        = 
        \sup_{x\in\Omega} \norm{\grad \inv \gamma_{\rho h}(\gamma_{\rho h}(x))-I}
        =
        \sup_{x\in\Omega} \norm{\inv{[\grad \gamma_{\rho h}(x)]}-I}
        \le C \rho,
        \quad (0 < \rho < 1).
    \]
    This proves all the bounds in \eqref{eq:shift-trans-bounds}
    involving both $\gamma_{\rho h}$ and $\gamma_{\rho h}$.
    The bounds involving $\iota$ and $\inv \gamma_{\rho h}$
    are proved analogously. For the $\inv \gamma_{\rho h}$
    bounds we use the fact at $y=\gamma_h(x)$ we have
    \[
        \jacobianf{m}{\inv \gamma_h}{y}=1/\jacobianf{m}{\gamma_h}{x}=1/d_h,
    \]
    as well as $\grad \gamma_h(\inv\gamma_h(y))=\grad \gamma_h(x)$,
    so that
    \[
        \lipjac{\inv \gamma_h}
        =
        \begin{pmatrix}
            \inv d_h I & 0 \\
            -c_{h}^T \inv d_{h} & 1 \\
        \end{pmatrix}.
    \]
    We skip the elementary details.
\end{proof}

We will also require an estimate of the following type.

\begin{lemma}
    \label{lemma:f-transform-estim}
    Let the Lipschitz transformation $\gamma_h: \Omega \to \Omega$ 
    have the form \eqref{eq:gamma-h}, and be identity outside
    $U \subset \Omega$. Let $u \in \BVspace(\Omega)$. Define
    \[
        M_{\gamma} \defeq \sup_{x \in U} \norm{\gamma_h(x)-x}.
    \]
    Then $M_\gamma = \norm{h(v)}_{L^\infty(V_\Gamma)}$ and
    \[
        \int_U \abs{u(\gamma_h(x))-u(x)} \d x 
        \le M_{\gamma_h} \abs{D u}(U).
    \]
\end{lemma}

\begin{proof}
    Let us again, without loss of generality after rotation and translation, if necessary, 
    that $z_\Gamma=(0, \ldots, 0, 1)$, so that $x=(v,t) \defeq (\vp{x}, \tp{x})$.
    Then 
    \[
        \gamma_h(v, t)=(v, \tilde \gamma_v(t)),
    \]
    and the slice
    \[
        u_v(t) \defeq u(v, t), \quad (-s < t < s; v \in V_\Gamma)
    \]
    satisfies $u_v \in \BVspace(-s, s)$ for $\L^{m-1}$-\ae $v \in V_\Gamma$.
    Thus
    \[
        \abs{u_v(\tilde \gamma(t; v))-u_v(t)} \le \abs{Du_v}([t, {\tilde \gamma(t; v)}])
        =\int_{-s}^{s} \chi_{[t, {\tilde \gamma(t; v)}]}(\tau) \d\abs{Du_v}(\tau).
    \]
    Here we use the convention $[a, b] \defeq [b, a]$ if $b<a$.
    Observe that $[t, \tilde \gamma(t; v)] \subset [t, t+M_{\gamma_h}]$ 
    or $[t, \tilde \gamma(t; v)] \subset [t, t-M_{\gamma_h}]$.
    Using Fubini's theorem and basic properties of one-dimensional
    slices in BV \cite{ambrosio2000fbv}, we may thus estimate
    \[
        \begin{split}
        \int_U \abs{u(\gamma_h(x))-u(x)} \d x 
        &
        =
        \int_{V_\Gamma} \int_{-s}^{s} \abs{u_v(\gamma_v(t))-u_v(t)} \d t \d\H^{m-1}(v)
        \\
        &
        \le
        \int_{V_\Gamma} \int_{-s}^{s} \int_{-s}^{s} \chi_{[t, {\tilde \gamma(t; v)}]}(\tau) \d\abs{Du_v}(\tau) \d t \d\H^{m-1}(v)
        \\
        &
        =
        \int_{V_\Gamma} \int_{-s}^{s} \int_{-s}^{s} \chi_{[t, {\tilde \gamma(t; v)}]}(\tau) \d t \d\abs{Du_v}(\tau) \d\H^{m-1}(v)
        \\
        &
        \le
        \int_{V_\Gamma} \int_{-s}^{s} M_{\gamma_h} \d\abs{Du_v}(\tau) \d\H^{m-1}(v)
        \\
        &
        =
        M_{\gamma_h} \abs{Du}(U).
        \end{split}
    \]
    
    Finally
    \[
        \norm{\gamma_h(x)-x}
        = \abs{h(v)} \cdot
        \begin{cases}
             \frac{\rr-f_\Gamma(v)+t}{\rr}, & f_\Gamma(v)-\rr < t < f_\Gamma(v), \\
            \frac{\rr+f_\Gamma(v)-t}{\rr}, & f_\Gamma(v) \le t < f_\Gamma(v)+\rr. 
        \end{cases}
    \]
    Clearly this gives $M_{\gamma_h} = \norm{h(v)}_{L^\infty(V_\Gamma)}$.
\end{proof}

In the following lemma, based on the construction of
Lemma \ref{lemma:h-shift-gamma}, we now construct our family
of shift transformations parametrised by the size $r>0$
of the neighbourhood of $x_0$ where the transformation is
performed, and the magnitude $\rho>0$ of the transformation.

\begin{lemma}
    \label{lemma:gamma-rho-h-r}
    Let $\Omega \subset \R^m$, and $\Gamma \subset \Omega$ 
    be a Lipschitz $(m-1)$-graph. Pick $x_0 \in \Gamma$ and
    $0 \ne \baseh \in \basehspace$ satisfying $0 \le \baseh \le 1$.
    Set
    \[
        h_r(v) \defeq r \baseh((v-\vp{x_0})/r),
        \quad (r>0).
    \]
    Then there exists $r_0>0$ such that each of the maps
    \[
        \gammax{\rho}{h}{r}(x) \defeq \vp{x} + \tilde \gamma_{3 r,\rho h_r}(\tp{x}; \vp{x}) z_\Gamma,
        \quad (-1 < \rho < 1,\, 0 < r < r_0),
    \]
    satisfies $\gammax{\rho}{h}{r} \in \LipClass(\Omega; U_r)$ for
    \[
        U_r \defeq x_0 + z_\Gamma^\perp \isect \B(0, r) + (3 + \lip f_\Gamma) (-r, r) z_\Gamma.
    \]
    Moreover, there exists a constant $C>0$ such that the transformations
    $\gammax{\rho}{h}{r}$, ($0 < r < r_0$, $-1 < \rho < 1$), satisfy 
    \[
        \begin{aligned}
        \bitransfull{\gammax{\rho}{h}{r}}{\gammax{-\rho}{h}{r}} & \le C \rho^2, 
        &
        \bitransfull{\gammax{\rho}{h}{r}}{\iota} & \le C \abs{\rho}, 
        \\
        M_{\gammax{\rho}{h}{r}} & = \abs{\rho} r, &
        \bitransfull{\inv{\gammax{\rho}{h}{r}}}{\iota} & \le C \abs{\rho}
        \end{aligned}
    \]
    for the various double-Lipschitz comparison constants defined in \eqref{eq:bitrans-components}.
\end{lemma}

\begin{remark}
    We can, for example, set $\baseh(v) \defeq \max\{0, 1-\norm{v}\}$.
\end{remark}

\begin{proof}
    We choose $r_0>0$ small enough that $U_r \subset \Omega$.
    Clearly $\gammax{\rho}{h}{r}$ and $\inv{\gammax{\rho}{h}{r}}$
    satisfy the Lusin $N$-property, and are Lipschitz mappings.
    Since $0 \le \baseh \le 1$, evidently also
    $M_\gamma = \abs{\rho} r$, and $\gamma$ reduces
    to the identity outside $U_r$. The remaining claims follow
    by applying Lemma \ref{lemma:h-shift-gamma} on
    \[
        \Gamma' \defeq g_\Gamma(V_\Gamma \isect \B(\vp{x_0}, r)) \subset \Gamma
    \]
    with $h=h_r$ and $s=3r$.
\end{proof}

%%%    
\section{Proof of the main result for $p>1$}
\label{sec:main}
%%%

We now begin the proof of Theorem \ref{theorem:jumpset-strict}.
Most of the proof consists of producing for the fidelity term
a counterpart of the double-Lipschitz comparability condition
of the regulariser. We do this in 
\S\ref{sec:fidestim} after a technical lemma in \S\ref{sec:tech}.
The fidelity estimates are based on the specific Lipschitz 
transformations constructed in the previous section. Averaging
over two different Lipschitz transformations
$\gammaone=\gammax{\rho}{h}{r}$ and $\gammatwo=\gammax{-\rho}{h}{r}$
will provide an $O(\rho)$ decrease estimate for the fidelity.
Importantly, in order to take advantage of the strict convexity
of $\phi$, we actually need as our improvement candidates 
convex combinations $\theta u + (1-\theta) \pf\gamma u$.
After dealing with the fidelity function, we apply in
\S\ref{sec:regestim} the
double-Lipschitz comparability to get an $O(\rho^2)$ 
increase estimate on the regulariser -- averaged over
the two transformations. After these estimates, the proof
of Theorem \ref{theorem:jumpset-strict} will be immediate.

%%%
\subsection{A technical lemma}
\label{sec:tech}
%%%

We can only perform fidelity estimation 
at a base point $x_0$ outside the set $Z_u$ we construct next. 
For this, given $u \in \BVspace(\Omega)$, we now fix a countably family 
of Lipschitz $(m-1)$-graphs $\{\Gamma_i\}_{i=1}^\infty$, 
such that
\begin{equation}
    \label{eq:ju-lambda-collection}
    \H^{m-1}(J_u \setminus \Union_{i=1}^\infty \Gamma_i)=0.
\end{equation}
Such a family exists by the rectifiability of $J_u$. 
For $\H^{m-1}$-\ae $x_0 \in J_u$, we may then find a Lipschitz graph
$\Gamma=\Gamma^{x_0}$ with $x_0 \in \Gamma^{x_0} \Subset \Gamma_i$ 
for some $i=i(x_0)$, satisfying the following two properties. 
Firstly $V_{\Gamma^{x_0}} \supset \B(P_{z_\Gamma}^\perp x_0, r(x_0))$ 
for some $r(x_0)>0$. This can clearly be satisfied 
since we can by Kirzbraun's theorem assume $V_\Gamma=z_\Gamma^\perp$.
Secondly we can take the traces of $u$ from both sides of
$\Gamma$ to exist at $x_0$ and agree with $u^\pm(x_0)$.
This can be seen from, e.g., the BV trace theorem
\cite[Theorem 3.77]{ambrosio2000fbv}.

\begin{lemma}
    \label{lemma:z-u}
    Let $u \in \BVspace(\Omega)$. Then there exists a Borel set $Z_u$
    with $\H^{m-1}(Z_u)=0$ such that
    every $x \in J_u \setminus Z_u$ is a Lebesgue
    point of the one-sided traces $u^\pm$, and
    \begin{equation}
        \label{eq:theta-du-gammaxi}
        \Theta^*_{m-1}(\abs{Du} \restrict (\Gamma^x)^+; x) = 0,
        \text{ and }
        \Theta^*_{m-1}(\abs{Du} \restrict (\Gamma^x)^-; x) = 0.
    \end{equation}
\end{lemma}
\begin{proof}
    For $\H^{m-1}$-\ae $x \in \Gamma_i$, we have both
    \[
        \Theta_{m-1}(\abs{Du} \restrict \Gamma_i; x)=\abs{u^+(x)-u^-(x)},
    \]
    and
    \[
        \Theta_{m-1}(\abs{Du}; x)=\abs{u^+(x)-u^-(x)},
    \]
    as follows from \cite[Theorem 2.83 \& Theorem 3.77]{ambrosio2000fbv} and, for the latter, a simple application of the former and the Besicovitch derivation theorem.
    Since
    \[
        \Theta^*_{m-1}(\abs{Du} \restrict (\Gamma_i^+ \union \Gamma_i^-); x) \le \Theta_{m-1}(\abs{Du}; x) < \infty
    \]
    we see using the definition of $\Theta^*_{m-1}$ that
    \[
        \Theta^*_{m-1}(\abs{Du} \restrict (\Gamma_i^+ \union \Gamma_i^-); x)
        +
        \Theta_{m-1}(\abs{Du} \restrict \Gamma_i; x)
        =
        \Theta_{m-1}(\abs{Du}; x).
    \]
    Therefore
    \[
        \Theta^*_{m-1}(\abs{Du} \restrict (\Gamma_i^+ \union \Gamma_i^-); x) = 0.
    \]
    This gives \eqref{eq:theta-du-gammaxi} for $\H^{m-1}$-ae $x \in J_u$. Finally $\H^{m-1}$-\ae $x \in J_u$ is a Lebesgue point of the traces $u^\pm$, so clearly the set $Z_u$ satisfying the claims exists.
\end{proof}

\begin{comment}
    We define $A_i$ as the set of points $x \in \Gamma_i$ such that
    \begin{equation}
        \label{eq:z-u-dens1}
        \Theta^*_{m-1}(\abs{Du} \restrict \Gamma_i^+; x) = 0,
        \text{ and }
        \Theta^*_{m-1}(\abs{Du} \restrict \Gamma_i^-; x) = 0.
    \end{equation}
    %The former holds for $\H^{m-1}$-\ae $x \in J_u$ 
    %by a variant of the Besicovitch-Marstrand-Mattila Theorem; 
    %see \cite[Theorem 2.83]{ambrosio2000fbv}.
    As a consequence of the BV extension
    theorem \cite[Proposition 3.21]{ambrosio2000fbv},
    there exist extensions $v_i^\pm$ of $u|\Omega \isect \Gamma_i^\pm$
    over $\Gamma_i$, satisfying $\H^{m-1}(\Gamma_i \isect J_{v_i})=0$.
    From this fact, \eqref{eq:z-u-dens1},
    and Proposition \ref{prop:bv-density}, we deduce
    \[
        \H^{m-1}(\Gamma_i \setminus  A_i)
        =
        \H^{m-1}(\Gamma_i \isect A_i^c)
        \le
        \H^{m-1}(\Gamma_i \isect J_{v_i})=0.
    \]
    Next, we define $B_i$ as the set of points $x \in \Gamma_i$
    such that $x$ is a Lebesgue point of the traces $u^+|\Gamma_i$
    and $u^-|\Gamma_i$. Again $\H^{m-1}(\Gamma_i \setminus B_i)=0$.
    If we now set
    \[
        Z_u \defeq J_u \setminus \Union_{i=1}^\infty (A_i \isect B_i),
    \]
    we deduce that $\H^{m-1}(Z_u)=0$.
\end{comment}

%%%
\subsection{The effect of the shift transformation on the fidelity}
\label{sec:fidestim}
%%%

\newcommand{\newu}[1][\rho]{\bar u_{#1,r}}%{\bar u^{#1}}
\newcommand{\newuz}[1][\rho]{\bar u_{#1,r,0}}%{\bar u^{#1}}
\newcommand{\newur}[1][\rho]{\bar u_{#1}}%{\bar u^{#1}}
\newcommand{\newuzr}[1][\rho]{\bar u_{#1,0}}%{\bar u^{#1}}
\newcommand{\gammarhoh}[1][\rho]{\gammax{#1}{h}{r}}
\newcommand{\gammarhohNEG}{\gammarhoh[-\rho]}

Recalling the definition of $\tilde J_f$ in \eqref{eq:tilde-js-u},
we now fix $x_0 \in J_u \setminus (\tilde J_f \union S_f \union Z_u)$. We also
let $\Gamma=\Gamma^{x_0}$, observing from the construction above
that the traces of $u$ from both sides of $\Gamma$ exist
at $x_0$ and agree with $u^\pm(x_0)$. We also recall the construction 
of Lemma \ref{lemma:gamma-rho-h-r}. We are given a base Lipschitz function
$\baseh \in \basehspace$
with $0 \le \baseh \le 1$. (Sometimes we only assume $-1 \le h \le 1$. 
This is explicitly stated.) We set
\[
    h_r(v) \defeq r h((v-\vp{x_0})/r),
\]
and define the Lipschitz transformations 
$\gammax{\rho}{h}{r}: \Omega \to \Omega$ by
\[
    \gammax{\rho}{h}{r}(x) \defeq \vp{x} + \tilde \gamma_{3r,\rho h_r}(\tp{x}; \vp{x}) z_\Gamma,
    \quad (-1 < \rho < 1,\, 0 < r < r_0).
\]
These transformations are the identity outside
\[
    U_r \defeq x_0 + z_\Gamma^\perp \isect \B(0, r) + (3 + \lip f_\Gamma) (-r, r) z_\Gamma,
\]
which we split into the halves
\[
    U_r^\pm = U_r \isect \Gamma^\pm.
\]
Here we choose $r>0$ small enough that $U_r \subset \Omega$.
We observe and put in our mind for later that
\[
    U_r \subset \B(x_0, Cr)
    \quad\text{for}\quad C=\sqrt{1 + (3+\lip f_\Gamma)^2}.
\]

We pick arbitrary $\theta \in [0, 1]$ as well as $r \in (0, r_0)$ 
and $\rho \in (0, 1)$. With these we define
\[
    \newu(x) \defeq \theta u(x) + (1-\theta) \pf\gammarhoh u(x),
\]
illustrated in Figure \ref{fig:unew}, as well as the piecewise constant functions
\[
    \begin{aligned}
    f_0 & \defeq \tilde f(x_0), \\
    u_0 & \defeq u^+(x_0) \chi_{U_r^+}+ u^-(x_0) \chi_{U_r^-}, \quad \text{and} \\
    \newuz & \defeq \theta u_0(x) + (1-\theta) \pf\gammarhoh u_0(x).
    %\defeq u^+(x_0) \chi_{\gamma(U_r^+)}+ u^-(x_0) \chi_{\gamma(U_r^-)}.
    \end{aligned}
\]
We aim to prove that for suitable $(\theta, \rho, r)$, either
$\newu[\rho]$  or $\newu[-\rho]$ is better than $u$. We do
this by averaging estimates over the two piecewise constant
functions. The interpolation parameter $\theta$ will in
particular be necessary to take advantage of the 
$p$-increase (or strict convexity) of $\phi$ for $p>1$.

\begin{figure}
    \centering
    %\asyinclude{asy/utilde.asy}
    \includegraphics{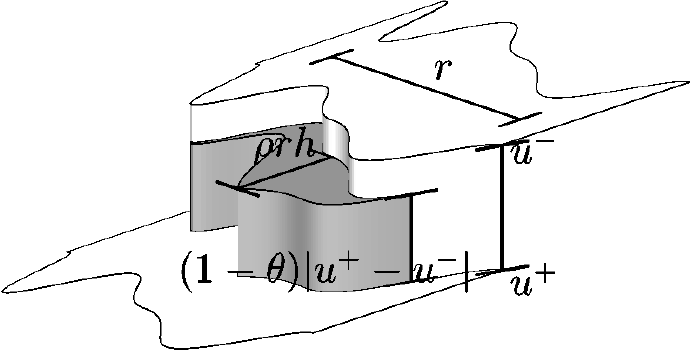}
    \caption{Illustration of the function $\newu(x) \defeq \theta u(x) + (1-\theta) \pf\gammarhoh u(x)$ (gray) constructed in §\ref{sec:fidestim} in comparison to $u$ (white).}
    \label{fig:unew}
\end{figure}

First, we have to estimate the discrepancy between the piecewise 
constant functions and the original ones. Without assuming
$u$ bounded, we could get an $O(\epsilon r^m)$ estimate. We are not
able to deal with this in the general analysis. We therefore have
to assume $u$ (locally) bounded, in order to improve this to
$O(\epsilon \rho r^m)$.

\begin{lemma}
    \label{lemma:phi-approx-1}
    Suppose $\phi$ is $p$-increasing with $1 \le p < \infty$.
    Suppose either $u, f \in \BVspace(\Omega) \isect L^\infty_{\loc}(\Omega)$,
    or that $p=1$ and $u, f \in \BVspace(\Omega)$.
    Let $x_0 \in J_u \setminus (\tilde J_f \union S_f \union Z_u)$.
    Given $\epsilon>0$, there exists $r_0>0$, independent of $\rho$
    and valid for every $\baseh \in \basehspace$ 
    with $-1 \le \baseh \le 1$, such that whenever $0<r<r_0$ and $-1 < \rho < 1$, then
    \begin{equation}
        \label{eq:phi-estim-limits}
        \begin{split}
        \int_{\Omega} \phi(\newu(x) & -f(x)) \d x - \int_{\Omega} \phi(u(x) -f(x)) \d x
        \\&
        \le
        \int_{U_r} \phi(\newuz(x) -f_0(x)) \d x - \int_{U_r} \phi(u_0(x) -f_0(x)) \d x 
        + \epsilon \abs{\rho} r^m.
        \end{split}
    \end{equation}
\end{lemma}

\begin{proof}
    We may without loss of generality assume $\rho>0$, since the case $\rho=0$
    is trivial, and the case $\rho<0$ can be handled by negating $\baseh$.
    We begin the proof by choosing $r>0$ small enough that $U_r \subset \Omega$.
    Further, if $p>1$, we pick $r>0$ small enough that $u$ and $f$ 
    are essentially bounded within $U_r$.
    Since $\gammarhoh$ is the identity outside $U_r$, it
    suffices to perform estimation with $U_r$.
    We do this with distinct arguments within the $O(\rho r^m)$ 
    sets $W_r^+ \defeq U_r^+ \isect \gamma(U_r^-)$ and 
    $W_r^- \defeq U_r^- \isect \gamma(U_r^+)$, 
    and within the $O(r^m)$ sets
    $\tilde W_r^+ \defeq U_r^+ \isect \gamma(U_r^+)$ and
    $\tilde W_r^- \defeq U_r^- \isect \gamma(U_r^-)$.
    We begin with the latter, observing that
    $\newuz(x)=u_0(x)$ for $x \in \tilde W_r^\pm$.
    Thus
    \begin{equation}
        \label{eq:tilde-wpm-estim0}
        \int_{\tilde W^\pm} \phi(\newuz(x) -f_0(x)) - \phi(u_0(x) -f_0(x)) \d x = 0.
    \end{equation}
    We estimate using Definition \ref{def:phi-increase} for $\L^m$-\ae $x \in \tilde W^\pm$ that
    \[
        \begin{split}
        \phi(\newu(x) -f(x)) & - \phi(u(x) -f(x))
        \\
        &
        \le
        C_\phi
        \left(\abs{\newu(x) -f(x)} - \abs{u(x) -f(x)}\right)
        \abs{\newu(x) -f(x)}^{p-1}
        \\
        &
        \le
        C_\phi(1-\theta) \abs{\pf\gammarhoh u(x) -u(x)}
        \abs{\newu(x) -f(x)}^{p-1}
        \\
        &
        \le
        C \abs{\pf\gammarhoh u(x) -u(x)}.
        \end{split}
    \]
    In the final inequality  we use the boundedness of $u$, 
    or alternatively $p=1$.
    We now employ Lemma \ref{lemma:f-transform-estim} to estimate
    \[
        \begin{split}
            \int_{\tilde W_r^\pm}
            \phi(\newu(x) -f(x)) - \phi(u(x) -f(x))
            \d x
            &
            \le
            C
            \int_{\tilde W_r^\pm}
            \abs{\pf\gammarhoh u(x) -u(x)}
            \d x
            \\
            &
            \le
            C M_{\inv\gammarhoh}
            \abs{Du}(U_r^\pm).
        \end{split}
    \]
    We have $M_{\inv\gammarhoh} \le \rho r$.
    By the construction of $Z_u$, because $x_0 \in J_u \setminus Z_u$, 
    choosing $r>0$ small enough, we can enforce
    $\abs{Du}(U_r^\pm) \le \epsilon r^{m-1}/(8C)$.
    Thus
    \begin{equation}
        \label{eq:tilde-wpm-estim}
        \int_{\tilde W^\pm}
        \phi(\newu(x) -f(x)) - \phi(u(x) -f(x))
        \d x
        \le \epsilon \rho r^m/8.
    \end{equation}
    
    It remains to estimate $\phi(\newu(x) -f(x)) - \phi(u(x) -f(x))$
    on $W_r^\pm$.
    By Definition \ref{def:phi-increase}, we have
    \[
        \begin{split}
        - \phi(u(x) -f(x))
        &
        \le
        - \phi(u_0(x) -f_0(x))
        \\ & \phantom{\le}
        +
        C_\phi\bigl(\abs{u_0(x) -f_0(x)} - \abs{u(x) - f(x)}\bigr)\abs{u_0(x) -f_0(x)}^{p-1}.
        \end{split}
    \]
    Since $x \mapsto \abs{u_0(x) -f_0(x)}^{p-1}$ is bounded, it follows
    for some constant $C>0$ that
    \begin{equation}
        \label{eq:wpm-noshift-est1}
        - \phi(u(x) -f(x))
        \le
        - \phi(u_0(x) -f_0(x))
        +
        C\abs{u_0(x) -f_0(x) - u(x) + f(x)}.
    \end{equation}
    We may estimate the final term in $W_r^\pm$ by
    \begin{equation}
        \label{eq:wpm-noshift-est2}
        \begin{split}
        \int_{W_r^\pm} \abs{u_0(x) & -f_0(x) - u(x) + f(x)} \d x
        \\
        &
        \le
        \int_{W_r^\pm} \abs{u^\pm(x_0) - u(x)} \d x 
        +
        \int_{W_r^\pm} \abs{\tilde f(x_0) - f(x)} \d x.
        \end{split}
    \end{equation}
    We now recall the notation $u_{x}^{z}(t)=u(x+tz)$ for one-dimensional
    restrictions of functions of bounded variation in a direction $z$
    with $\norm{z}=1$, as well as the formula \cite{ambrosio2000fbv}
    \[
        \abs{\iprod{Du}{z}(A)}=\int_{P_z^\perp A} \abs{Du_x^z}(\{t \in \R \mid x+tz \in A\})  \d x.
    \]
    Writing $h_r=h_r^{(+)} - h_r^{(-)}$ for $h_r^{(+)}, h_r^{(-)}\ge 0$,
    we may then further estimate
    \[
        \begin{split}
        \int_{W_r^\pm} & \abs{u^\pm(x_0)  - u(x)} \d x 
        \\
        &
        =
        \int_{P_{z_\Gamma}^\perp U_r} 
            \int_{f_\Gamma(v)}^{f_\Gamma(v)\pm\rho h_r^{(\pm)}(v)}
            \abs{u^\pm(x_0) - u(v + t z_\Gamma)} \d t \d v
        \\
        &
        =
        \int_{P_{z_\Gamma}^\perp U_r} 
            \int_{0}^{\rho h_r^{(\pm)}(v)}
            \abs{u^\pm(x_0) - u(g_\Gamma(v) \pm t z_\Gamma)} \d t \d v
        \\
        &
        \le
        \int_{P_{z_\Gamma}^\perp U_r} 
            \int_{0}^{\rho h_r^{(\pm)}(v)}
            \abs{u^\pm(g_\Gamma(v)) - u(g_\Gamma(v) \pm t z_\Gamma)} + \abs{u^\pm(g_\Gamma(v))-u^\pm(x_0)} \d t \d v
        \\
        &
        \le
        \int_{P_{z_\Gamma}^\perp U_r} 
            \int_{0}^{\rho h_r^{(\pm)}(v)}
            \abs{Du_{g_\Gamma(v)}^{z_\Gamma}}(\pm[0, t]) \d t \d v
        +
        \rho r
        \int_{P_{z_\Gamma}^\perp U_r}
            \abs{u^\pm(g_\Gamma(v))-u^\pm(x_0)} \d v
        \\
        &
        \le
        \rho r
        \int_{P_{z_\Gamma}^\perp U_r}
            \abs{Du_{g_\Gamma(v)}^{z_\Gamma}}(\pm[0, \rho r]) \d v
        +
        \rho r
        \int_{\Gamma \isect U_r}
            \abs{u^\pm(x)-u^\pm(x_0)} \d x
        \\
        &
        \le
        \rho r \abs{Du}(U_r^\pm)
        +
        \rho r
        \int_{\Gamma \isect U_r}
            \abs{u^\pm(x)-u^\pm(x_0)} \d x.
        \end{split}
    \]
    In the semifinal step we have used on the second term the area formula on the
    transformation $\inv g_\Gamma=P_{z_\Gamma}^\perp$, observing that
    $\jacobian_m P_{z_\Gamma}^\perp \le 1$.
    Since $x_0 \in J_u \setminus Z_u$  is a Lebesgue point of $u^\pm$ on $\Gamma$ and
    $\Theta^*_{m-1}(\abs{Du} \restrict \Gamma^\pm; x_0)=0$, choosing $r>0$ small enough,
    we can make the final quantity less than $\rho r \cdot \epsilon r^{m-1}/(16C)$. 
    This proves an estimate on the first term on the right hand side of
    \eqref{eq:wpm-noshift-est2}. The second term we approximate by analogous
    arguments on $\tilde f(x_0)-f$, using the fact that $x_0 \not \in \tilde J_f$ for
    the term $\abs{Df}(U_r^\pm)$.
    Referring back to \eqref{eq:wpm-noshift-est1} and \eqref{eq:wpm-noshift-est2},
    we then deduce
    \begin{equation}
        \label{eq:untrans-wrpm}
        -\int_{W_r^\pm}
        \phi(u(x) -f(x)) \d x
        \le
        - \int_{W_r^\pm}
        \phi(u_0(x) -f_0(x)) \d x
        + \epsilon\rho r^m/8.
    \end{equation}
    
    It remains to estimate the transformed terms on $W_r^\pm$.
    Similarly to \eqref{eq:wpm-noshift-est1},
    using the essential boundedness of \emph{both} $u$ and $f$ on $U_r$,
    or $p=1$, we deduce for $\L^m$-\ae $x \in W_r^\pm$
    \begin{equation}
        \label{eq:wpm-noshift-est1-x}
        \phi(\newu(x) -f(x))
        \le
        \phi(\newuz(x) -f_0(x))
        +
        C\abs{\newuz(x) -f_0(x) - \newu(x) + f(x)}.
    \end{equation}
    A simple application of the area formula and repeating
    the arguments for the untransformed term above, yields
    again the estimate
    \begin{equation}
        \label{eq:trans-wrpm}
        \int_{W_r^\pm}
        \phi(\newu(x) -f(x)) \d x
        \le
        \int_{W_r^\pm}
        \phi(\newuz(x) -f_0(x)) \d x
        + \epsilon\rho r^m/8.
    \end{equation}
    Summing the estimates
    \eqref{eq:tilde-wpm-estim0},
    \eqref{eq:tilde-wpm-estim},
    \eqref{eq:untrans-wrpm} and \eqref{eq:trans-wrpm},
    and minding that $\newu=u$ outside $U_r$,
    we deduce \eqref{eq:phi-estim-limits}.
\end{proof}

\begin{lemma}
    \label{lemma:phi-const-approx}
    Suppose $\phi$ is $p$-increasing with $p>1$. Then there 
    exist $\theta \in (0, 1)$ and a constant
    $C = C(\phi, u^\pm(x_0), \tilde f(x_0)) > 0$ such that
    for $\rho>0$ holds
    \[
        \begin{split}
        \int_{U_r} \phi(\newuz[\rho](x) -f_0(x)) \d x 
        &
        + \int_{U_r} \phi(\newuz[-\rho](x) -f_0(x)) \d x 
        \\
        &
        - 2 \int_{U_r} \phi(u_0(x) -f_0(x)) \d x 
        \le - C \rho r^m.
        \end{split}
    \]
\end{lemma}

\begin{figure}
    \centering
    %\asyinclude{asy/ab.asy}
    \includegraphics{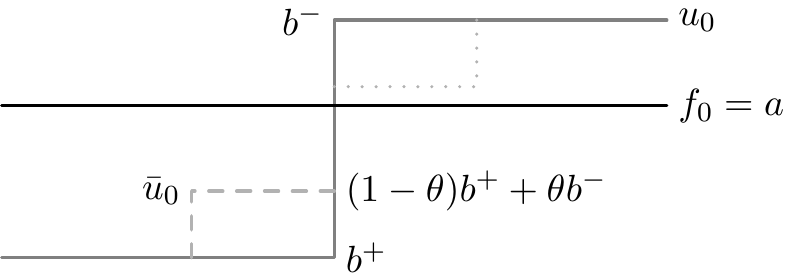}
    \caption{The situation in the proof of Lemma \ref{lemma:phi-const-approx}. Here $\bar u_0=\newuz$ for $\rho>0$ (dashed line). The dotted ine indicates the alternative $\newuz[-\rho]$ that we may need to use if the curvature of $J_u$ is high at $x_0$, as measured by the double-Lipschitz comparison condition for $R$.}
    \label{fig:ab}
\end{figure}

\begin{proof}
    Let $\rho>0$. Observe that outside
    $\hat W_r^\pm \defeq U_r^\pm \isect \gammarhoh[\pm\rho](U_r^\mp)$
    the transformed piecewise constant functions agree with the originals,
    \[
        \newuz[\pm\rho](x)=u_0(x),
        \quad
        (x \in U_r \setminus \hat W_r^\pm).
    \]
    Letting
    \[
        K^\pm \defeq \esssup_{x \in \hat W_r^\pm} \bigl(\abs{\newuz[\pm\rho](x) -f_0(x)} - \abs{u_0(x) -f_0(x)}\bigr)
    \]
    and using Definition \ref{def:phi-increase}, we derive for $x \in \hat W_r^\pm$ the estimate
    \[
        \begin{split}
        \phi(\newuz[\pm\rho](x) &-f_0(x))  - \phi(u_0(x) -f_0(x)) 
        \\ &
        \le
        C_\phi\bigl(\abs{\newuz[\pm\rho](x) -f_0(x)} - \abs{u_0(x) -f_0(x)}\bigr)
        \abs{\newuz[\pm\rho](x) -f_0(x)}^{p-1}
        \\
        &
        \le
        K^\pm C_\phi \abs{\newuz[\pm\rho](x) -f_0(x)}^{p-1}.
        \end{split}
    \]
    Overall we then have
    \begin{equation}
        \label{eq:fuz-k-estim}
        \begin{split}
        \int_{U_r} \phi(\newuz[\pm\rho](x) -f_0(x)) \d x & - \int_{U_r} \phi(u_0(x) -f_0(x)) \d x 
        \\
        &
        \le K^\pm C_\phi \int_{\hat W_r^\pm} \abs{\newuz[\pm\rho](x) -f_0(x)}^{p-1} \d x.
        \end{split}
    \end{equation}

    Let us denote $b^\pm \defeq u^\pm(x_0)$ and $a \defeq \tilde f(x_0)$. Then
    \[
        \begin{aligned}
        \abs{\newuz[\rho](x) -f_0(x)}^{p-1} & = \zeta^+ \defeq \abs{(1-\theta)(b^- -a) + \theta (b^+  - a)}^{p-1},
        \quad (x \in \hat W_r^+), \quad \text{and} \\
        \abs{\newuz[-\rho](x) -f_0(x)}^{p-1} & = \zeta^- \defeq \abs{(1-\theta)(b^+ -a) + \theta (b^-  - a)}^{p-1},
        \quad (x \in \hat W_r^-).
        \end{aligned}
    \]
    Moreover,
    \[
        \begin{aligned}
        K^+ & = \abs{(1-\theta)(b^- -a) + \theta (b^+  - a)}
                    - \abs{b^+  - a},
            \quad\text{and}
        \\
        K^- & = \abs{(1-\theta)(b^+ -a) + \theta (b^-  - a)}
                    - \abs{b^-  - a}.
        \end{aligned}
    \]
    By Lemma \ref{lemma:triangle-mod} in the appendix
    \[
        \L^m(\hat W_r^+)=\L^m(\hat W_r^-)
        =
        \int_{V_\Gamma} \rho \abs{h_r(v)} \d \H^{m-1}(v)
        =
        C_h \rho r^m
    \]
    for some constant $C=C_h$. Thus summing \eqref{eq:fuz-k-estim} for $\pm\rho$ gives
    \[
        \begin{split}
        \int_{U_r} \phi(\newuz[\rho](x) -f_0(x)) \d x
        &
        +
        \int_{U_r} \phi(\newuz[-\rho](x) -f_0(x)) \d x
        \\
        &
        \phantom{-}
        - 2 \int_{U_r} \phi(u_0(x) -f_0(x)) \d x 
        \\
        &
        \le C_\phi C_h \rho r^m
        ( K^+ \zeta^+ + K^- \zeta^-).
        \end{split}
    \]
    In order to reach our conclusion, it therefore remains to show that
    \begin{equation}
        \label{eq:k-zeta-neg}
        K^+ \zeta^+ + K^- \zeta^- < 0.
    \end{equation}
    We calculate
    \begin{equation}
        \label{eq:kplusest}
        K^+ \le (1-\theta)(\abs{b^- - a} - \abs{b^+ - a}),
    \end{equation}
    and
    \begin{equation}
        \label{eq:kminusest}
        K^- \le (1-\theta)(\abs{b^+ - a} - \abs{b^- - a}).
    \end{equation}
    Therefore
    \[
        K^+ \zeta^+ + K^- \zeta^-
        \le
        (1-\theta)(\abs{b^- - a} - \abs{b^+ - a})(\zeta^+ - \zeta^-).
    \]
    We concentrate on the case $b^+  <  b^-$ with $a \ge (b^+ + b^-)/2$, 
    other cases handled analogously by appropriate changes of roles and
    negations. This is illustrated in Figure \ref{fig:ab}. Then
    \begin{equation}
        \label{eq:bplusminusa}
        \abs{b^+-a} \ge \abs{b^- - a}.
    \end{equation}
    If \eqref{eq:bplusminusa} holds strictly, 
    we have $\zeta^+ > \zeta^-$ for $\theta \in (0, 1)$ large enough.
    This shows \eqref{eq:k-zeta-neg} and is the only place where we need 
    the assumption $p>1$. If \eqref{eq:bplusminusa} does
    not hold strictly, i.e., $a = (b^+ + b^-)/2$, we
    may have $\zeta^+ = \zeta^-$, but observe that $K^+=0$ and
    \eqref{eq:kplusest} holds strictly
    for large $\theta$. Indeed, this is the case whenever
    $b^+ < a, b^-$, because some interpolation
    $(1-\theta)b^- + \theta b^+$ is always closer to $a$ than $b^+$ is.
    We have thus proved \eqref{eq:k-zeta-neg}, and may conclude
    the proof of the lemma.
\end{proof}

\begin{lemma}
    \label{lemma:phi-approx}
    %Suppose Assumption \ref{ass:phi-increase} holds with $1 < p < 1^*$. 
    Suppose $\phi$ is $p$-increasing with $1 < p < \infty$, and
    \emph{both} $u, f \in \BVspace(\Omega) \isect L^\infty_\loc(\Omega)$.
    Let $x_0 \in J_u \setminus (\tilde J_f \union S_f \union Z_u)$.
    Then there exist $\theta \in (0, 1)$, $r_0>0$, independent of $\rho$,
    and a constant $C = C(\phi, u^\pm(x_0), \tilde f(x_0)) > 0$,
    such that whenever $0<r<r_0$ and $0 < \rho < 1$, it holds
    \begin{equation}
        \label{eq:phi-approx}
        \int_{\Omega} \phi(\newu[\rho](x) -f(x)) \d x 
        +
        \int_{\Omega} \phi(\newu[-\rho](x) -f(x)) \d x 
        - 2\int_{\Omega} \phi(u(x) -f(x)) \d x
        \le - C \rho r^m.
    \end{equation}
\end{lemma}

\begin{proof}
    Combine Lemma \ref{lemma:phi-approx-1} and Lemma \ref{lemma:phi-const-approx},
    choosing $\epsilon>0$ small enough in the former.
\end{proof}

\subsection{The effect of the shift transformation on the regulariser}
\label{sec:regestim}

We now summarise the estimates we get for the regulariser $R$
using double-Lipschitz comparability and the Lipschitz transformations
of \S\ref{sec:lipschitz}. 
\begin{lemma}
    \label{lemma:r-approx}
    Suppose $R$ is a double-Lipschitz comparable and $x_0 \in J_u \setminus Z_u$.
    Then there exists a constant $C=C(u, x_0)$ and $r_0>0$ such that
    for $0 < r < r_0$ and $0< \rho < 1$ holds
    \[
        R(\newu[\rho]) + R(\newu[-\rho]) - 2R(u)
        \le C\rho^2 r^{m-1}.
    \]
\end{lemma}
\begin{proof}
    We know from Lemma \ref{lemma:h-shift-gamma}
    the existence of a constant $C>0$ such that $\bitransfull{\gammarhoh}{\gammarhohNEG}$ defined in \eqref{eq:bitransfull} satisfies
    \[
        \bitransfull{\gammarhoh}{\gammarhohNEG}
        \le C \rho^2.
    \]
    By convexity
    \[
        R(\newu[\rho]) + R(\newu[-\rho]) - 2R(u)
        \le (1-\theta)\bigl(R(\pf\gammarhoh u) + R(\pf\gammarhohNEG u) - 2R(u)).
    \]
    By the double-Lipschitz comparability of $R$, we
    thus find that
    \[
        R(\newu[\rho]) + R(\newu[-\rho]) - 2 R(u) \le 
         (1-\theta) \RCa \bitransfull{\gammarhoh}{\gammarhohNEG} 
        \abs{D u}(\closure U_r).
    \]
    Since $x_0 \in J_u \setminus Z_u$, whenever $r>0$ is small 
    enough, we have for a constant $C'>0$  that
    \[
        \abs{D u}(\closure U_r))
        \le
        \abs{D u}(\B(x_0, C' r)) 
        \le
        2 \omega_{m-1} \abs{u^+(x_0)-u^-(x_0)}  (C'r)^{m-1}.
    \]
    The claim follows by combining the above estimates.
\end{proof}

%%%
%\section{Proof of the main result}
%\label{sec:jumpset}
%%%

\subsection{Patching it all together}

We may finally prove our main result for $p>1$ 
by combining the above lemmas as follows.

\begin{proof}[Proof of Theorem \ref{theorem:jumpset-strict}]
    %Let us denote by $A_u$ the $\H^{m-1}$-null set from 
    %Definition \ref{def:lipschitz-trans}, such that $R$ is 
    %double-Lipschitz comparable at every $x_0 \in \Omega \setminus A_u$.
    Suppose, to reach a contradiction, that
    $\H^{m-1}(J_u \setminus J_f)>0$. Since $\H^{m-1}(Z_u)=0$, $\H^{m-1}(S_f \setminus J_f)=0$ with $J_f \subset S_f$, and by Proposition \ref{prop:bv-density}, $\H^{m-1}(\tilde J_f \setminus J_f)=0$ with $J_f \subset \tilde J_f$, we may select a point $x_0 \in J_u \setminus (\tilde J_f \union S_f \union Z_u)$.
    We then apply Lemma \ref{lemma:phi-approx} for the fidelity estimate
    \begin{equation}
        \label{eq:phi-approx-combine}
        \begin{split}
        \int_{\Omega} \phi(\newu[\rho](x) -f(x)) \d x 
        &
        +
        \int_{\Omega} \phi(\newu[-\rho](x) -f(x)) \d x 
        \\
        &
        - 2\int_{\Omega} \phi(u(x) -f(x)) \d x
        \le
        - C_1 \rho r^m,
        \end{split}
    \end{equation}
    where $0 < r < r_1$ and $0 < \rho < 1$.
    Lemma \ref{lemma:r-approx} gives for the regulariser 
    and $0 < r < r_2$ the estimate
    \[
        R(\newu[\rho]) - R(\newu[-\rho]) - 2R(u)
        \le C_2 \rho^2 r^{m-1}.
    \]
    All the constants $C_1, C_2, r_1, r_2>0$ are independent of $\rho \in (0, 1)$.
    Picking $0 < r < \min\{r_1, r_2\}$, and summing these estimates, we obtain
    \[
        \begin{split}
        \left(\int_{\Omega} \phi(\newu[\rho](x) -f(x)) \d x + R(\newu[\rho])\right)
        &
        +
        \left(\int_{\Omega} \phi(\newu[-\rho](x) -f(x)) \d x + R(\newu[-\rho])\right)
        \\
        &
        \phantom{\le}
         - 2\left(\int_{\Omega} \phi(u(x) -f(x)) \d x + R(u)\right)
         \\
        &
        \le
        C_2\rho^2 r^{m-1} - C_1 \rho r^m.
        \end{split}
    \]
    If $\rho>0$ is small enough, this is negative.
    Thus either
    \[
        \begin{aligned}
        \int_{\Omega} \phi(\newu[\rho](x) -f(x)) \d x + R(\newu[\rho]) & < \int_{\Omega} \phi(u(x) -f(x)) \d x + R(u), \quad \text{or}\\
        \int_{\Omega} \phi(\newu[-\rho](x) -f(x)) \d x + R(\newu[-\rho]) & < \int_{\Omega} \phi(u(x) -f(x)) \d x + R(u).
        \end{aligned}
    \]
    This contradicts the optimality of $u$.
    Therefore $\H^{m-1}(J_u \setminus J_f)=0$.
\end{proof}

\section{Considerations for the $L^1$ fidelity}
\label{sec:main-l1}

The only difficulty in extending the proofs of 
\S\ref{sec:main} for $p$-increasing ($p>1$) fidelities
$\phi$ to the $L^1$ fidelity $\phi(x)=x$,
is Lemma \ref{lemma:phi-const-approx}. We do not necessarily
get a constant $C>0$ there, but $C=0$. Then the argument
in the proof of Theorem \ref{theorem:jumpset-strict} does not
go through. This has the consequence that it is possible to
have $\H^{m-1}(J_u \setminus J_f) > 0$. However, as we
will see in this section, the residual $J_u \setminus J_f$ 
has a regular structure, although our result is somewhat
weaker than the known result for $\TV$ \cite{duval2009tvl1}.

In this section, we still define $h_r$ as before, as
\[
    h_r(v) \defeq r \baseh((v-\vp{x_0})/r),
\]
for some $0 \ne \baseh \in \basehspace$, with $0 \le \baseh \le 1$
unless explicitly stated otherwise. However, we also denote
\[
    h_r^{v_0}(v) \defeq r \baseh((v-v_0)/r),
\]
when we want to be more explicit about the base point.
Noting that $I_r > 0$ under these assumptions, we denote
for brevity
\[
    I_r \defeq \int_{V_\Gamma} h_r \d v,
\]
and
\[
    \lambda_u(x) \defeq \abs{u^+(x)-u^-(x)}.
\]

Finally, we need the missing technical curvature definition for Theorem \ref{theorem:jumpset-l1}.

\begin{definition}
    \label{def:r-curvature}
    Let $R$ be an admissible regularisation functional
    on $\BVspace(\Omega)$. We define the \emph{transformation differential}
    of $R$ at $u$ as
    \[
        \TDIFF_u^R(\gamma) \defeq \lim_{\rho \downto 0} \frac{R(\pf{(\rho \gamma+ (1-\rho)\iota)} u)-R(u)}{\rho},
        \quad (\gamma \in \LipClass(\Omega)),
    \]
    when the limit exists.
    With $\baseh \in \basehspace$, we then define the \emph{pointwise $R$-curvature}
    at $x_0$ along the Lipschitz graph $\Gamma$, if the limit exists, as
    \[
        \CURVP_u^{R,\Gamma,\baseh}(x_0) \defeq \lim_{r \downto 0} \inv I_r \TDIFF_R(\gamma_{h_r}),
        %\CURVP_u^{R,\Gamma,\baseh}(x_0) \defeq \lim_{r \downto 0} 
        %    \frac{\omega_{m-1}\TDIFF_u^R(\gamma_{h_r})}
        %         {I_1 r \abs{Du}(\B(x_0, r))},
        \quad (x_0 \in \Gamma).
    \]
    Since our results do not depend on the choice of $\baseh$, we 
    simply write $\CURVP_u^{R,\Gamma}(x_0) \defeq \CURVP_u^{R,\Gamma,\baseh}(x_0)$
    in the statement of Theorem \ref{theorem:jumpset-l1}.
    %Without reference to $\Gamma$ and $\baseh$, we also set
    %\[
    %    \CURVP_u^{R}(x_0) \defeq \CURVP_u^{R,\Gamma,\baseh}(x_0)
    %\]
    %for $\Gamma=\Gamma_i$ for some $\Gamma_i \ni x_0$ from
    %\eqref{eq:ju-lambda-collection}, when the limit is uniquely defined. 
\end{definition}

\begin{remark}
    In many ways, it would make more sense to define the pointwise
    curvature by
    \[
        \tilde \CURVP_u^{R,\Gamma,\baseh}(x_0) \defeq \lim_{r \downto 0} 
            \frac{\omega_{m-1}\TDIFF_u^R(\gammax{1}{h}{r})}
                 {I_1 r \abs{Du}(\B(x_0, r))}.
    \]
    At a point $x_0 \in J_u$ this would give the normalised
    curvature $\CURVP_u^{R,\Gamma,\baseh}(x_0)/\lambda_u(x_0)$;
    see Lemma \ref{lemma:r-curvature}.
    For practical reasons, in order to be able to work easily 
    at points $x_0 \in \Gamma \setminus J_u$,
    we use the earlier definition, however.
\end{remark}

%%%
\subsection{Expressions for the pointwise $R$-curvature}
%%%

The proofs of the curvature condition in Theorem \ref{theorem:jumpset-l1} is mainly
based on the expressions in the following two lemmas, plus the regularity results
that follows.
Since the proofs of these two lemmas are mainly technical and not very informative,
we relegate them to Appendix \ref{appendix:l1}.

The first one of the lemmas is our rough counterpart of Lemma \ref{lemma:gamma-rho-h-r} for $p>1$. The proof uses a different technique, based on only one-sided Lipschitz
estimates.

\begin{lemma}
    \label{lemma:r-curvature-sol}
    Suppose $u$ solves \eqref{eq:prob}.
    Let $\Gamma \subset \Omega$ be a Lipschitz $(m-1)$-graph.
    Then
    \[
        \CURVP_u^{R,\Gamma,\baseh}(x_0)=\lambda_u(x_0)C_\phi,
        \quad (\H^{m-1}\text{-\ae} x_0 \in (J_u \setminus J_f) \isect \Gamma).
    \]
\end{lemma}

For the second one of the curvature lemmas, giving a more familar mean curvature
expression for the pointwise $R$-curvature, we require our rather strong assumption
that $Du$ is ``essentially piecewise constant'' around $J_f \setminus J_u$.

\begin{lemma}
    \label{lemma:div-curv}
    Let $\Gamma \subset \Omega$ be a Lipschitz $(m-1)$-graph.
    At $\H^{m-1}$-\ae point $x_0 \in \Gamma$, if
    $\Theta_m(\abs{Du} \restrict \Omega \setminus \Gamma; x_0)=0$, then
    \[
        \CURVP_u^{R,\Gamma,\baseh}(x_0) 
        =
        -\divergence \frac{\DIFFSS f_{\Gamma}(\pgp{x_0})}{\sqrt{1+\norm{\DIFFSS f_{\Gamma}(\pgp{x_0})}^2}}
        \cdot
         \RCs \lambda_u(x_0).
    \]
\end{lemma}

\subsection{Regularity}

The following result will allow us to obtain additional regularity from
the expression for the pointwise $R$-curvature in Lemma \ref{lemma:div-curv}.

\begin{lemma}
    \label{lemma:regul}
    Let $f \in W^{1,\infty}(V)$, and $0 \le \lambda \in L^\infty(V)$,
    where $V \subset z^\perp \isect \B(0, r_0)$, $z \in \R^m$, $r_0 > 0$. 
    Suppose for some $\alpha>0$ we have
    \begin{equation}
        \label{eq:regul-curvature}
        -\divergence \frac{\DIFFSS f(v)}{\sqrt{1+\norm{\DIFFSS f(v)}^2}}=\inv\alpha,
    \end{equation}
    for $\H^{m-1}$-\ae $v \in V$. Then
    \eqref{eq:regul-curvature} holds for every $v \in V$,
    and $f \in C^{2,\allLambda}(V)$.
\end{lemma}
    
\begin{proof}
    Since \eqref{eq:regul-curvature} holds almost everywhere,
    also
    \begin{equation}
        \label{eq:curvconst-no-lambda}
        \int_V \adaptiprod{\frac{\DIFFSS f(v)}{\sqrt{1+\norm{\DIFFSS f(v)}^2}}}{\DIFFSS h(v)} \d v
        = \inv\alpha \int_V h(v) \d v,
        \quad
        (h \in W_0^{1,\infty}(V)).
    \end{equation}
    After showing $C^2$ regularity of $f$, 
    \eqref{eq:regul-curvature} then holds for all $v \in V$. 
    To show this, let us start by defining
    \[
        \Jfunct(q) \defeq
        \int_V \sqrt{1+\norm{\DIFFSS q(x)}^2} \d x
        -\inv\alpha \int_V q(x)-f(x) \d x,
    \]
    and consider the problem
    \begin{equation}
        \label{eq:regul-min}
        \min \{\Jfunct(q) \mid q \in Q\},
    \end{equation}
    where the domain
    \[
        Q \defeq \{q \in W^{1,\infty}(V) \mid q|\BD V = f|\BD V\}.
    \]
    Since $t \mapsto \sqrt{1+t^2}$ is strictly convex,
    and the weak solution $h \in W^{1,\infty}_0(V)$ to the 
    differential equation
    % (Green's identity => |D(h-\tilde h)|=0.)
    \[
        \DIFFSS h=\DIFFSS \tilde h \text{ on } V,
        \quad
        h|\BD V =0,
    \]
    is unique for any $\tilde h \in W^{1,\infty}_0(V)$,
    we find that $\Jfunct$ is strictly convex on $Q$.
    It therefore has a unique minimiser.
    But \eqref{eq:curvconst-no-lambda} is exactly the necessary
    and sufficient optimality condition for $f$ to be
    a minimiser of $\Jfunct$. This can be deduced from
    the following paragraphs studying a slightly modified
    functional $\hat\Jfunct$. At this moment it is important
    to notice that $f$ must then be the unique minimiser of $\Jfunct$.
    
    We intend to use the regularity results on quasilinear elliptic partial
    differential equations to show that $f$ satisfies the claimed
    regularity properties. To do this, we however need to force some 
    coercivity/ellipticity properties. We therefore define
    \[
        \hat \Jfunct(q)
        \defeq
        \int_V g(\DIFFSS q(x)) \d x - \inv \alpha \int_V\bigl( u(x)-f(x)\bigr) \d x,
    \]
    where 
    \[
        g(p) \defeq
            \sqrt{1+\norm{p}^2} 
            %+ \frac{1}{2}\norm{p-\DIFFSS f(x)}^2
            +
            \sum_{i=1}^m \psi(\abs{p_i}),
    \]
    and
    \[
        \psi(t)=C_0 \bigl(\max\{t-K, 0\}\bigr)^4
    \]
    for $K \defeq 2\sup_{x \in V} \norm{\DIFFSS f(x)}$, and
    yet undetermined $C_0 > 0$.
    Observe that $\psi \in C^3(\R)$.
    Now $\Jfunct \le \hat \Jfunct$ and $\hat \Jfunct(f)=\Jfunct(f)$,
    so that $\hat \Jfunct$ also has the unique minimiser $f$ within $U$.
    
    We now write the optimality conditions for potential minimisers
    $q \in Q$ of $\hat \Jfunct$. As differentials of the mappings
    \[
        \rho \mapsto \hat \Jfunct(q+\rho h),
        \quad (h \in W^{1, \infty}_0(V)),
    \]
    we calculate
    \[
        \int_V
            \iprod{\grad g(\DIFFSS (q+\rho h)(x))}{\DIFFSS h(x)}
            \d x
        - \inv \alpha \int_V h(x) \d x.
    \]
    If $q$ minimises $\hat \Jfunct$, we obtain by setting $\rho=0$
    the optimality condition
    \[
        \int_V
            \iprod{\grad g(\DIFFSS q(x))}{\DIFFSS h(x)}
            \d x
        = 
        \inv \alpha \int_V h(x) \d x,
        \quad (h \in W^{1, \infty}_0(V)),
    \]
    where
    \[
        \grad g(p) = \frac{p}{\sqrt{1+\norm{p}^2}} + 
            \sum_{i=1}^m \psi'(\abs{p_i})(\sign p_i)e_i.
    \]
    It follows that $q$ is a solution of the
    quasilinear elliptic differential equation
    \begin{equation}
        \label{eq:regul-de}
        \divergence A(\DIFFSS q) = -\inv \alpha,
        \quad
        q \in W^{1, \infty}(V), 
        \quad
        q|\BD V = f|\BD V,
    \end{equation}
    where
    \[
        A(p)=(A_1(p),\ldots,A_m(p)) \defeq \grad g(p).
    \]
    We want to show that \eqref{eq:regul-de}
    has a solution $\hat q \in C^{2,\lambda}(V)$ for any $\lambda \in (0, 1)$.
    But $\hat \Jfunct$ had the unique minimiser $f$. 
    Therefore $f=\hat q \in C^{2,\lambda}(V)$ for any $\lambda \in (0, 1)$.
    From \eqref{eq:regul-de} and the definition of $\psi$ we have
    moreover that \eqref{eq:regul-curvature} holds.
    
    To show the existence of a solution $\hat q \in C^{2,\lambda}(V)$ ,
    we employ \cite[Theorem 15.19]{gilbarg2001elliptic}. To do so, we 
    need to show that
    \begin{equation}
        \label{eq:ell-cond-1}
        A_i \in C^{1,\lambda}(\R^m),\quad (i=1,\ldots,m),
    \end{equation}
    that for some $\tau > - 1$ and $C_1, C_2 \in \R$ both
    \begin{equation}
        \label{eq:ell-cond-2}
        \iprod{p}{A(p)} \ge \norm{p}^{2+\tau} - C_1,
    \end{equation}
    and
    \begin{equation}
        \label{eq:ell-cond-3}
        \norm{\grad A(p)} \le C_2 (1+\norm{p})^\tau.
    \end{equation}
    It is to force these conditions, why we introduced
    the $\psi$-penalty in $g$.
    Minding that $\psi \in C^3(\R)$, 
    we have $A \in C^2(\R^m; \R^m)$, whence
    condition \eqref{eq:ell-cond-1} readily follows.
    Moreover, we have for some $C_0, C_1 \ge 0$ that
    \[
        \sum_{i=1}^m \psi'(p_i)\abs{p_i}
        \ge C_0\sum_{i=1}^m \abs{p_i}^4 - C_1
        \ge \norm{p}^4 - C_1,
    \]
    so that
    \[
        \iprod{p}{A(p)}
        = 
        \frac{\norm{p}^2}{\sqrt{1+\norm{p}^2}} + 
            \sum_{i=1}^m \psi'(p_i)\abs{p_i}
        \ge
        \norm{p}^4 - C_1.
    \]
    Thus \eqref{eq:ell-cond-2} holds with $\tau=2$.
    It remains to show that \eqref{eq:ell-cond-3} also
    holds with $\tau=2$.
    We have
    \[
        \grad A(p)=\frac{I}{\sqrt{1+\norm{p}^2}}
            -\frac{p \otimes p}{(1+\norm{p}^2)^{3/2}}
            +\sum_{i=1}^m \psi''(\abs{p_i}) e_i \otimes e_i.
    \]
    The first two terms are bounded, while for
    for some $C_3, C_4 > 0$ we have the estimate
    \[
        \abs{\psi''(\abs{p_i})} \le C_3 + C_4 p_i^2.
    \]
    Therefore, for some $C_5, C_2 > 0$, we find that
    \[
        \norm{\grad A(p)} \le C_5 + C_4 \norm{p}^2
        \le C_2(1+\norm{p})^2.
    \]
    This proves \eqref{eq:ell-cond-3}, concluding the proof.
\end{proof}

\subsection{Proof of the main result for $p=1$}

We may now prove our main result for 1-increasing fidelities
by combining the results of the above lemmas.

\begin{proof}[Proof of Theorem \ref{theorem:jumpset-l1}]
    Let $\{\Gamma_i\}_{i=1}^\infty$ be the collection of
    Lipschitz graphs from \eqref{eq:ju-lambda-collection}.
    Repeated application of Lemma \ref{lemma:r-curvature-sol} 
    for each $\Gamma=\Gamma_i$, ($i \in \Z^+$) shows that
    \begin{equation}
        \label{eq:jumpset-l1-proof-curv-1}
        \CURVP_u^{R,\Gamma_i,\baseh}(x_0)=\lambda_u(x_0)C_\phi
        \quad\text{for}\quad
        \H^{m-1}\text{-\ae}  x_0 \in (J_u \setminus J_f) \isect \Gamma_i.
    \end{equation}
    Choosing $\Lambda_i=\Gamma_i$ proves \eqref{eq:jumpset-l1-r-curvature}.

    It remains to prove the higher regularity \eqref{eq:jumpset-l1-curvature}
    under assumptions of approximate piecewise constancy.
    We fix $\Gamma=\Gamma_i$ for some $i \in \Z^+$.
    By assumption $\Theta_m(\abs{Du} \restrict \Omega \setminus \Gamma; x_0)=0$
    at $\H^{m-1}$-\ae $x_0 \in \Gamma$. By Lemma \ref{lemma:div-curv} we therefore
    have at $\H^{m-1}$-\ae $x_0 \in \Gamma$ the expression
    \[
        \CURVP_u^{R,\Gamma,\baseh}(x_0) 
        =
        -\divergence \frac{\DIFFSS f_{\Gamma}(\pgp{x_0})}{\sqrt{1+\norm{\DIFFSS f_{\Gamma}(\pgp{x_0})}^2}}
        \cdot
        \lambda_u(x_0).
    \]
    We plan to use Lemma \ref{lemma:regul}. This depends on 
    the domain $V_{\Gamma}$ being open, and \eqref{eq:jumpset-l1-proof-curv-1}
    holding $\H^{m-1}$-\ae on $\Gamma=\Gamma_i$, instead of just on $ (J_u \setminus J_f) \isect \Gamma_i$.  We therefore need a tiny extra argument before the application of the lemma.
    
    We recollect from the proof of Lemma \ref{lemma:div-curv} the expression
    \[
        \CURVP_u^{R,\Gamma,\baseh}(x_0) 
        =\lim_{r \downto 0} \curvf_{1,f}(h_r/I_r) \cdot \lambda_u(x_0)
    \]
    for
    \[
        \curvf_{1,f}(h) \defeq \int_V \adaptiprod{\frac{\DIFFSS f_\Gamma(v)}{\sqrt{1+\norm{\DIFFSS f_\Gamma(v)}^2}}}{\DIFFSS h(v)} \d v.
    \]
    The functional $\curvf_{1,f}$ is continuous
    on $H_0^1(V_\Gamma)$. From this, it is not difficult to
    see that each of the sets
    \[
        W_{k,r} \defeq \{ v \in V_\Gamma \mid \abs{\curvf_{1,f}(h_r^v/I_r) - C_\phi} \ge 1/k \},
        \quad (k \in \Z^+,\, r > 0).
    \]
    is closed within $V_\Gamma$. Therefore so are the sets
    \[
        W_k^\sigma \defeq \Isect_{0<r<\sigma} W_{k,r},
        \quad
        (k \in \Z^+,\, \sigma > 0),
    \]
    consisting of points $v \in V_\Gamma$ where
    \[
        \liminf_{r \downto 0} \abs{\curvf_{1,f}(h_r^v/I_r) - C_\phi} \ge 1/k.
    \]
    
    Suppose $\ri W_k^\sigma \ne \emptyset$.
    However $\H^{m-1}(\ri W_k^\sigma \isect (J_u \setminus J_f))=0$,
    because otherwise we would find a set of positive mass of points
    $x_0 \in J_u \setminus J_f$ where
    %Proposition \ref{proposition:r-curvature}\ref{item:r-curvature-first} does not hold,
    $\CURVP_u^{R,\Gamma,\baseh}(x_0) \ne \lambda_u(x_0)C_\phi$,
    and reach a contradiction to \eqref{eq:jumpset-l1-proof-curv-1}.
    %But, we could then replace $\Gamma_i$
    %by $g_{\Gamma_i}(V_{\Gamma_i} \setminus W_k)$ in our
    %collection \eqref{eq:ju-lambda-collection}. 
    Because we could always replace $\Gamma$ by $\Gamma \setminus g_\Gamma(\closure \ri  W_k^\sigma)$,
    we may therefore
    assume that $\ri W_k^\sigma = \emptyset$ for every $k \in \Z^+$
    and $\sigma>0$. In particular every $W_k^\sigma$ is nowhere dense in $V_\Gamma$.
    The Baire category theorem then says that
    \[
        V_\Gamma' \defeq \Isect_{k=1}^\infty \Isect_{j=1}^\infty (V_\Gamma \setminus W_k^{1/j})
        = \Isect_{k=1}^\infty \Isect_{j=1}^\infty \Union_{0<r<1/j} (V_\Gamma \setminus W_{k,r}).
    \]
    is a dense open set within $V_\Gamma$. But every $v \in V_\Gamma'$
    satisfies for some $r_j \downto 0$ the limit
    \[
        \lim_{j \to \infty} \curvf_{1,f}(h_{r_j}^v/I_{r_j}) = C_\phi.
    \]
    Considering that $\lim_{r \downto 0} \curvf_{1,f}(h_{r}^v/I_{r})$
    exists at a Lebesgue point of $\DIFFSS f_\Gamma$, we deduce that
    $\CURVP_u^{R,\Gamma,\baseh}(g_\Gamma(v))=\lambda_u(g_\Gamma(v))C_\phi$
    for $\H^{m-1}$-\ae $v \in V_\Gamma'$. Moreover, as required
    \[
        \H^{m-1}(J_u \setminus (J_f \union \Lambda))=0
    \]
    for $\Lambda = \Union_{i=1}^\infty \Lambda_i$
    with $\Lambda_i \defeq g_\Gamma(V_{\Gamma_i}')$.
    Finally, we refer to Lemma \ref{lemma:regul} (on a small ball 
    around any given point $v \in V_{\Lambda_i}$) to 
    conclude the regularity of $\Lambda_i$ under the assumption that
    $\Theta_m(\abs{Du} \restrict \Omega \setminus \Lambda_i; x)=0$
    at $\H^{m-1}$-\ae point $x \in \Lambda_i$.
\end{proof}

%%%
\section*{Acknowledgements}
%%%

This manuscript has been prepared over the course of several years,
on the side of ``paying work''. %of more applied projects attractive to funding bodies.
While the author was at the Institute for Mathematics and Scientific 
Computing at the University of Graz, this work was financially supported 
by the SFB research program F32 ``Mathematical Optimization and Applications in 
Biomedical Sciences'' of the Austrian Science Fund (FWF). 
While at the Department of Applied Mathematics and Theoretical Physics
at the University of Cambridge, this work was financially supported by the 
King Abdullah University of Science and Technology (KAUST) Award No.~KUK-I1-007-43 
as well as the EPSRC / Isaac Newton Trust Small Grant ``Non-smooth geometric 
reconstruction for high resolution MRI imaging of fluid transport
in bed reactors'', and the EPSRC first grant Nr.~EP/J009539/1
``Sparse \& Higher-order Image Restoration''.
At the Research Center on Mathematical Modeling (Modemat) at the
Escuela Politécnica Nacional de Quito, this work has been
supported by the Prometeo initiative of the Senescyt.

The author would like to express sincere gratitude to Simon Morgan, 
Antonin Chambolle, Kristian Bredies, and Carola-Bibiane Schönlieb
for fruitful discussions. Moreover, the author is grateful to Kostas
Papafitsoros for an additional proofreading of the manuscript.

\appendix

%%%
\section{Areas and lengths under Lipschitz transformations}
\label{appendix:lip}
%%%

We state and prove the following well-known expressions
for completeness.

\begin{lemma}
    \label{lemma:triangle-mod}
    Let $\Gamma \subset \R^m$ be a Lipschitz $(m-1)$-graph.
    Suppose $x_0 \in \Gamma$ and $\delta_0 > 0$ are such
    that $\B(\inv g_\Gamma(x_0), \delta_0) \subset V_\Gamma$.
    Let $\lambda \in L^1(\Gamma)$ and
    $0 \le h \in W_0^{1,\infty}(V_\Gamma)$, 
    with $\norm{h}_{L^1(V_\Gamma)} > 0$.
    Define
    \[
        \Gamma_\rho \defeq 
        \gamma_{\rho h}(\Gamma) = \{ g_\Gamma(v) + \rho h(v) z_\Gamma \mid v \in V_\Gamma \},
        \quad
        (\rho \in \R).
    \]
    Then
    \begin{align}
        \label{eq:gamma-rho-lm}
        \L^m(\Gamma_\rho^- \isect \Gamma^+)
        +
        \L^m(\Gamma_\rho^+ \isect \Gamma^-)
        & =
        \abs{\rho} \int_{V_\Gamma} h(v) \d v, \\
        \label{eq:gamma-rho-hm11}
        \int_\Gamma \lambda(x) \d \H^{m-1}(x)
        & =
        \int_{V_\Gamma}
            \lambda(g_\Gamma(v)) \sqrt{1+\norm{\DIFFSS f_\Gamma(v)}^2} 
            \d v,
        \quad
        \text{and}
        \\
        \label{eq:gamma-rho-hm12}
        \int_{\Gamma_\rho} \lambda(x) \d \H^{m-1}(x)
        & =
        \int_{V_\Gamma \isect}
            \lambda(g_\Gamma(v)) \sqrt{1+\norm{\DIFFSS f_\Gamma(v)+\rho \DIFFSS h(v)}^2}
            \d v.
    \end{align}
    %\begin{align}
    %    \label{eq:gamma-rho-lm}
    %    \L^m(\Gamma_\rho^- \isect \Gamma^+)
    %    & =
    %    \abs{\rho}\delta^{m-1}\omega_{m-1}/m, \\
    %    \label{eq:gamma-rho-hm11}
    %    \H^{m-1}(\Gamma \setminus \Gamma_\rho)
    %    & =
    %    \int_{V_\Gamma \isect \B(\inv g_\Gamma(x_0), \delta)}
    %        \sqrt{1+\norm{\DIFFSS f_\Gamma(v)}^2} 
    %        \d v,
    %    \quad
    %    \text{and}
    %    \\
    %    \label{eq:gamma-rho-hm12}
    %    \H^{m-1}(\Gamma_\rho \setminus \Gamma)
    %    & =
    %    \int_{V_\Gamma \isect \B(\inv g_\Gamma(x_0), \delta)}
    %        \sqrt{1+\norm{\DIFFSS f_\Gamma(v)}^2+\rho^2\delta^{-2} - 2\rho\inv \delta \DIFFSS f_\Gamma(v) \frac{v-x_0}{\norm{v-x_0}}}
    %        \d v.
    %\end{align}
\end{lemma}
\begin{proof}
    Without loss of generality, we may assume that $x_0=0$,
    and $\inv{g_\Gamma}(x_0)=0$. 
    We thus calculate for $\rho \ge 0$ that
    \[
        %\begin{split}
        \Gamma_\rho^- \isect \Gamma^+
        %&
        %=
        %    \{ x \in \Gamma^+ \mid \norm{x-P_{z_\Gamma}^\perp x} < \rho h(P_{z_\Gamma}^\perp x) \}
        %\\
        %&
        %=
        %    \{ x+\tau z_\Gamma \mid
        %        x \in \Gamma,\,
        %        0 < \tau < \rho h(P_{z_\Gamma}^\perp x) \}
        %\\
        %&
        =
            \{ g_\Gamma(v)+\tau z_\Gamma \mid
                v \in V_\Gamma,\,
                0 < \tau < \abs{\rho} h(v) \}.
        %\end{split}
    \]
    Analogously for $\rho < 0$ we obtain
    \[
        \Gamma_\rho^- \isect \Gamma^+
        =
            \{ g_\Gamma(v)-\tau z_\Gamma \mid
                v \in V_\Gamma,\,
                0 < \tau < \abs{\rho}h(v) \}.
    \]
    Application of Fubini's theorem therefore confirms
    \eqref{eq:gamma-rho-lm}.
    % through the reasoning
    %\[
    %    \L^m(\Gamma_\rho^- \isect \Gamma^+)
    %    =
    %    \int_{\Gamma_\rho^- \isect \Gamma^+} \d x
    %    =
    %    \int_{V_\Gamma} \int_0^{\abs{\rho} h(v)} \d\tau \d\H^{m-1}(v)
    %    =
    %    \int_{V_\Gamma}
    %        \abs{\rho} h(v) \d \H^{m-1}(v).
    %\]
    
    To prove \eqref{eq:gamma-rho-hm11}, we write
    \[
        \Gamma = \{ g_\Gamma(v) \mid v \in V_\Gamma\}.
    \]
    Thus the area formula \eqref{eq:areaformula} gives
    \[
        \int_\Gamma \lambda(x) \d \H^{m-1}(x)
        =\int_{V_\Gamma}
            \lambda(g_\Gamma(v)) \jacobianf{m-1}{g_\Gamma}{v} \d v.
    \]
    Writing
    \[
        g_\Gamma(v)=v+f_\Gamma(v)z_\Gamma
    \]
    we have
    \[
        \DIFFSS g_\Gamma(v)=H^* + \DIFFSS f_\Gamma(v) \otimes z_{\Gamma},
    \]
    where $H \in \L(V_\Gamma; \R^m)$
    is the embedding operator $H(v)=v$.
    We have $H^* \circ H =I$ and $H^* z_\Gamma=0$.
    Consequently
    \[
        \DIFFSS g_\Gamma(v) \circ [{\DIFFSS g_\Gamma(v)}]^*
        =
        I+\DIFFSS f_\Gamma(v) \otimes \DIFFSS f_\Gamma(v),
    \]
    The eigenvalue of ${\DIFFSS g_\Gamma(v)}^* \circ {\DIFFSS g_\Gamma(v)}$.
    corresponding to eigenvector $\DIFFSS f_\Gamma(v)$ is
    $1+\norm{\DIFFSS f_\Gamma(v)}^2$,
    while the other eigenvalues equal $1$.
    Thus
    \[
        \jacobianf{m-1}{g_\Gamma}{v}
        =\sqrt{\det(\DIFFSS g_\Gamma(v) \circ [\DIFFSS g_\Gamma(v)]^*)}
        =\sqrt{1+\norm{\DIFFSS f_\Gamma(v)}^2},
    \]
    giving the well-known formula for $\H^{m-1}(\Gamma \setminus \Gamma_\rho)$
    and confirming \eqref{eq:gamma-rho-hm11}.
    
    The same arguments prove also \eqref{eq:gamma-rho-hm12}.
    Indeed, we have
    \[
        \Gamma_\rho = \{ q_\rho(v) \mid v \in V_\Gamma\}.
    \]
    where
    \[
        q_\rho(v) \defeq v+(f_\Gamma+h)(v) z_\Gamma,
    \]
    so the area formula \eqref{eq:areaformula} again gives
    \[
        \int_{\Gamma_\rho} \lambda(x) \d \H^{m-1}(x)
        =\int_{V_\Gamma}
            \lambda(g_\Gamma(v))
            \jacobianf{m-1}{q_\rho}{v} \d v,
    \]
    where
    \[
        \jacobianf{m-1}{q_\rho}{v}=\sqrt{1+\norm{\DIFFSS f_\Gamma(v)+\rho \DIFFSS h_\Gamma(v)}^2}.
        \qedhere
    \]
\end{proof}

\section{Proofs for $p=1$ (§\ref{sec:main-l1})}
\label{appendix:l1}

In this appendix, we provide the proofs for some of the lemmas required for the
$p=1$ case in §\ref{sec:main-l1}. We begin with a modified version of the 
idelity estimate Lemma \ref{lemma:phi-approx} in Appendix \ref{app:l1fid}.
Observe that this version does not require the $L^\infty_\loc$ bounds.
Employing this result, we then prove Lemma \ref{lemma:r-curvature-sol} and
Lemma \ref{lemma:div-curv}, mainly concerning the regulariser, in Appendix
\ref{appendix:r-curvature-sol} and Appendix \ref{appendix:div-curv} respectively.

\subsection{Estimate of the fidelity}
\label{app:l1fid}

\begin{lemma}
    \label{lemma:phi-approx-l1}
    %Suppose Assumption \ref{ass:phi-increase} holds with $1 < p < 1^*$. 
    Suppose $\phi(t)$ is $1$-increasing,
    and that $u, f \in \BVspace(\Omega)$.
    Let $x_0 \in J_u \setminus (S_f \union Z_u)$.
    Given $\epsilon>0$, there exist $\sigma \in\{-1,+1\}$,
    $r_0>0$ and $\theta \in [0, 1]$,
    such that for any $r \in (0, r_0)$, $-1 < \rho < 1$, and
    $\baseh \in \basehspace$ with $-1 \le \baseh \le 1$, we have
    \begin{equation}
        \label{eq:phi-approx-l1}
        \begin{split}
        \int_{\Omega} \phi(\newu[\rho](x) & -f(x)) \d x 
        %&
        - \int_{\Omega} \phi(u(x) -f(x)) \d x
        \\
        &
        \le \epsilon\rho r^m
            - \sigma (1-\theta) \rho C_\phi \lambda_u(x_0) I_r,
                %\int_{V_\Gamma} \lambda_u(g_\Gamma(v)) h_r(v) \d v,
        \quad
        (0 < r < r_0, -1 < \rho < 1).
        \end{split}
    \end{equation}
\end{lemma}

\begin{proof}
    We assume that $\rho \ge 0$; the case $\rho<0$ can be handled by
    negating $\baseh$. Using Lemma \ref{lemma:phi-approx-1} to pass
    from the functions $(\newu[\rho], u, f)$ to $(\newuz[\rho], u_0, f_0)$,
    it suffices to prove the estimate corresponding to \eqref{eq:phi-approx-l1} 
    for the latter piecewise constant functions.

    Referring to Definition \ref{def:phi-increase} we deduce
    \[
        \begin{split}
        \phi(\newuz[\rho](x) & -f_0(x)) - \phi(u_0(x) -f_0(x))
        \\ &
        \le
        C_\phi(\abs{\newuz[\rho](x) -f_0(x)} - \abs{u_0(x) -f_0(x)}) \abs{\newuz[\rho](x) -f_0(x)}^0.
        \end{split}
    \]
    This is non-zero only in $W_r^+ \defeq U_r^+ \isect \gammax{\rho}{h}{r}(U_r^-)$,
    and in $W_r^- \defeq U_r^- \isect \gammax{\rho}{h}{r}(U_r^+)$.
    For $x \in W_r^+$, as in the proof of Lemma \ref{lemma:phi-const-approx},
    \[
        \begin{split}
        K^+ & \defeq \abs{(1-\theta)(b^- -a) + \theta (b^+  - a)} - \abs{b^+  - a}
        \\
        & =
        \abs{\newuz[\rho](x) -f_0(x)} - \abs{u_0(x) -f_0(x)}.
        \end{split}
    \]
    Here we again denote $b^\pm \defeq u^\pm(x_0)$ and $a \defeq \tilde f(x_0)$.
    Also
    \[
        \zeta^+ \defeq \abs{(1-\theta)(b^- -a) + \theta (b^+  - a)}^{0}
        =\abs{\newuz[\rho](x) -f_0(x)}^0.
    \]
    For $x \in W_r^-$, we have
    \[
        \begin{split}
        K^- & \defeq \abs{(1-\theta)(b^+ -a) + \theta (b^-  - a)} - \abs{b^-  - a}
        \\
        & =
        \abs{\newuz[\rho](x) -f_0(x)} - \abs{u_0(x) -f_0(x)}.
        \end{split}
    \]
    Moreover,
    \[
        \zeta^- \defeq \abs{(1-\theta)(b^+ -a) + \theta (b^-  - a)}^{0}
        = \abs{\newuz[\rho](x) -f_0(x)}^0.
    \]
    Writing $h_r=h_r^{(+)}-h_r^{(-)}$ for $h_r^{(+)}, h_r^{(-)}\ge 0$, thus
    \[
        \begin{split}
        L & \defeq
        \int_{U_r} \phi(\newuz[\rho](x) -f_0(x)) - \phi(u_0(x) -f_0(x)) \d x
        \\
        &
        \le
        C_\phi K^+ \zeta^+ \int_{V_\Gamma} \rho h_r^{(+)}(v) \d v
        +
        C_\phi K^- \zeta^- \int_{V_\Gamma} \rho h_r^{(-)}(v) \d v.
        \end{split}
    \]
    As in the proof of Lemma \ref{lemma:phi-const-approx}, we
    concentrate on the case $b^+  <  b^-$ with $a \ge (b^+ + b^-)/2$,
    setting $\sigma=1$.
    Other cases are handled analogously with $\sigma=-1$.
    Since $b^+ \ne b^-$, there are at most two choices of $\theta$ 
    for which $\zeta^+ = 0 $ or $\zeta^-=0$. 
    Simply by the triangle inequality
    \[
        K^- \le (1-\theta)\abs{b^+ - b^-}.
    \]
    Further, choosing $\theta < 1$ large enough while maintaining
    $\zeta^\pm = 1$, we can ascertain
    \[
        K^+ \le (1-\theta)(b^+ - b^-) = -(1-\theta)\abs{b^+ - b^-}.
    \]
    Thus
    \[
        L
        \le C_\phi  \abs{b^+-b^-} \left(\int_{V_\Gamma} \rho h_r^{(-)}(v) \d v
                            - \int_{V_\Gamma} \rho h_r^{(+)}(v) \d v \right)
        = -C_\phi \rho \lambda_u(x_0) I_r.
        \qedhere
    \]
    %Finally, since $x_0 \in Z_u$ is a Lebesgue point of $u^\pm$,
    %we have
    %\[
    %    -\lambda_u(x_0) \int_{V_\Gamma} h_r \d x
    %    \le \epsilon r^m -\int_{V_\Gamma} \lambda_u(g_\Gamma(v)) h_r(v) \d v.
    %    \qedhere
    %\]
\end{proof}

\begin{remark}
    Looking at the proof of Lemma \ref{lemma:phi-const-approx} with 
    a little bit more care in the approximation of $K^-$ above,
    we could in fact get the stronger double-sided estimate
    \eqref{eq:phi-approx} if $u$ ``jumps through $f$'',
    in other words, if
    \[
        \min\{u^+(x), u^-(x)\} < \tilde f(x_0) < \max\{u^+(x), u^-(x)\}.
    \]
\end{remark}

%%%
\subsection{Proof of Lemma \ref{lemma:r-curvature-sol}}
\label{appendix:r-curvature-sol}
%%%

%Saadaanko jotenkin diffiksillä myös Cantor-osa kuuseen?
%%Että saataisiin edes absoluuttisen jatkuvuuden säilytys, 
%ehkä jopa Lipschitz, jos nyt ei tarkkaa karaterisaatiota 
%rakenteesta kuten 1d:ssä.

The proof, as well as the proof of Lemma \ref{lemma:div-curv},
depends on the following technical result on the pointwise
$R$-curvature.

\begin{comment}
\begin{remark}
    \label{remark:pw-curvature}
    At $\H^{m-1}$-\ae $x_0 \in \Gamma$, the pointwise curvature 
    may alternatively be written
    \[
        \CURVP_u^{R,\Gamma,\baseh}(x_0) \defeq \lim_{r \downto 0} 
            \frac{\TDIFF_u^R(\gammax{1}{h}{r})}
                 {I_r \lambda_u(x_0)}.
    \]
    This is because $I_r=r^m I_1$ and
    $\lim_{r\downto 0} r^{m-1} \omega_{m-1} \abs{Du}(\B(x_0, r))
    =\Theta_{m-1}(\abs{Du}; x_0)=\lambda_u(x_0)>0$.
    The definition above is slightly more general,
    and can in some situation make sense even when $x_0 \not \in J_u$.
    \TODO{Examples?}
\end{remark}
\end{comment}

\begin{lemma}
    \label{lemma:r-curvature}
    Let $\alpha > 0$ and $\Gamma \subset \Omega$ be 
    a Lipschitz $(m-1)$-graph. Then at 
    %$\H^{m-1}$-\ae $x_0 \in \Gamma$, 
    any $x_0 \in \Gamma$, either of the following holds.
    \begin{enumroman}
        \item
        \label{item:r-curvature-first}
        $\CURVP_u^{R,\Gamma,\baseh}(x_0)=\inv\alpha$. %\lambda_u(x_0)$.
        \item
        \label{item:r-curvature-second}
        For some $\kappa \in (0, 1)$ there exist $r_j \downto 0$
        and for each $j$, some $\rho_{j, i} \downto 0$ such that
        for every $i, j \in \Z^+$ either
        \begin{equation}
            \label{eq:r-curvature-second1}
            R(\pf{\gammax{\rho_{j,i}}{h}{r_j}} u)  - R(u)
            - \kappa \inv\alpha \rho_{j,i} 
            %\int_{V_\Gamma} \lambda_u(g_\Gamma(v)) h_{r_j}(v) \d v  
            %\lambda_u(x_0) 
            I_r
            \le 0,
        \end{equation}
        or
        \begin{equation}
            \label{eq:r-curvature-second2}
            R(\pf{\gammax{\rho_{j,i}}{h}{r_j}} u)  - R(u)
            - \inv \kappa \inv\alpha \rho_{j,i} 
            %\int_{V_\Gamma} \lambda_u(g_\Gamma(v)) h_{r_j}(v) \d v 
            %\lambda_u(x_0) 
            I_r
            \ge 0.
        \end{equation}
    \end{enumroman}
\end{lemma}

%\begin{remark}
%    There is no reason for the specific choice of $\lambda_u$ above;
%    we could use any other $\lambda \in L^1(\Gamma)$ and
%    Lebesgue point $x_0$ of $\lambda$.
%\end{remark}

\begin{proof}
    %We pick $x_0$ a Lebesgue point of $\lambda_u$,
    %satisfying $\Theta_{m-1}(\abs{Du}; x_0)=\lambda_u(x_0)$.
    %This holds $\H^{m-1}$-\ae by the Besicovitch derivation theorem.
    We observe that $\gammax{\rho}{h}{r}=\rho \gammax{1}{h}{r}+ (1-\rho)\iota$,
    so that
    \[
        \TDIFF_u^R(\gammax{1}{h}{r})
        = \limsup_{\rho \downto 0} \frac{R(\pf{\gammax{\rho}{h}{r}} u)-R(u)}{\rho}.
    \]
    Suppose case \ref{item:r-curvature-second} does not hold at $x_0 \in \Gamma$.
    Fix $\kappa \in (0, 1)$. Then there exists $r_0>0$ such
    that for every $0 < r < r_0$, there exists $\rho^r > 0$ such that
    for every $0 < \rho < \rho^r$ we have bounds
    \begin{equation}
        \label{eq:r-curvature-bound-1}
        \inv\kappa \inv\alpha \rho 
        %\int_{V_\Gamma} \lambda_u(g_\Gamma(v)) h_{r}(v) \d y
        %\lambda_u(x_0) 
        I_r
        \le
        R(\pf{\gammax{\rho}{h}{r}} u)  - R(u)
        \le
        \kappa \inv\alpha \rho 
        %\int_{V_\Gamma} \lambda_u(g_\Gamma(v) h_{r}(v) \d y.
        %\lambda_u(x_0) 
        I_r.
    \end{equation}
    If we define the upper and lower transformation differentials
    \[
        \TDIFF_{u}^{R,*}(\gamma) \defeq \limsup_{\rho \downto 0}
            \frac{R(\pf\gamma u)  - R(u)}{\rho}
        \quad\text{and}\quad
        \TDIFF_{u,*}^{R}(\gamma) \defeq \liminf_{\rho \downto 0}
            \frac{R(\pf\gamma u)  - R(u)}{\rho},
    \]
    then dividing \eqref{eq:r-curvature-bound-1} by $\rho$ and 
    letting $\rho \downto 0$, it follows  that
    \[
        \begin{split}
            \inv\kappa \inv\alpha 
            %\int_{V_\Gamma} \lambda_u(g_\Gamma(v)) h_{r}(v) \d v
            %\lambda_u(x_0) 
            I_r
            &
            \le
            \TDIFF_{u,*}^{R}(\gammax{\rho}{h}{r})
            \le
            \TDIFF_{u}^{R,*}(\gammax{\rho}{h}{r})
            %\\
            %&
            \le
            \kappa \inv\alpha 
            %\int_{V_\Gamma} \lambda_u(g_\Gamma(v)) h_{r}(v) \d v.
            %\lambda_u(x_0) 
            I_r.
        \end{split}
    \]
    Dividing by $I_{r}$ and letting $r \downto 0$, we find
    %, since $x_0$ was a Lebesgue point of $\lambda_u$, we find that
    \[
        %\begin{split}
            \inv\kappa \inv\alpha %\lambda_u(x_0)
            %&
            \le
            \liminf_{r \downto 0} \inv I_r
            \TDIFF_{u,*}^{R}(\gammax{\rho}{h}{r})
            %\\
            %&
            \le
            \limsup_{r \downto 0} \inv I_r
            \TDIFF_{u}^{R,*}(\gammax{\rho}{h}{r})
            %\\
            %&
            \le
            \kappa \inv\alpha %\lambda_u(x_0).
        %\end{split}
    \]
    Since $\kappa \in (0, 1)$ was arbitrary, %recalling
    %Remark \ref{remark:pw-curvature} 
    we deduce
    \[
        %\begin{split}
            \inv\alpha %\lambda_u(x_0)
            %&
            %=
            %\liminf_{r \downto 0} \inv I_r
            %\TDIFF_{u,*}^{R}(\gammax{\rho}{h}{r})
            %\\
            %&
            %=
            %\limsup_{r \downto 0} \inv I_r
            %\TDIFF_{u}^{R,*}(\gammax{\rho}{h}{r})
            = \CURVP_u^{R,\Gamma,\baseh}(x_0).
        %\end{split}
    \]
    In particular, $\CURVP_u^{R,\Gamma,\baseh}(x_0)$ exists
    and case \ref{item:r-curvature-first} holds.
\end{proof}

With this, we are ready to prove the curvature expression claimed in Lemma \ref{lemma:r-curvature-sol}.

\begin{proof}[Proof of Lemma \ref{lemma:r-curvature-sol}]
    By Lemma \ref{lemma:phi-approx-l1}, given $\epsilon>0$, 
    we have the existence of $r_0>0$ and $\sigma \in \{-1, +1\}$ such that
    \begin{equation}
        \label{eq:r-curvature-second1-use1}
        %\begin{split}
        \int_{\Omega} \phi(\newu[\rho](x)  -f(x)) \d x 
        %&
        - \int_{\Omega} \phi(u(x) -f(x)) \d x
        %\\
        %&
        \le \epsilon\rho r^m
            - \sigma (1-\theta) \rho C_\phi 
            %\int_{V_\Gamma} \lambda_u(g_\Gamma(v)) h_r(v) \d v,
            \lambda_u(x_0) I_r,
        %\quad
        %(0 < r < r_0, -1 < \rho < 1).
        %\end{split}
    \end{equation}
    whenever $0 < r < r_0$ and $-1 < \rho < 1$.
    We concentrate on the $\sigma=+1$; the case $\sigma=-1$ is 
    handled analogously (or simply by exchanging sides $\Gamma^+$ and $\Gamma^-$).
    We suppose first that Lemma \ref{lemma:r-curvature}\ref{item:r-curvature-second}
    holds for some $x_0 \in (\Gamma \isect J_u) \setminus (J_f \union Z_u)$
    with $\alpha=1/(\lambda_u(x_0)C_\phi)$. We further concentrate first
    on the case that \eqref{eq:r-curvature-second1} holds.
    Then for some $\kappa \in (0, 1)$, a sequence
    $r_j \downto 0$, and for each fixed $j$, a sequence
    $\rho_{j,i} \downto 0$, we have
    \begin{equation}
        \notag
        R(\newur[\rho_{j,i},r_j])  - R(u)
        \le
        (1-\theta)\left(
            R(\pf{\gammax{\rho_{j,i}}{h}{r_j}}U)  - R(u)
        \right)
        \le
        (1-\theta) \kappa C_\phi \rho_{j,i} 
        %\int_{V_\Gamma} \lambda_u(g_\Gamma(v)) h_{r_j}(v) \d v.
        \lambda_u(x_0) I_{r_j}.
    \end{equation}
    Taking $r=r_j$ and $\rho=\rho_{j,i}$ and summing with
    \eqref{eq:r-curvature-second1-use1}, 
    we find that
    \[
        \begin{split}
        \biggl(
            \int_{\Omega} \phi(\newur[\rho_{j,i},r_j](x)  -f(x)) \d x 
            &
            + R(\newur[\rho_{j,i},r_j])
        \biggr) - \biggl(
            \int_{\Omega} \phi(u(x) -f(x)) \d x 
            + R(u)
        \biggr)
        \\
        &
        \le \epsilon\rho_{j,i} r_j^m
            - (1-\kappa)(1-\theta) \rho_{j,i} C_\phi 
            %\int_{V_\Gamma} \lambda_u(g_\Gamma(v)) h_{r_j}(v) \d v,
            \lambda_u(x_0) I_{r_j},
        \\
        &
        \le \epsilon \rho_{j,i} r_j^m - C'' \rho_{j,i} r_j^{m}.
        \end{split}
    \]
    Here the constant $C''>0$. Choosing $j$ large enough that
    we can take $\epsilon=C''/2$, we can make this negative.
    This yields a contradiction to $u$ solving \eqref{eq:prob}.
    
    Suppose then that \eqref{eq:r-curvature-second2} holds.
    Then for some $\kappa > 0$ we have
    \begin{equation}
        \label{eq:r-curvature-second2-use1}
        R(\newur[\rho_{j,i},r_j])  - R(u)
        \ge
        (1-\theta) \inv\kappa C_\phi \rho_{j,i} 
        %\int_{V_\Gamma} \lambda_u(g_\Gamma(v)) h_{r_j}(v) \d v.
        \lambda_u(x_0) I_{r_j},
    \end{equation}
    We now use the double-Lipschitz comparability we already used in
    the proof of Lemma \ref{lemma:r-approx}. Namely,
    choosing $r>0$ small enough, we have for some constant $C'>0$ that
    \[
        R(\newur[\rho,r])  + R(\newur[-\rho,r])  - R(u)
        \le C' \rho^2 r^{m-1}.
    \]
    We may w.log assume that all $\{r_i\}_{i \in \Z^+}$ are small enough
    for this. Using \eqref{eq:r-curvature-second2-use1} it follows
    \begin{equation}
        \label{eq:r-curvature-second2-transfer}
        R(\newur[-\rho_{j,i},r_j])  - R(u)
        \le
        C' \rho_{j,i}^2 r_j^{m-1}
        -
        (1-\theta) \inv\kappa C_\phi \rho_{j,i} 
        % \int_{V_\Gamma} \lambda_u(g_\Gamma(v)) h_{r_j}(v) \d v.
        \lambda_u(x_0) I_{r_j}.
    \end{equation}
    Summing \eqref{eq:r-curvature-second1-use1} 
    with \eqref{eq:r-curvature-second2-transfer} we find that
    \[
        \begin{split}
        \biggl(
            \int_{\Omega} \phi(\newur[-\rho_{j,i},r_j](x) & -f(x)) \d x 
            %&
            + R(\newur[-\rho_{j,i},r_j])
        \biggr) - \biggl(
            \int_{\Omega} \phi(u(x) -f(x)) \d x
            + R(u)
        \biggr)
        \\
        &
        \le C' \rho_{j,i}^2 r_j^{m-1}
            + \epsilon\rho_{j,i} r_j^m
            + (1-\inv\kappa)(1-\theta) \rho_{j,i} C_\phi 
            %\int_{V_\Gamma} \lambda_u(g_\Gamma(v)) h_{r_j}(v) \d v
            \lambda_u(x_0) I_{r_j}
        \\
        &
        \le C' \rho_{j,i}^2 r_j^{m-1}
            + \epsilon \rho_{j,i} r_j^m
             - C'' \rho_{j,i} r_j^{m}.
        \end{split}
    \]
    Again the constant $C''>0$.
    Choosing $j \in \Z^+$ large enough that we can take $0 < \epsilon < C''/2$,
    and then choosing $i$ large enough that $\rho_{j,i}^2$ becomes very small,
    we can make the right hand side negative. Thus we again have a contradiction
    to $u$ solving \eqref{eq:prob}.
    It follows that Lemma \ref{lemma:r-curvature}\ref{item:r-curvature-first}
    has to hold for all $x_0 \in (\Gamma \isect J_u) \setminus (J_f \union Z_u)$.
    With our choice $\alpha=1/(\lambda_u(x_0)C_\phi)$,
    this says exactly that
    $\CURVP_u^{R,\Gamma,\baseh}(x_0)=\lambda_u(x_0)C_\phi$
    for $\H^{m-1}$-\ae $x_0 \in (J_u \setminus J_f) \isect \Gamma$.
\end{proof}

%%%
\subsection{Proof of Lemma \ref{lemma:div-curv}}
\label{appendix:div-curv}
%%%

We use the next lemma to translate the pointwise $R$-curvature for general $R$ into pointwise $\TV_\Gamma$-curvature for $\TV_\Gamma(u) \defeq \abs{Du}(\Gamma)$ under the approximate piecewise constancy assumption.

\begin{lemma}
    \label{lemma:curvp-d-pf}
    Let $\Gamma \subset \Omega$ be a Lipschitz $(m-1)$-graph
    and $x_0 \in \Gamma$.
    Suppose $R$ is separably double-Lipschitz comparable, and
    $\Theta_m(\abs{Du} \restrict \Omega \setminus \Gamma; x_0)=0$.
    Then
    \begin{equation}
        \label{eq:curvp-d-pf}
        \CURVP_u^{R,\Gamma,\baseh}(x_0) = \lim_{r \downto 0} \inv I_r \tilde \TDIFF_u^R(\gammax{1}{h}{r}),
    \end{equation}
    where
    \begin{equation}
        \notag
        %\CURVP_u^{R,\Gamma,\baseh}(x_0) 
        \tilde \TDIFF_u^R(\gammax{1}{h}{r})
        \defeq \lim_{\rho \downto 0} 
            \RCs
            \frac{
            \abs{D \pf{\gammax{\rho}{h}{r}} u}(\gammax{\rho}{h}{r}(\Gamma)) - \abs{D u}(\Gamma)
            }{\rho}.
    \end{equation}
\end{lemma}

\begin{proof}
    We recall that
    \[
        \TDIFF_u^R(\gammax{1}{h}{r})
        = \lim_{\rho \downto 0} \frac{R(\pf{\gammax{\rho}{h}{r}} u)-R(u)}{\rho}.
    \]
    By Definition \ref{def:lipschitz-trans} we have
    \[
        R(\pf{\gammax{\rho}{h}{r}} u) - R(u) \le 
        \RCa \bitransfull{\gammax{\rho}{h}{r}}{\iota} 
            \abs{D u}(\closure U \setminus \Gamma)
            +
        \RCs \bigl(\abs{D \pf{\gammax{\rho}{h}{r}} u}(\gammax{\rho}{h}{r}(\Gamma)) - \abs{D u}(\Gamma)\bigr).
    \]
    Using Lemma \ref{lemma:gamma-rho-h-r} to estimate
    $\bitransfull{\gammax{\rho}{h}{r}}{\iota}$, and
    minding the assumption $\Theta_m(\abs{Du} \restrict \Omega \setminus \Gamma; x_0)=0$,
    we deduce for arbitrary $\epsilon>0$, for small enough $r>0$ that
    \begin{equation}
        \notag
        \TDIFF_u^R(\gammax{1}{h}{r})
        \le 
        \epsilon r^m
        +
        \RCs \limsup_{\rho \downto 0}
        \frac{\abs{D \pf{\gammax{\rho}{h}{r}} u}(\gammax{\rho}{h}{r}(\Gamma)) - \abs{D u}(\Gamma)}{\rho}.
    \end{equation}
    Dividing by $I_r$, letting $r \downto 0$, 
    and afterwards letting $\epsilon \downto 0$, yields the upper bound
    \begin{equation}
        \notag
        \CURVP_u^{R,\Gamma,\baseh}(x_0) \le \lim_{r \downto 0} \inv I_r \tilde \TDIFF_u^R(\gammax{1}{h}{r}).
    \end{equation}
    
    In order to derive the corresponding lower bound,
    we set $\gamma \defeq \gammax{\rho}{h}{r}$ and
    $v \defeq \pf\gamma u$. Then
    \[
        \begin{split}
        R(\pf\gamma u) - R(u) 
        %&
        =
        R(v) - R(\pf{\inv\gamma} v)
        %\\
        &
        \ge
        -\RCa \bitransfull{\inv\gamma}{\iota} 
            \abs{D u}(\closure U \setminus \gamma(\Gamma))
        \\ & \phantom{\ge}
            -
        \RCs\frac{\abs{D \pf{\gamma} v}(\Gamma) - \abs{D v}(\gamma(\Gamma))}{\rho}.
        \end{split}
    \]
    Using Lemma \ref{lemma:gamma-rho-h-r} to estimate
    $\bitransfull{\inv\gamma}{\iota}$, and proceeding
    as above, we deduce \eqref{eq:curvp-d-pf}.
\end{proof}

We now proceed with the final missing proof.

\begin{proof}[Proof of Lemma \ref{lemma:div-curv}]
    We define the shorthand notation
    $\lambda(v) \defeq \lambda_u(g_\Gamma(v))$, and
    observe from Lemma \ref{lemma:triangle-mod} that
    \[
        \abs{D \pf{\gammax{\rho}{h}{r}} u}(\gammax{\rho}{h}{r}(\Gamma))
        =
        \int_{V_\Gamma} \lambda(v) \sqrt{1+\norm{\DIFFSS (f_\Gamma+\rho h_r)(v)}^2} \d v,
    \]
    and
    \[
        \abs{D u}(\Gamma)
        =
        \int_{V_\Gamma} \lambda(v) \sqrt{1+\norm{\DIFFSS f_\Gamma(v)}^2} \d v.
    \]
    We further calculate using convexity of the norm, and concavity of
    the square root that
    \begin{multline}
        \notag
        \rho \curvf_{\lambda,f}(h_r) \le
        \int_V \lambda(v) \sqrt{1+\norm{\DIFFSS (f_\Gamma+\rho h_r)(v)}^2} \d v
        - \int_V \lambda(v) \sqrt{1+\norm{\DIFFSS f_\Gamma(v)}^2} \d v
        \\
        \le
        \rho \curvf_{\lambda,f}(h_r) 
        %\rho \int_V \lambda(v) \adaptiprod{\frac{\DIFFSS f(v)}{\sqrt{1+\norm{\DIFFSS f(v)}^2}}}{\DIFFSS h_r^{v_0}(v)} \d v
        +
        \rho^2 \int_V \lambda(v) \frac{\norm{\DIFFSS h_r(v)}^2}{\sqrt{1+\norm{\DIFFSS f_\Gamma(v)}^2}} \d v,
    \end{multline}
    where
    \[
            \curvf_{\lambda,f}(h) \defeq \int_V \lambda(v) \adaptiprod{\frac{\DIFFSS f_\Gamma(v)}{\sqrt{1+\norm{\DIFFSS f_\Gamma(v)}^2}}}{\DIFFSS h(v)} \d v.
    \]
    Thus
    \[
        \tilde \TDIFF_u^{R}(\gammax{1}{h}{r}) = \RCs \curvf_{\lambda,f}(h_r).
    \]
    Minding Lemma \ref{lemma:curvp-d-pf}, we deduce
    \[
        \CURVP_u^{R,\Gamma,\baseh}(x_0) =\lim_{r \downto 0} \RCs \curvf_{\lambda,f}(h_r/I_r).
    \]
    Since $\DIFFSS f_\Gamma$ is bounded, the claim is immediate
    if $\inv g_\Gamma(x_0)$ is a Lebesgue point of $\lambda$ 
    and $\DIFFSS f_\Gamma$, that is $\H^{m-1}$-\ae
\end{proof}

%%%
%\bibliography{abbrevs,bib,bib-own}

\begin{thebibliography}{10}

\bibitem{allard2008total}
{\sc W.K. Allard}, {\em Total variation regularization for image denoising,
  {I}. {G}eometric theory}, SIAM Journal on Mathematical Analysis, 39 (2008),
  pp.~1150--1190.

\bibitem{alter2005characterization}
{\sc F.~Alter, V.~Caselles, and A.~Chambolle}, {\em A characterization of
  convex calibrable sets in {$\R^n$}}, Mathematische Annalen, 332 (2005),
  pp.~329--366.

\bibitem{ambrosio2000fbv}
{\sc L.~Ambrosio, N.~Fusco, and D.~Pallara}, {\em Functions of Bounded
  Variation and Free Discontinuity Problems}, Oxford University Press, 2000.

\bibitem{benning2012ground}
{\sc M.~Benning and M.~Burger}, {\em Ground states and singular vectors of
  convex variational regularization methods}, Methods and Applications of
  Analysis, 20 (2013), pp.~295--334.

\bibitem{bertozzi2004low}
{\sc A.~L. Bertozzi and J.~B. Greer}, {\em Low-curvature image simplifiers:
  Global regularity of smooth solutions and {L}aplacian limiting schemes},
  Communications on Pure and Applied Mathematics, 57 (2004), pp.~764--790.

\bibitem{bredies2009tgv}
{\sc K.~Bredies, K.~Kunisch, and T.~Pock}, {\em Total generalized variation},
  SIAM Journal on Imaging Sciences, 3 (2011), pp.~492--526.

\bibitem{l1tgv}
{\sc K.~Bredies, K.~Kunisch, and T.~Valkonen}, {\em Properties of
  {$L^1$-$\mbox{TGV}^2$}: The one-dimensional case}, Journal of Mathematical
  Analysis and Applications, 398 (2013), pp.~438--454.

\bibitem{burger2012regularized}
{\sc M.~Burger, M.~Franek, and {C.-B.} Schönlieb}, {\em Regularized regression
  and density estimation based on optimal transport}, Applied Mathematics
  Research eXpress, 2012 (2012), pp.~209--253.

\bibitem{caselles2008discontinuity}
{\sc V.~Caselles, A.~Chambolle, and M.~Novaga}, {\em The discontinuity set of
  solutions of the {TV} denoising problem and some extensions}, Multiscale
  Modeling and Simulation, 6 (2008), pp.~879--894.

\bibitem{chambolle2004meanalgorithm}
{\sc A.~Chambolle}, {\em An algorithm for mean curvature motion}, Interfaces
  and Free Boundaries, 6 (2004), p.~195.

\bibitem{chambolle97image}
{\sc A.~Chambolle and P.-L. Lions}, {\em Image recovery via total variation
  minimization and related problems}, Numerische Mathematik, 76 (1997),
  pp.~167--188.

\bibitem{chan2000high}
{\sc T.~Chan, A.~Marquina, and P.~Mulet}, {\em High-order total variation-based
  image restoration}, SIAM Journal on Scientific Computation, 22 (2000),
  pp.~503--516.

\bibitem{chan2005aspects}
{\sc T.~F. Chan and S.~Esedoglu}, {\em Aspects of total variation regularized
  {$L^1$} function approximation}, SIAM Journal on Applied Mathematics, 65
  (2005), pp.~1817--1837.

\bibitem{chan2002euler}
{\sc T.~F. Chan, S.~H. Kang, and J.~Shen}, {\em Euler's elastica and
  curvature-based inpainting}, SIAM Journal on Applied Mathematics,  (2002),
  pp.~564--592.

\bibitem{dacorogna1990partial}
{\sc B.~Dacorogna and J.~Moser}, {\em On a partial differential equation
  involving the jacobian determinant}, Annales de l'Institut Henri
  Poincar{\'e}. Analyse non lin{\'e}aire, 7 (1990), pp.~1--26.

\bibitem{DFLM2009}
{\sc G.~Dal~Maso, I.~Fonseca, G.~Leoni, and M.~Morini}, {\em A higher order
  model for image restoration: the one-dimensional case}, SIAM Journal on
  Mathematical Analysis, 40 (2009), pp.~2351--2391.

\bibitem{didas2009properties}
{\sc S.~Didas, J.~Weickert, and B.~Burgeth}, {\em Properties of higher order
  nonlinear diffusion filtering}, Journal of Mathematical Imaging and Vision,
  35 (2009), pp.~208--226.

\bibitem{duval2009tvl1}
{\sc V.~Duval, J.~F. Aujol, and Y.~Gousseau}, {\em The {TVL1} model: A
  geometric point of view}, Multiscale Modeling and Simulation, 8 (2009),
  pp.~154--189.

\bibitem{fonseca2007mmc}
{\sc I.~Fonseca and G.~Leoni}, {\em Modern methods in the calculus of
  variations: {$L^p$} spaces}, Springer Verlag, 2007.

\bibitem{gilbarg2001elliptic}
{\sc D.~Gilbarg and N.S. Trudinger}, {\em Elliptic partial differential
  equations of second order}, Classics in mathematics, Springer, 2001.

\bibitem{guidotti2014anisotropic}
{\sc P.~Guidotti}, {\em Anisotropic diffusions of image processing from
  {Perona-Malik} on}, Advanced Studies in Pure Mathematics,  (2014).
\newblock To appear.

\bibitem{hintermuller2006infeasible}
{\sc M.~Hinterm{\"u}ller and G.~Stadler}, {\em An infeasible primal-dual
  algorithm for total bounded variation--based inf-convolution-type image
  restoration}, SIAM Journal on Scientific Computation, 28 (2006), pp.~1--23.

\bibitem{HiWu13_siims}
{\sc M.~Hinterm\"uller and T.~Wu}, {\em Nonconvex {TV$^q$}-models in image
  restoration: Analysis and a trust-region regularization--based superlinearly
  convergent solver}, SIAM Journal on Imaging Sciences, 6 (2013),
  pp.~1385--1415.

\bibitem{HiWu14_coap}
\leavevmode\vrule height 2pt depth -1.6pt width 23pt, {\em A superlinearly
  convergent {$R$}-regularized {Newton} scheme for variational models with
  concave sparsity-promoting priors}, Computational Optimization and
  Applications, 57 (2014), pp.~1--25.

\bibitem{tuomov-tvq}
{\sc M.~Hintermüller, T.~Valkonen, and T.~Wu}, {\em Limiting aspects of
  non-convex {$\mbox{TV}^\varphi$} models}.
\newblock Submitted, 2014.

\bibitem{huang1999statistics}
{\sc J.~Huang and D.~Mumford}, {\em Statistics of natural images and models},
  in IEEE Conference on Computer Vision and Pattern Recognition (CVPR), vol.~1,
  1999.

\bibitem{tuomov-krtv}
{\sc J.~Lellmann, D.~Lorenz, C.-B. Sch{\"o}nlieb, and T.~Valkonen}, {\em
  Imaging with {K}antorovich-{R}ubinstein discrepancy}, SIAM Journal on Imaging
  Sciences, 7 (2014), pp.~2833--2859.

\bibitem{lysaker2003noise}
{\sc M.~Lysaker, A.~Lundervold, and {X.-C.} Tai}, {\em Noise removal using
  fourth-order partial differential equation with applications to medical
  magnetic resonance images in space and time}, IEEE Transactions on Image
  Processing, 12 (2003), pp.~1579--1590.

\bibitem{meyer2002oscillating}
{\sc Y.~Meyer}, {\em Oscillating patterns in image processing and nonlinear
  evolution equations}, American Mathematical Society, 2001.

\bibitem{Nik2002}
{\sc M.~Nikolova}, {\em Minimizers of cost functions involving nonsmooth
  data-fideltiy terms. {A}pplication of processing of outliers}, SIAM Journal
  on Numerical Analysis, 40 (2002), pp.~965--994.

\bibitem{ochsiterated}
{\sc P.~Ochs, A.~Dosovitskiy, T.~Brox, and T.~Pock}, {\em An iterated l1
  algorithm for non-smooth non-convex optimization in computer vision}, in IEEE
  Conference on Computer Vision and Pattern Recognition (CVPR), 2013.

\bibitem{papafitsoros2013study}
{\sc K.~Papafitsoros and K.~Bredies}, {\em A study of the one dimensional total
  generalised variation regularisation problem}.
\newblock Preprint, 2013.

\bibitem{papafitsoros2012combined}
{\sc K.~Papafitsoros and C.-B. Sch{\"o}nlieb}, {\em A combined first and second
  order variational approach for image reconstruction}, Journal of Mathematical
  Imaging and Vision, 48 (2014), pp.~308--338.

\bibitem{papafitsoros2015asymptotic}
{\sc K.~Papafitsoros and T.~Valkonen}, {\em Asymptotic behaviour of total
  generalised variation}, in Fifth International Conference on Scale Space and
  Variational Methods in Computer Vision (SSVM), 2015.
\newblock Accepted.

\bibitem{perona1990scalespace}
{\sc P.~Perona and J.~Malik}, {\em Scale-space and edge detection using
  anisotropic diffusion}, IEEE Transactions on Pattern Analysis and Machine
  Intelligence, 12 (1990), pp.~629--639.

\bibitem{rindler2013strictly}
{\sc F.~Rindler and G.~Shaw}, {\em Strictly continuous extensions of
  functionals with linear growth to the space {BV}}.
\newblock preprint, 2013.

\bibitem{ring2000structural}
{\sc W.~Ring}, {\em Structural properties of solutions to total variation
  regularization problems}, ESAIM: Mathematical Modelling and Numerical
  Analysis, 34 (2000), pp.~799--810.

\bibitem{Rud1992}
{\sc L.~Rudin, S.~Osher, and E.~Fatemi}, {\em Nonlinear total variation based
  noise removal algorithms}, Physica D, 60 (1992), pp.~259--268.

\bibitem{shen2003euler}
{\sc J.~Shen, S.~Kang, and T.~Chan}, {\em Euler's elastica and curvature-based
  inpainting}, SIAM Journal on Applied Mathematics, 63 (2003), pp.~564--592.

\bibitem{tuomov-bd}
{\sc T.~Valkonen}, {\em Transport equation and image interpolation with {SBD}
  velocity fields}, Journal de math{\'e}matiques pures et appliqu{\'e}es, 95
  (2011), pp.~459--494.

\bibitem{tuomov-jumpset2}
\leavevmode\vrule height 2pt depth -1.6pt width 23pt, {\em The jump set under
  geometric regularisation. {Part 2}: Higher-order approaches}.
\newblock Submitted, July 2014.

\bibitem{tuomov-nlpdhgm}
\leavevmode\vrule height 2pt depth -1.6pt width 23pt, {\em A primal-dual hybrid
  gradient method for non-linear operators with applications to {MRI}}, Inverse
  Problems, 30 (2014), p.~055012.

\bibitem{vese2003modelingtextures}
{\sc L.~A. Vese and S.~J. Osher}, {\em Modeling textures with total variation
  minimization and oscillating patterns in image processing}, Journal of
  Scientific Computing, 19 (2003), pp.~553--572.

\bibitem{weickert1998anisotropic}
{\sc J.~Weickert}, {\em Anisotropic Diffusion in Image Processing}, ECMI
  series, B.G.~Teubner Stuttgart, 1998.

\bibitem{Yin2007tvl1}
{\sc W.~Yin, D.~Goldfarb, and S.~Osher}, {\em The total variation regularized
  {$L^1$} model for multiscale decomposition}, Multiscale Modeling and
  Simulation, 6 (2007), pp.~190--211.

\end{thebibliography}
 \providecommand{\homesiteprefix}{http://iki.fi/tuomov/mathematics}
  \providecommand{\eprint}[1]{\href{http://arxiv.org/abs/#1}{arXiv:#1}}

%%%

\end{document}